\documentclass[11pt]{amsart}
\usepackage{amsmath, amssymb, latexsym}
\usepackage{mathrsfs}
\usepackage{enumerate}
\usepackage{color, url}
\usepackage{bbm}
\usepackage[colorlinks=true,linkcolor=blue,citecolor=purple]{hyperref}
\usepackage{tikz}
\usepackage{geometry}
\usetikzlibrary{positioning}
\numberwithin{equation}{section}
\newtheorem{theorem}{Theorem}[section]
\newtheorem{corollary}[theorem]{Corollary}
\newtheorem{lemma}[theorem]{Lemma}
\newtheorem{proposition}[theorem]{Proposition}

\newtheorem{example}[theorem]{Example}
\theoremstyle{definition}
\newtheorem{definition}[theorem]{Definition}
\newtheorem{remark}[theorem]{Remark}
\newtheorem{innercustomthm}{{\bf Theorem}}
\newenvironment{customthm}[1]
{\renewcommand\theinnercustomthm{#1}\innercustomthm}
{\endinnercustomthm}

\newtheorem{innercustomprop}{{\bf Proposition}}
\newenvironment{customprop}[1]
{\renewcommand\theinnercustomprop{#1}\innercustomprop}
{\endinnercustomthm}

\newcommand{\Hom}{\operatorname{Hom}}

\newcommand{\A}{\mathbf{A}}
\newcommand{\B}{\mathbf{B}}
\newcommand{\C}{\mathbb{C}}

\newcommand{\Q}{\mathbf{Q}}

\newcommand{\Z}{\mathbf{Z}}
\newcommand{\h}{\mathfrak{h}}
\newcommand{\g}{\mathfrak{g}}
\allowdisplaybreaks
\title[Crystal bases and canonical bases]
{Crystal bases and canonical bases  for \\ quantum Borcherds-Bozec algebras}

\author[Zhaobing Fan]{Zhaobing Fan}
\address{Harbin Engineering University,
Harbin, China}
\email{fanzhaobing@hrbeu.edu.cn}
\thanks{ }

\author[Shaolong Han]{Shaolong Han${}^*$}
\address{Beijing International Center for Mathematical Research, Peking University, Beijing 100871, China}
\email{algebra@hrbeu.edu.cn}
\thanks{${}^*$The corresponding author.}

\author[Seok-Jin Kang]{Seok-Jin Kang}
\address{Korea Research Institute of Arts and Mathematics,
Asan-si, Chungcheongnam-do, 31551, Korea}
\email{soccerkang@hotmail.com}
\thanks{}

\author[Young Rock Kim]{Young Rock Kim}
\address{Graduate School of Education, Hankuk University of Foreign Studies, Seoul, 02450,  Korea}
\email{rocky777@hufs.ac.kr} %


\begin{document}

\begin{abstract}
Let $U_{q}^{-}(\g)$ be the negative half of a quantum Borcherds-Bozec algebra $U_{q}(\g)$
and $V(\lambda)$ be the irreducible highest weight module with $\lambda \in P^{+}$.
In this paper, we investigate the structures, properties and their close connections 
between crystal bases and canonical bases of $U_{q}^{-}(\g)$ and $V(\lambda)$.
We first re-construct crystal basis theory with modified Kashiwara operators.
While going through Kashiwara's grand-loop argument, we prove several important lemmas,
which play crucial roles in the later developments of the paper.
Next, based on the theory of canonical bases on quantum Bocherds-Bozec algebras, we
introduce the notion of primitive canonical bases and  prove that
primitive canonical bases coincide with lower global bases.
\end{abstract}

\maketitle
\tableofcontents
\section{Introduction}
\vskip 2mm
\subsection{Background}
In representation theory, it is always an important task to construct explicit bases of algebraic objects
because those bases provide a deep insight in studying the various features and properties
of these algebraic objects.
The {\it quantum groups}, as a new class of non-commutative, non-cocommutative Hopf algebras, were
discovered independently by Drinfeld and Jimbo in their study of quantum Yang-Baxter equation
and 2-dimensional solvable lattice model \cite{Dr85, Jimbo85}.
For the past forty years, the quantum groups have attracted a lot of research activities due to their
close connection with representation theory, combinatorics, knot theory, mathematical physics, etc.
Among others, Lusztig's {\it canonical basis theory} and Kashiwara's {\it crystal basis theory}
are regarded as one of the most prominent achievements in the representation theory of quantum groups
\cite{Lus90a, Lus91, Kashi90, Kashi91}.
The canonical basis theory was developed in a geometric way, while the crystal  basis theory was
constructed using algebraic methods.

\vskip 3mm

From geometric point of view, Lusztig's canonical basis theory is closely related to the theory of
perverse sheaves on the representation variety of quivers without loops.
In \cite{Bozec2014b, Bozec2014c}, Bozec extended Lusztig's theory to the study of perverse sheaves
for the quivers with loops, thereby introduced the notion of {\it quantum Borcherds-Bozec algebras}.
From algebraic point of view, the quantum Borcherds-Bozec algebras can be regarded as a huge
generalization of quantum groups and quantum Borcherds algebras \cite{Dr85, Jimbo85, Kang95}.

\vskip 3mm

The theory of  canonical bases, crystal bases and global bases for
quantum Borcherds-Bozec algebras have been
developed and investigated in  \cite{Bozec2014b, Bozec2014c,FKKT22}.
For the case of quantum groups associated with symmetric Cartan matrices,
Grojnowski and Lusztig discovered that the canonical bases coincide with global bases \cite{GL93}.
Moreover, for the case of quantum Borcherds algebras associated with symmetric Borcherds-Cartan
matrices without isotropic simple roots, Kang and Schiffmann showed that
the canonical bases coincide with the global bases \cite{KS06}.

\vskip 3mm

The aim of this paper is to investigate the deep connections
between most significant bases for quantum Borcherds-Bozec algebras:
canonical bases and crystal/global bases.
We will show that the canonical bases coincide with global bases. 
Moreover, we expect there are much more to be explored in the theory of
quantum Borcherds-Bozec algebras from various points of view.

\vskip 3mm

\subsection{New crystal basis theory}
Let $U_{q}^{-}(\g)$ be the negative half of a quantum Borcherds-Bozec algebra $U_{q}(\g)$
associated with a Borcherds-Cartan datum $(A, P, P^{\vee}, \Pi, \Pi^{\vee})$
and let $V(\lambda)$ be the irreducible highest weight module with $\lambda \in P^{+}$.
For our purpose, we re-construct the crystal basis theory for $V(\lambda)$ and $U_{q}^{-}(\g)$.
More precisely, we first define a new class of Kashiwara operators on $V(\lambda)$
and  $U_{q}^{-}(\g)$ which is a modified version of the ones
given in \cite{Bozec2014c}. The main difference from Bozec's definition is the case of $i\in I^{\text{iso}}$, 
where we define the Kashiwara operators as follows (Definition \ref{def:Kashiwara operator}, Definition \ref{def:Kashiwara operators}):
\begin{equation*}
 \widetilde{e}_{il} u = \sum_{\mathbf{c} \in {\mathcal C}_{i}} \,  \mathbf{c}_{l} \, {\mathtt b}_{i, \mathbf{c} \setminus \{l\}} u_{\mathbf{c}},\quad
  \widetilde{f}_{il}  u = \sum_{\mathbf{c} \in {\mathcal C}_{i}} \, \frac{1} {\mathbf{c}_{l} + 1} \,
		{\mathtt b}_{i,\{l\} \cup \mathbf{c})} u_{\mathbf{c}}.
\end{equation*}

We use these new Kashiwara operators to define the pairs $(L(\lambda), B(\lambda))$ and $(L(\infty), B(\infty))$ for $V(\lambda)$ and $U_{q}^{-}(\g)$, respectively. Then we prove that all the interlocking, inductive statements in Kashiwara's grand-loop argument are true,
thereby proving the existence and uniqueness of these crystal bases:
\begin{customthm}{\bf A}[Theorem \ref{thm:crystal basis of V}, Theorem \ref{thm:crystal basis of U}]\hfill
	\begin{enumerate}
\item The pair $(L(\lambda),B(\lambda))$ is a crystal basis of $V(\lambda)$.
	\vskip 2mm
\item The pair $(L(\infty),B(\infty))$ is a crystal basis of $U_{q}^{-}(\g)$.
\end{enumerate}
	\end{customthm}

We further use these new crystal bases to construct global bases for $V(\lambda)$ and $U_{q}^{-}(\g)$
and then verify that the global basis theory developed in \cite{FKKT22}
remains true with an appropriate modification.

\vskip 3mm
\subsection{Canonical bases and global bases}
In order to study the connection between canonical bases and global bases,
we define the notion of {\it primitive canonical bases}.
Recall that in \cite{FKKT22}, we gave an alternative presentation of $U_{q}(\g)$ in terms of
primitive generators which arise naturally from Bozec's algebra isomorphism
$\phi:U_{q}^{-}(\g) \rightarrow U_{q}^{-}(\g)$ \cite{Bozec2014b, Bozec2014c}
(See Proposition \ref{prop:primitive} in this paper).
The primitive canonical bases are defined as the image of canonical bases
under the isomorphism $\phi$.

\vskip 3mm

In Proposition \ref{prop:Bozec geometry}, we recall Bozec's geometric results on
canonical basis $\B$ and in Corollary \ref{cor:BB-B}, we rewrite them in an algebraic way.
Thus in Corollary \ref{cor:betabeta-beta}, we obtain an interpretation of Bozec's results
on the primitive canonical basis $\B_{\Q}$. Using some critical properties of  Lusztig's bilinear form $(\  , \  )_{L}$ and Kashiwara's bilinear form $(\  ,  \  )_{K}$, we prove the following  Propositions which play an important role in the later development.
\begin{customprop}{\bf B}[Proposition \ref{prop:compareLK}]
 For all $x, y \in U_{q}^-(\g)$, we have 	
\begin{equation*}
	(x, y)_{L} = (x, y)_{K} \ \ \text{mod} \ \, q\, \A_{0}.
\end{equation*}
	\end{customprop}

\begin{customprop}{\bf C}[Proposition \ref{prop:prim bilinear}]
 For all $x, y \in U_{q}^-(\g)$, we have
\begin{equation*}
	(\phi(x), \phi(y))_{L} = (x, y)_{L}.
\end{equation*}
\end{customprop}	
Combining all these results, we can apply Grojnowski-Lusztig's argument to our setting,
from which we conclude that the primitive canonical basis $\B_{\Q}$ coincides with the lower
global basis $\B(\infty)$.
As an immediate consequence, we deduce that the primitive canonical basis $\B_{\Q}^{\lambda}$
coincides with the lower global basis $\B(\lambda)$.

\vskip 3mm
\subsection{Organization}
This paper is organized as follows.

In the first part, we focus on the re-construction
of crystal basis theory for quantum Borcherds-Bozec algebras.
More precisely, in Section \ref{sec:qBBalg}, we recall the original definition of
quantum Borcherds-Bozec algebras and their alternative presentation
in terms of primitive generators.
In Section \ref{sec:crystal}, we define a new class of Kashiwara operators
and construct the crystal bases $(L(\lambda), B(\lambda))$ for $V(\lambda)$
and $(L(\infty), B(\infty))$ for $U_{q}^{-}(\g)$.
We also review some of the basic theory of abstract crystals and
give a simplified description of tensor product rule for quantum Borcherds-Bozec algebras.
In Section \ref{sec:grand-loop}, with the new class of Kashiwara operators,
we go through all the interlocking, inductive statements in Kashiwara's grand-loop argument
and show that all of them are still true in our much more general setting.
Hence we prove the existence and uniqueness of the crystal bases
$(L(\lambda), B(\lambda))$ and $(L(\infty), B(\infty))$.
As by-products, we obtain several important lemmas which will be used in later
parts of this work in a critical way (for example, Lemma \ref{lem:LLA0}).
In Section \ref{sec:global bases}, we study the lower global bases $\B(\lambda)$
and $\B(\infty)$ following the outline given in \cite{FKKT22}.

\vskip 3mm

The second part of this paper is devoted to the study of relations between canonical bases
and global bases.
More precisely, in Section \ref{sec:primitive},
we recall the geometric construction of canonical basis $\B$
and define the notion of primitive canonical basis $\B_{\Q}$.
We then give a very brief review of some homological formulas,
which leads to defining geometric bilinear form $( \ , \  )_{G}$
on perverse sheaves \cite{Lus10}.
The geometric results  proved by Bozec \cite{Bozec2014b, Bozec2014c}
are expressed in algebraic language and then translated
to the corresponding statements for primitive canonical bases.
We close this section with several important key  lemmas
on global bases which are necessary to apply Grojnowski-Lusztig's argument.

\vskip 3mm

In Section \ref{sec:Primlobal}, we first identify the geometric bilinear form
and Lusztig's bilinear form using the fact that both of them are Hopf pairings.
We then show that Lusztig's bilinear form and Kashiwara's bilinear form
are equivalent to each other up to $\text{mod} \ q \, \A_{0}$.
Using the key lemmas proved in Section \ref{sec:primitive},
we can apply Grojnowski-Lusztig's argument to conclude
the primitive canonical basis $\B_{\Q}$ coincides with the lower global basis $\B(\infty)$.
It follows immediately that the primitive canonical basis $\B^{\lambda}_{\Q}$ is identical
to the lower global basis $\B(\lambda)$.

\vskip 3mm


\vskip 3mm

\noindent
{\bfseries  Acknowledgements.}
Z. Fan was partially supported by the NSF of China grant 12271120, the NSF of Heilongjiang Province grant JQ2020A001, and the Fundamental Research Funds for the central universities. S.-J. Kang was supported by China grant YZ2260010601. Young Rock Kim was supported by the National Research Foundation
of Korea grant 2021R1A2C1011467 and
 Hankuk University of Foreign Studies Research Fund.

\section{Quantum Borcherds-Bozec algebras} \label{sec:qBBalg}

\vskip 2mm

Let $I$ be an index set which can be countably infinite. An integer-valued matrix $A=(a_{ij})_{i,j \in I}$ is called an {\it
even symmetrizable Borcherds-Cartan matrix} if it satisfies the following conditions:
\begin{itemize}
\item[(i)] $a_{ii}=2, 0, -2, -4, ...$,

\item[(ii)] $a_{ij}\le 0$ for $i \neq j$,

\item[(iii)] there exists a diagonal matrix $D=\text{diag} (s_{i} \in \mathbf Z_{>0} \mid i \in I)$ such that $DA$ is symmetric.
\end{itemize}

\vskip 2mm

\noindent Set $I^{\text{re}}=\{i \in I \mid a_{ii}=2 \}$,
$I^{\text{im}}=\{i \in I \mid a_{ii} \le 0\}$ and
$I^{\text{iso}}=\{i \in I \mid a_{ii}=0 \}$.

\vskip 3mm

A {\it Borcherds-Cartan datum} consists of :

\begin{itemize}

\item[(a)] an even symmetrizable Borcherds-Cartan matrix $A=(a_{ij})_{i,j \in I}$,

\item[(b)] a free abelian group $P$, the {\it weight lattice},

\item[(c)] $P^{\vee} := \Hom(P, \Z)$, the {\it dual weight lattice},

\item[(d)] $\Pi=\{\alpha_{i} \in P  \mid i \in I \}$, the set of {\it simple roots},

\item[(e)] $\Pi^{\vee}=\{h_i \in P^{\vee} \mid i \in I \}$, the set of {\it simple coroots}

\end{itemize}

\vskip 2mm

\noindent satisfying the following conditions

\vskip 1mm

\begin{itemize}

\item[(i)] $\langle h_i, \alpha_j \rangle = a_{ij}$ for all $i, j \in I$,

\item[(ii)] $\Pi$ is linearly independent over $\mathbf Q$,

\item[(iii)] for each $i \in I$, there exists an element $\Lambda_{i} \in\mathrm P$ such that $$\langle h_j , \Lambda_i
\rangle = \delta_{ij} \ \ \text{for all} \ i, j \in I.$$
\end{itemize}

\vskip 2mm
\noindent
The elements $\Lambda_i$  $(i \in I)$ are called the {\it fundamental weights}.

\vskip 3mm

Given an even symmetrizable Borcherds-Cartan matrix, it can be shown that such a Borcherds-Cartan datum always exists, which
is not necessarily unique.

\vskip 3mm

We denote by
$$P^{+}:=\{\lambda \in P \mid \langle h_i, \lambda \rangle \ge 0 \ \text{for all} \ i \in I \},$$
the set of {\it dominant integral weights}. The free abelian group $R:= \bigoplus_{i \in I} \Z \, \alpha_i$ is called the {\it root lattice}.
Set $R_{+}: = \sum_{i \in I} \Z_{\ge 0}\, \alpha_{i}$ and $R_{-}: = -R_{+}$.
 Let ${\mathfrak h} := \Q \otimes_{\Z} P^{\vee}$ be the {\it Cartan subalgebra}.

\vskip 3mm

Since $A$ is symmetrizable and $\Pi$  is linearly independent over $\Q$,
there exists a non-degenerate symmetric bilinear
form $( \ , \ )$ on ${\mathfrak h}^{*}$ satisfying
$$(\alpha_{i}, \lambda) = s_{i} \langle h_{i}, \lambda \rangle \quad
\text{for all}  \ \lambda \in {\mathfrak h}^{*}.$$

\vskip 3mm

For each  $i \in I^{\text{re}}$, we define the {\it simple reflection} $r_{i} \in GL(\h^{*})$ by
$$r_{i}(\lambda)= \lambda - \langle h_{i}, \lambda \rangle \,  \alpha_{i} \ \ \text{for} \ \lambda \in {\mathfrak h}^{*}.$$

\noindent
The subgroup $W$ of $GL({\mathfrak h}^{*})$ generated by the simple reflections $r_{i}$ $(i \in I_{\text{re}})$ is called the {\it Weyl group} of the Borcherds-Cartan datum given above.  It is easy to check that $(\ , \ )$ is $W$-invariant.

\vskip 3mm

Let $q$ be an indeterminate.
For $i\in I$ and $n \in \mathbf Z_{> 0}$, we define
$$q_{i}  = q^{s_i}, \quad q_{(i)} = q^{\frac{(\alpha_{i}, \alpha_{i})}{2}},
\quad [n]_{i} = \dfrac{q_{i}^{n} - q_{i}^{-n}} {q_{i} - q_{i}^{-1}}, \quad [n]_{i} ! = [n]_{i} [n-1]_{i} \cdots [1]_{i}.$$

\vskip 3mm

Set $I^\infty:= I^{\text{re}} \cup (I^{\text{im}}
\times \Z_{>0})$ 
and let
$\mathscr{F} = \Q(q) \langle f_{il} \mid (i,l) \in I^{\infty} \rangle$ be the free associative algebra generated by the
formal symbols $f_{il}$ with $(i, l) \in I^{\infty}$.
By setting $\deg f_{il} = - l \alpha_{i}$, then $\mathscr{F}$ becomes a $R_{-}$-graded algebra.
For a homogeneous element $x \in \mathscr{F}$, we denote by $|x|$ the degree of $x$
and for a subset $A \subset R_{-}$, we define
$$\mathscr{F}_{A} = \{x \in \mathscr{F} \mid |x| \in A \}.$$

\vskip 2mm

Following \cite{Ringel90},
we define a {\it twisted multiplication} on $\mathscr{F} \otimes \mathscr{F}$ by
\begin{equation*}
(x_{1} \otimes x_{2}) (y_{1} \otimes y_{2})
= q^{-(|x_{2}|, |y_{1}|)} x_{1} y_{1} \otimes x_{2} y_{2}
\end{equation*}
for all homogeneous elements $x_{1}, x_{2}, y_{1}, y_{2} \in \mathscr{F}$.

\vskip 2mm

We also define a $\Q(q)$-algebra homomorphism
$\delta: \mathscr{F} \longrightarrow \mathscr{F} \otimes \mathscr{F}$
given by
\begin{equation} \label{eq:comult}
\delta(f_{il}) = \sum_{m+n = l} q_{(i)}^{-mn} f_{im} \otimes f_{in} \ \ \text{for} \ (i,l)\in I^{\infty},
\end{equation}
where we understand $f_{i0}=1$ and $f_{il}=0$ for $l<0$.
Then $\mathscr{F}$ becomes a $\Q(q)$-bialgebra.

\vskip 3mm

\begin{proposition} \cite{Lus10, Bozec2014b, Bozec2014c} \label{prop:Lusztig}
{\rm
Let $\nu = (\nu_{il})_{(i,l) \in I_{\infty}}$ be a family of non-zero elements in $\Q(q)$.
Then there exists a symmetric bilinear form $(\ , \ )_{L} : \mathscr{F} \times  \mathscr{F} \longrightarrow \Q(q)$
such that
\begin{itemize}
\item[(a)] $(\mathbf{1}, \mathbf{1})_{L} = 1$,

\item[(b)] $(f_{il}, f_{il})_{L} = \nu_{il}$ for $(i,l) \in I^{\infty}$,

\item[(c)] $(x, y)_{L} = 0$ if $|x| \neq |y|$,

\item[(d)] $(x, yz)_{L} = (\delta(x), y \otimes z)_{L}$ for all $x, y, z \in \mathscr{F}$.

\end{itemize}

}
\end{proposition}

\vskip 3mm

Let $\mathscr{R}$ be the radical of $(\ , \ )_{L}$ on $\mathscr{F}$.  Assume that
\begin{equation} \label{eq:nu}
\nu_{il} \equiv 1\!\!\!\! \mod q\, \Z_{\ge 0}[[q]] \ \ \text{for all} \ i \in I^{\text{im}} \setminus I^{\text{iso}}
\ \text{and} \ l>0.
\end{equation}

\vskip 2mm

Then it was shown in \cite{Bozec2014b, Bozec2014c} that the radical $\mathscr{R}$ is generated by the elements
\begin{equation} \label{eq:radical}
\begin{aligned}
& \sum_{r+s = 1 - la_{ij}} (-1)^r f_{i}^{(r)} f_{jl} f_{i}^{(s)} \ \text{for} \ i \in I^{\text{re}}, \ i \neq (j,l) \in I^{\infty}, \\
& f_{il} f_{jk} - f_{jk} f_{il} \ \text{for all} \ (i,l), (j,k) \in I^{\infty} \ \text{and} \ a_{ij} =0,
\end{aligned}
\end{equation}
where $f_{i}^{(n)} = f_{i}^n / [n]_{i}!$ for $i \in I^{\text{re}}$.

\vskip 3mm

Given a Borcherds-Cartan datum $(A, P, P^{\vee}, \Pi,  \Pi^{\vee})$,
we define $\widehat {U}$ to be  the associative algebra over $\Q(q)$ with $\mathbf 1$,
generated by the elements $q^h$ $(h\in P^{\vee})$ and $e_{il},
f_{il}$ $((i,l) \in I^{\infty})$ with defining relations
\begin{equation} \label{eq:rels}
\begin{aligned}
& q^0=\mathbf 1,\quad q^hq^{h'}=q^{h+h'} \ \ \text{for} \ h,h' \in P^{\vee} \\
& q^h e_{jl}q^{-h} = q^{l \langle h, \alpha_j \rangle} e_{jl},
\ \ q^h f_{jl}q^{-h} = q^{-l \langle h, \alpha_j \rangle} f_{jl}\ \ \text{for} \ h \in P^{\vee},\ (j,l)\in I^{\infty}, \\
& \sum_{r+s=1-la_{ij}}(-1)^r
{e_i}^{(r)}e_{jl}e_i^{(s)}=0 \ \ \text{for} \ i\in
I^{\text{re}},\ (j,l)\in I^{\infty} \ \text {and} \ i \neq (j,l), \\
& \sum_{r+s=1-la_{ij}}(-1)^r
{f_i}^{(r)}f_{jl}f_i^{(s)}=0 \ \ \text{for} \ i\in
I^{\text{re}},\ (j,l)\in I^{\infty} \ \text {and} \ i \neq (j,l), \\
& e_{ik}e_{jl}-e_{jl}e_{ik} = f_{ik}f_{jl}-f_{jl}f_{ik} =0 \ \ \text{for} \ a_{ij}=0.
\end{aligned}
\end{equation}

\noindent
We extend the grading on $\widehat{U}$ by setting $|q^h|=0$ and $|e_{il}|= l \alpha_{i}$.

\vskip 3mm

The algebra $\widehat{U}$ is endowed with a comultiplication
$\Delta: \widehat{U} \rightarrow \widehat{U} \otimes \widehat{U}$
given by
\begin{equation} \label{eq:Comult}
\begin{aligned}
& \Delta(q^h) = q^h \otimes q^h, \\
& \Delta(e_{il}) = \sum_{m+n=l} q_{(i)}^{mn}e_{im}\otimes K_{i}^{-m}e_{in}, \\
& \Delta(f_{il}) = \sum_{m+n=l} q_{(i)}^{-mn}f_{im}K_{i}^{n}\otimes f_{in},
\end{aligned}
\end{equation}
where $K_i= q_{i}^{h_{i}} =q^{s_ih_i}$ $(i \in I)$.

\vskip 3mm

Let $\widehat{U}^+$ (resp. $\widehat{U}^{-}$) be the subalgebra of $\widehat{U}$
generated by $e_{il}$ (resp. $f_{il}$) for all $(i,l) \in I^{\infty}$.
In particular, $\widehat{U}^- \cong \mathscr{F} / \mathscr{R}$.

\vskip 3mm

We denote by $\widehat{U}^{\leq0}$ be the subalgebra of $\widehat{U}$
generated by $q^h$ $(h \in P^{\vee})$ and $f_{il}$ $((i,l) \in I^{\infty})$.
We extend $( \ , \ )_L$ to a symmetric bilinear form $( \ , \ )_L$ on $\widehat{U}^{\leq 0}$  by setting
\begin{equation} \label{eq:biliearUminus}
\begin{aligned}
& (q^h,\mathbf 1)_L=1,\ (q^h,f_{il})_L=0, \\
& (q^h,K_j)_L=q^{-\langle h, \alpha_j \rangle}.
\end{aligned}
\end{equation}

\noindent
Moreover, we define $(\ , \ )_{L}$ on $\widehat{U}^{+}$ by
\begin{equation} \label{eq:bilinearUplus}
(x, y)_{L} = (\omega(x), \omega(y))_{L} \ \ \text{for} \ x, y \in \widehat{U}^{+},
\end{equation}
where $\omega:\widehat{U} \longrightarrow \widehat{U}$  is the involution defined by
$$\omega(q^h)=q^{-h},\ \omega(e_{il})=f_{il},\ \omega(f_{il})=e_{il}\ \
\text{for}\ h \in P^{\vee},\ (i,l)\in I^{\infty}.$$

\vskip 3mm

For any $x\in \widehat{U}$, we will  use the Sweedler notation to write
$$\Delta(x)=\sum x_{(1)}\otimes x_{(2)}.$$
Following the Drinfeld double construction, the quantum Borcherds-Bozec aalgebra
is defined as follows.

\vskip 3mm

\begin{definition} The {\it quantum Borcherds-Bozec algebra} $U_q(\g)$ associated with a Borcherds-Cartan datum $(A, P,P^{\vee}, \Pi, \Pi^{\vee})$ is the quotient algebra of $\widehat{U}$ defined by relations
\begin{equation}\label{drinfeld}
\sum(a_{(1)},b_{(2)})_L\omega(b_{(1)})a_{(2)}=\sum(a_{(2)},b_{(1)})_La_{(1)}\omega(b_{(2)})\ \ \text{for all}\ a,b \in \widehat{U}^{\leq0}.
\end{equation}
\end{definition}


\vskip 3mm

Let $U^+_q(\g)$ (resp. $U^-_q(\g)$) be the subalgebra of $U_q(\g)$ generated by $e_{il}$
(resp. $f_{il}$) for $(i,l)\in I^{\infty}$
and let $U^{0}_q(\g)$ be the subalgebra of $U_q(\g)$ generated by $q^h$ $(h\in P^{\vee})$.
Then we have the {\it triangular decomposition} \cite{KK20}
$$U_{q}(\g) \cong U_{q}^-(\g) \otimes U_{q}^0(\g) \otimes U_{q}^+(\g).$$

\vskip 3mm
\noindent
For simplicity, we often write
 $U$ (resp. $U^+$ and $U^-$) for $U_q(\g)$ (resp. $U^+_q(\g)$ and $U^-_q(\g)$).

 \vskip 3mm

Let $^{-}:U_q(\g)\rightarrow U_q(\g)$ be the $\Q$-linear involution given by
\begin{equation}\label{eq:bar}
\overline{e_{il}}=e_{il},\quad \overline{f_{il}}=f_{il},\quad \overline{K_i}=K_i^{-1},\quad \overline{q}=q^{-1}
\end{equation} 	
for $(i,l)\in I^{\infty}$ and $i\in I$.

 \vskip 3mm

 The following proposition will play an extremely important role in our work.

 \vskip 2mm

 \begin{proposition} \cite{Bozec2014b, Bozec2014c} \label{prop:primitive}
 {\rm
 For each $i \in I^{\text{im}}$ and $l>0$, there exist unique elements
 ${\mathtt a}_{il}$, ${\mathtt b}_{il} = \omega({\mathtt a}_{il})$ satisfying the
 following conditions.

 \begin{itemize}

 \item[(a)] $\Q(q) \langle e_{i1}, e_{i2}, \ldots, e_{il} \rangle
 = \Q(q) \langle {\mathtt a}_{i1}, {\mathtt a}_{i2}, \ldots, {\mathtt a}_{il} \rangle$,

\vskip 2mm

 \hspace{-4mm}$\Q(q) \langle f_{i1}, f_{i2}, \ldots, f_{il} \rangle
 = \Q(q) \langle {\mathtt b}_{i1}, {\mathtt b}_{i2}, \ldots, {\mathtt b}_{il} \rangle$,

 \vskip 2mm

 \item[(b)] $({\mathtt a}_{il}, u)_{L} = 0$ for all $u \in \Q(q) \langle e_{ik} \mid k < l \rangle$,

\vskip 2mm

 \hspace{-4mm}$({\mathtt b}_{il}, w)_{L} = 0$ for all $w \in \Q(q) \langle f_{ik} \mid k<l \rangle$,

 \vskip 2mm

 \item[(c)] ${\mathtt a}_{il} - e_{il} \in \Q(q) \langle e_{ik} \mid k<l \rangle$, \
 ${\mathtt b}_{il} - f_{il} \in \Q(q) \langle f_{ik} \mid k<l \rangle$,

 \vskip 2mm

 \item[(d)] $\overline{{\mathtt a}_{il}} = {\mathtt a}_{il}$, \ $\overline{{\mathtt b}_{il}} = {\mathtt b}_{il}$,

 \vskip 2mm

 \item[(g)] $\delta({\mathtt a}_{il}) = {\mathtt a}_{il} \otimes\mathbf 1 +\mathbf 1 \otimes {\mathtt a}_{il}$, \
 $\delta({\mathtt b}_{il}) = {\mathtt b}_{il} \otimes\mathbf 1 +\mathbf 1 \otimes {\mathtt b}_{il}$.

 \end{itemize}
 }
 \end{proposition}

\vskip 3mm

Let $\tau_{il} = ({\mathtt a}_{il}, {\mathtt a}_{il})_{L} =  ({\mathtt b}_{il}, {\mathtt b}_{il})_{L}.$
In \cite{FKKT22}, we obtain a new presentation of the quantum Borcherds-Bozec algebra $U_{q}(\g)$
in terms of {\it primitive generators} $q^{h}$ $(h \in P^{\vee})$,
${\mathtt a}_{il}$, ${\mathtt b}_{il}$ $((i,l) \in I^{\infty})$.

\vskip 3mm

\begin{theorem}\cite[Theorem 2.5]{FKKT22}
{\rm The quantum Borcherds-Bozec algebra $U_{q}(\g)$ is equal to the
associative algebra over $\Q(q)$ with $\mathbf{1}$ generated by
 $q^{h}$ $(h \in P^{\vee})$, ${\mathtt a}_{il}$, ${\mathtt b}_{il}$ $((i,l) \in I^{\infty})$
 with the defining relations
 \begin{equation} \label{eq:primitive}
\begin{aligned}
& q^0=\mathbf 1,\quad q^hq^{h'}=q^{h+h'} \ \ \text{for} \ h,h' \in P^{\vee}, \\
& q^h \mathtt a_{jl}q^{-h} = q^{l \langle h, \alpha_j \rangle} \mathtt a_{jl},
\ \ q^h \mathtt b_{jl}q^{-h} = q^{-l \langle h, \alpha_j \rangle } \mathtt b_{jl}
\ \ \text{for} \ h \in P^{\vee}\ \text{and}\ (j,l)\in I^{\infty}, \\
& \sum_{r + s = 1 - l a_{ij}} (-1)^r
{\mathtt a}_i^{(r)}\mathtt a_{jl}\mathtt a_i^{(s)}=0
\ \ \text{for} \ i\in 	I^{\text{re}},\ (j,l)\in I^{\infty} \ \text {and} \ i \neq (j,l), \\
& \sum_{r + s = 1 - l a_{ij}} (-1)^r
{\mathtt b}_i^{(r)}\mathtt b_{jl}\mathtt b_i^{(s)}=0
\ \ \text{for} \ i\in 	I^{\text{re}},\ (j,l)\in I^{\infty} \ \text {and} \ i \neq (j,l), \\
& \mathtt a_{il}\mathtt b_{jk} - \mathtt b_{jk}\mathtt a_{il}=\delta_{ij}\delta_{kl}\tau_{il}(K_{i}^{l} - K_{i}^{-l}), \\
& \mathtt a_{il}\mathtt a_{jk}-\mathtt a_{jk}\mathtt a_{il} = \mathtt b_{il}\mathtt b_{jk}-\mathtt b_{jk}\mathtt b_{il} =0 \ \ \text{for} \ a_{ij}=0.
\end{aligned}
\end{equation}

}
\end{theorem}

\vskip 3mm

\noindent
Note that $U^+ = \langle {\mathtt a}_{il} \mid (i, l) \in I_{\infty} \rangle$
and $U^- = \langle {\mathtt b}_{il} \mid (i, l) \in I_{\infty} \rangle$.

\vskip 3mm

The algebra $U_{q}(\g)$ has a comultiplication induced by \eqref{eq:comult} and Proposition \ref{prop:primitive}.
\begin{equation} \label{eq:Comult}
\begin{aligned}
& \Delta(q^h) = q^h \otimes q^h, \\
& \Delta({\mathtt a}_{il}) = {\mathtt a}_{il} \otimes K_{i}^{-l}  + \mathbf 1 \otimes {\mathtt a}_{il}, \\
& \Delta({\mathtt b}_{il}) = {\mathtt b}_{il} \otimes \mathbf 1 + K_{i}^l  \otimes {\mathtt b}_{il}.
\end{aligned}
\end{equation}

\vskip 2mm

\noindent
Moreover, we define the counit and antipode by
\begin{equation}
\begin{aligned}
& \epsilon(q^h) = 1, \ \ \epsilon({\mathtt a}_{il}) =  \epsilon({\mathtt b}_{il}) = 0, \\
& S({\mathtt a}_{il}) = - {\mathtt a}_{il} K_{i}^l, \ \ S({\mathtt b}_{il}) = - K_{i}^{-l} {\mathtt b}_{il},
\end{aligned}
\end{equation}
then the quantum Borcherds-Bozec algebra $U_{q}(\g)$ becomes a Hopf algebra.

\vskip 3mm

From now on, we will take
$$\tau_{il} = (1 - q_{i}^{2l})^{-1} \ \ \text{for} \ (i, l) \in I^{\infty}.$$

\vskip 2mm

Set $A_{il}:=-q^l_i\mathtt a_{il}$ and $E_{il}:=-K_i^l\mathtt a_{il}$. Then we have
\begin{align}
&A_{il}\mathtt b_{jk}-\mathtt b_{jk}A_{il}=\delta_{ij}\delta_{kl}\frac{K_i^l-K_i^{-l}}{q_i^l-q_i^{-l}},\\
&E_{il}\mathtt b_{jk}-q_i^{-kla_{ij}}\mathtt b_{jk}E_{il}=\delta_{ij}\delta_{kl}\frac{1-K_i^{2l}}{1-q_i^{2l}}.\label{Eb}
\end{align}

\vskip 5mm

We now briefly review some of the basic properties of the category $\mathcal{O}_{\text{int}}$.
Let $U_{q}(\g)$ be a quantum Borcherds-Bozec algebra and let $M$ be a $U_{q}(\g)$-module. We say that $M$ has
a {\it weight space decomposition} if
$$M = \bigoplus_{\mu \in P} M_{\mu}, \ \ \text{where}
 \ M_{\mu} = \{ m \in M \mid q^{h} \,  m = q^{\langle h, \mu \rangle} m \
\text{for all} \ h \in P^{\vee} \}.$$
\noindent
We denote  $\text{wt}(M):=\{\mu \in
\h^* \mid M_{\mu} \neq 0 \}$.

\vskip 3mm

A $U_{q}(\g)$-module $V$ is called a {\it highest weight module with highest weight $\lambda$} if there is a non-zero vector $v_{\lambda}$ in $V$ such that
\begin{enumerate}
\item[(i)] $q^{h} \,  v_{\lambda} = q^{\langle h, \lambda \rangle} v_{\lambda}$ for all $h \in P^{\vee}$,

\item[(ii)] $e_{il} \, v_{\lambda} = 0$ for all $(i,l) \in I^{\infty}$,

\item[(iii)] $V=U_{q}(\g) v_{\lambda}$.
\end{enumerate}

\vskip 2mm

Such a vector $v_{\lambda}$ is called a {\it highest weight vector} with highest weight $\lambda$.
Note that $V_{\lambda} = \Q(q) v_{\lambda}$ and $V$ has a weight space decomposition
 $V = \bigoplus_{\mu \le \lambda} V_{\mu}$,
 where $\mu \le \lambda$ means $\lambda - \mu \in R_{+}$. For each $\lambda \in P$, there exists a unique irreducible highest weight module,
 which is denoted by $V(\lambda)$.

 \vskip 3mm

 \begin{proposition}  \cite{KK20} \label{prop:hw}
{\rm Let $\lambda \in P^{+}$ be a dominant integral weight and let $V(\lambda) = U_{q}(\g) \, v_{\lambda}$ be the irreducible highest weight module with highest weight $\lambda$ and highest weight vector $v_{\lambda}$.
Then the following statements hold.

\vskip 2mm
\begin{enumerate}
\item[(a)] If $i \in I^{\text{re}}$, then ${\mathtt b}_{i}^{\langle h_{i}, \lambda \rangle +1} v_{\lambda} =0$.

\vskip 2mm

\item[(b)] If $i \in I^{\text{im}}$ and $\langle h_{i}, \lambda \rangle =0$, then ${\mathtt b}_{il} \, v_{\lambda}=0$ for all $l>0$.
\end{enumerate}

\vskip 2mm

Moreover, if $i \in I^{\text{im}}$ and $\mu \in \text{wt}(V(\lambda))$, we have

\vskip 2mm

\begin{enumerate}

\item[(i)] $\langle h_{i}, \mu \rangle \ge 0$,

\vskip 2mm

\item[(ii)] if $\langle h_{i}, \mu \rangle  =0$, then $V(\lambda)_{\mu - l \alpha_{i}} =0$  for all $l>0$,

\vskip 2mm

\item[(iii)] if $\langle h_{i}, \mu \rangle  =0$, then $f_{il}(V(\lambda)_{\mu})=0$,

\vskip 2mm

\item[(iv)] if $\langle h_{i} , \mu \rangle \le -l a_{ii}$, then $e_{il}(V(\lambda)_{\mu})=0$.

\end{enumerate}

}
\end{proposition}

 \vskip 3mm

Motivated by Proposition \ref{prop:hw}, we define the category ${\mathcal O}_{\text{int}}$ as follows.

\vskip 2mm

\begin{definition} \label{def:Oint}

The {\it category} $\mathcal O_{\text{int}}$ consists of $U_{q}(\g)$-modules $M$ such that

 \begin{enumerate}

 \item[(a)] $M$ has a weight space decomposition $M = \oplus_{\mu \in P} M_{\mu}$
 with $\dim M_{\mu} < \infty$ for all $\mu \in \text{wt}(M)$,

 \vskip 2mm

 \item[(b)] there exist finitely many wrights $\lambda_{1}, \ldots, \lambda_{s} \in P$ such that
 $$\text{wt}(M) \subset \cup_{j=1}^s (\lambda_j - R_{+}),$$

 \vskip 2mm

\item[(c)] if $i \in I^{\text{re}}$, ${\mathtt b}_{i}$ is locally nilpotent
on $M$,

\vskip 2mm

\item[(d)] if $i \in I^{\text{im}}$, we have $\langle h_i, \mu
\rangle \ge 0$ for all $\mu \in \text{wt}(M)$,

\vskip 2mm

\item[(e)] if $i \in I^{\text{im}}$ and $\langle h_{i}, \mu \rangle =0$,
then ${\mathtt b}_{il}(M_{\mu})=0$,

\vskip 2mm

\item[(f)] if $i \in I^{\text{im}}$ and $\langle h_i, \mu \rangle
\le - l a_{ii}$, then ${\mathtt a}_{il}(M_{\mu})=0$.

 \end{enumerate}
 \end{definition}

 \vskip 3mm

\begin{remark} \hfill

\begin{itemize}

\item[(i)] By (b), ${\mathtt a}_{il}$ is locally nilpotent on $M$.

\vskip 2mm

\item[(ii)] If $i \in I^{\text{im}}$, then ${\mathtt b}_{il}$ are not necessarily locally nilpotent.

\vskip 2mm

\item[(iii)] The irreducible highest weight $U_{q}(\g)$-module
$V(\lambda)$ with $\lambda \in P^{+}$ is an object of the category
${\mathcal O}_{\text{int}}$.

\vskip 2mm

\item[(iv)] A submodule or a quotient module of a $U_{q}(\g)$-module in  the category
${\mathcal O}_{\text{int}}$ is again an object of ${\mathcal O}_{\text{int}}$.

\vskip 2mm

\item[(v)] A finite number of direct sums or a finite number of tensor products of $U_{q}(\g)$-modules in the
category ${\mathcal O}_{\text{int}}$ is again an object of ${\mathcal O}_{\text{int}}$.

\end{itemize}
\end{remark}

\vskip 3mm

The fundamental properties of the category ${\mathcal O}_{\text{int}}$ are given below.

\vskip 3mm

\begin{proposition} \hfill

{\rm
\begin{itemize}

\vskip 2mm

\item[(a)] If a highest weight module $V = U_{q}(\g) v_{\lambda}$ satisfies the
conditions (a) and (b) in Proposition \ref{prop:hw}, then $V \cong V(\lambda)$ with $\lambda \in P^{+}$.

\vskip 2mm

\item[(b)] The category ${\mathcal O}_{\text{int}}$ is semisimple.

\vskip 2mm

\item[(c)] Every simple object in the category ${\mathcal O}_{\text{int}}$ has the form
$V(\lambda)$ for some $\lambda \in P^{+}$.

\end{itemize}

}

\end{proposition}

\vskip 3mm

\section{Crystal bases}\label{sec:crystal}

\vskip 2mm

Let  $\mathbf{c} = (c_1, \ldots, c_r) \in \Z_{\ge 0}^r$ be a
sequence of non-negative integers.
We define $|\mathbf{c}|:= c_1 + \cdots + c_r$.
We say that $\mathbf{c}$ is a {\it composition} of $l$, denoted  by
$\mathbf{c} \vdash l$, if $|\mathbf{c}| = l$.
If $c_1 \ge c_2 \ge \ldots \ge c_r$, we say that $\mathbf{c}$ is a {\it partition} of $l$.
For each $i\in I^{\text{im}} \setminus I^{\text{iso}}$ (resp. $i \in I^{\text{iso}}$),
we denote by $\mathcal{C}_{i, l}$ the set of compositions (resp. partitions) of $l$
and set $\mathcal{C}_{i} = \bigsqcup_{l \ge 0} \mathcal{C}_{i,l}$.
For $i \in I^{\text{re}}$, we define $\mathcal{C}_{i,l} = \{ l \}$.

\vskip 3mm

For $\mathbf{c} = (c_1, \ldots, c_r)$, we define
$${\mathtt a}_{i, \mathbf{c}} = {\mathtt a}_{i c_1} {\mathtt a}_{i c_2} \cdots {\mathtt a}_{i c_r},
\quad {\mathtt b}_{i, \mathbf{c}} = {\mathtt b}_{i c_1}{\mathtt b}_{i c_2} \cdots {\mathtt b}_{i c_r}.$$

\noindent
Note that $\{ {\mathtt a}_{i, \mathbf{c}} \mid \mathbf{c} \vdash l \}$
(resp. $\{ {\mathtt b}_{i, \mathbf{c}} \mid \mathbf{c} \vdash l \}$)
forms a basis of $U_{q}(\g)_{l \alpha_i}$ (resp. $U_{q}(\g)_{- l \alpha_i}$).

\vskip 5mm

\subsection{Crystal bases for $V(\lambda)$} \label{sub:Vlambda} \hfill

\vskip 2mm

Let $M = \oplus_{\mu \in P} M_{\mu}$ be a $U_{q}(\g)$-module in the category ${\mathcal O}_{\text{int}}$
and let  $u \in M_{\mu}$ for $\mu \in \text{wt}(M)$.

\vskip 3mm

For $i \in I^{\text{re}}$, by \cite{Kashi91}, the vector $u$ can be written uniquely as
\begin{equation} \label{eq:real string}
u = \sum_{k \ge 0} {\mathtt b}_{i}^{(k)} u_k
\end{equation}
such that

\vskip  2mm

\begin{enumerate}

\item[(i)] ${\mathtt a}_{i} u_k = 0$ for all $k \ge 0$,

\vskip 2mm

\item[(ii)] $u_k \in M_{\mu + k \alpha_{i}}$,

\vskip 2mm

\item[(iii)] $u_{k} = 0$ if $\langle h_{i}, \mu + k \alpha_{i} \rangle = 0 $.

\end{enumerate}

\vskip 3mm

For $i \in I^{\text{im}}$, by \cite{Bozec2014b, Bozec2014c}, the vector 
$u$ can be written uniquely as
\begin{equation} \label{eq:imaginary string}
u = \sum_{\mathbf c\in\mathcal C_i} {\mathtt b}_{i, \mathbf{c}} u_{\mathbf{c}}
\end{equation}
such that

\vskip 2mm

\begin{enumerate}

\item[(i)] ${\mathtt a}_{ik} u_{\mathbf{c}} = 0$ for all $k >0$,

\vskip 2mm

\item[(ii)] $u_{\mathbf{c}} \in M_{\mu+ |\mathbf{c}| \alpha_{i}}$,

\vskip 2mm

\item[(iii)] $u_{\mathbf{c}} = 0$ if $\langle h_{i}, \mu + |\mathbf{c}| \alpha_{i} \rangle = 0$.

\end{enumerate}

\vskip 2mm

The expressions \eqref{eq:real string} and \eqref{eq:imaginary string} are called the
{\it $i$-string decomposition} of $u$.
Note that (i) is equivalent to saying that
$A_{ik} u_{\mathbf{c}} = E_{ik} u_{\mathbf{c}} = 0$ for all  $k>0$.

\vskip 3mm



\vskip 2mm



\vskip 3mm

Given the $i$-string decompositions \eqref{eq:real string} and \eqref{eq:imaginary string},
we define the {\it Kashiwara operators} on $M$ as follows.

\vskip 3mm

\begin{definition} \label{def:Kashiwara operator} \hfill

\vskip 3mm

(a) For $i \in I^{\text{re}}$, we define
\begin{equation} \label{eq:real Kas}
\begin{aligned}
& \widetilde{e}_{i} u = \sum_{k \ge 1} {\mathtt b}_{i}^{(k-1)} u_{k}, \\
& \widetilde{f}_{i}  u = \sum_{k \ge 0} {\mathtt b}_{i}^{(k+1)} u_{k}.
\end{aligned}
\end{equation}

\vskip 2mm

(b) For $i \in I^{\text{im}}\setminus I^{\text{iso}}$ and $l>0$, we define
\begin{equation} \label{eq: nonisoKas}
\begin{aligned}
& \widetilde{e}_{il} u = \sum_{\mathbf{c} \in {\mathcal C}_{i}:  c_1 = l} \,
{\mathtt b}_{i, \mathbf{c} \setminus c_1} u_{\mathbf{c}},\\
& \widetilde{f}_{il}  u = \sum_{\mathbf{c} \in {\mathcal C}_{i}}
{\mathtt b}_{i,(l, \mathbf{c})} u_{\mathbf{c}}.
\end{aligned}
\end{equation}

\vskip 2mm

(c) For $i \in I^{\text{iso}}$ and $l>0$, we define
\begin{equation} \label{eq:isoKas}
\begin{aligned}
& \widetilde{e}_{il} u = \sum_{\mathbf{c} \in {\mathcal C}_{i}} \,  \mathbf{c}_{l} \,
{\mathtt b}_{i, \mathbf{c} \setminus \{l\}} u_{\mathbf{c}},\\
& \widetilde{f}_{il}  u = \sum_{\mathbf{c} \in {\mathcal C}_{i}} \, \frac{1} {\mathbf{c}_{l} + 1} \,
{\mathtt b}_{i,\{l\} \cup \mathbf{c}} u_{\mathbf{c}},
\end{aligned}
\end{equation}
where ${\mathbf c}_{l}$ denotes the number of $l$ in $\mathbf{c}$.

\end{definition}

\vskip 2mm

It is easy to see that $\widetilde{e}_{il}\circ\widetilde{f}_{il}=\text{id}_{M_{\mu}}$ for
$(i,l) \in I^{\infty}$ and $\langle h_i,\mu \rangle>0$.

\vskip 3mm

Let $\A_{0} = \{f \in \Q(q) \mid \text{$f$ is regular at $q=0$} \}$.
Then we have an isomorphism
$$\A_{0} / q \A_{0} \cong \Q, \quad f + q \A_{0} \longmapsto f(0).$$

\vskip 3mm

\begin{definition} \label{def:crystal lattice Llambda} \hfill

\vskip 2mm

Let $M$ be a $U_{q}(\g)$-module in the category ${\mathcal O}_{\text{int}}$
and let $L$ be a free $\A_{0}$-submodule of $M$.
The submodule $L$ is called a {\it crystal lattice} of $M$ if the following conditions hold.

\begin{enumerate}

\vskip 2mm

\item[(a)] $\Q \otimes_{\A_{0}} L \cong M$,

\vskip 2mm

\item[(b)] $L =\oplus_{\mu \in P} L_{\mu}$, where $L_{\mu} = L \cap M_{\mu}$,

\vskip 2mm

\item[(c)] $\widetilde{e}_{il} L \subset L$, $\widetilde{f}_{il} L \subset L$
for $(i, l) \in I^{\infty}$.

\end{enumerate}

\end{definition}

\vskip 3mm

Since the operators $\widetilde e_{il}$, $\widetilde f_{il}$ preserve $L$,
they induce the operators
$$\widetilde{e}_{il}, \, \widetilde{f}_{il}: L/qL \longrightarrow L/qL.$$

\vskip 3mm

\begin{definition} \label{def:crystal basis Blambda} \hfill

\vskip 2mm

Let $M$ be a $U_{q}(\g)$-module in the category ${\mathcal O}_{\text{int}}$.
A {\it crystal basis} of $M$ is a pair $(L, B)$ such that
\begin{enumerate}

\vskip 2mm

\item[(a)] $L$ is a crystal lattice of $M$,

\vskip 2mm

\item[(b)] $B$ is a $\Q$-basis of $L/qL$,

\vskip 2mm

\item[(c)] $B=\sqcup_{\mu \in P}\ B_\mu$, where
$B_\mu=B\cap{(L/qL)}_\mu$,

\vskip 2mm

\item[(d)] $\widetilde{e}_{il}B\subset B\cup\{0\}$, $\widetilde{f}_{il} B\subset B\cup\{0\}$
for $(i, l) \in I^{\infty}$,

\vskip 2mm

\item[(e)] for any $b,b'\in B$ and $(i,l)\in I^{\infty}$, we have
$\widetilde{f}_{il}b=b'$ if and only if $b=\widetilde{e}_{il}b'$.

\end{enumerate}
\end{definition}

\vskip 3mm

\begin{lemma}\label{euEuqL}
{\rm
Let $M$ be a $U_q(\g)$-module in the category ${\mathcal O}_{\text{\rm int}}$
and $(L, B)$ be a crystal basis of $M$.
For any $u\in M_{\mu}$, we have
\begin{equation*}
\widetilde{e}_{il}\,u \equiv E_{il}\,u \ \text{mod} \ qL
\ \ \text{for} \ (i,l) \in I^{\infty}.
\end{equation*}
}
\end{lemma}

\begin{proof}
Let $u=\mathtt b_{i,\mathbf c}u_0$ such that $E_{ik}u_0=0$ for any $k>0$. Let $m:=\langle h_i,\text{wt}(u_0)\rangle$.

\vskip 2mm

\begin{enumerate}
	\item[(a)] Suppose $i\notin I^{\text{iso}}$ and let $\mathbf c=(c_1,\cdots,c_r)\in\mathcal C_{i,l}$.
	
	\vskip 2mm
	
	\begin{enumerate}
	\item[(i)] If $c_1=l$, by \eqref{Eb}, we have
	\begin{align*}
		E_{il}(u)&=	E_{il}(\mathtt b_{i,\mathbf c}u_0)=E_{il}\mathtt b_{il}(\mathtt b_{i,\mathbf c'}u_0)\\
		&=(q_i^{l^2a_{ii}}\mathtt b_{il}E_{il}+\frac{1-K_i^{2l}}{1-q_i^{2l}})\mathtt b_{i,\mathtt c'}u_0\\
		&\equiv\mathtt b_{i,\mathbf c'}u_0\equiv\widetilde{e}_{il}u\!\!\!\mod{q L}.
	\end{align*}
	
\vskip 2mm

\item[(ii)] If $c_1=k\neq l$, we have
\begin{align*}
	E_{il}(u)&=E_{il}(\mathtt b_{i,\mathbf c}u_0)=E_{il}\mathtt b_{ik}(\mathtt b_{i,\mathbf c'}u_0)\\
	&=q_i^{-kla_{ii}}\mathtt b_{ik}E_{il}(\mathtt b_{i,\mathbf c'}u_0)\equiv 0 \equiv \widetilde{e}_{il}u\!\!\! \mod{q L}.
\end{align*}
\end{enumerate}	
	
\vskip 2mm

\item[(b)] If $i\in I^{\text{iso}}$, we have
\begin{align}
&\langle h_i,\text{wt}(u_0)-\alpha\rangle=m\ \text{for any}\ \alpha\in R_+,\notag\\
&E_{il}\mathtt b_{il}-\mathtt b_{il}E_{il}=\frac{1-K_i^{2l}}{1-q_i^{2l}},\label{Eb-bE}\\
&E_{il}\mathtt b_{ik}-\mathtt b_{ik}E_{il}=0 \ \text{if}\ k\neq l.\notag
\end{align}

\vskip 2mm

\begin{enumerate}

\vskip 2mm

	\item[(iii)] By induction on \eqref{Eb-bE}, one can prove:
	\begin{align*}
	E_{il}(\mathtt b_{il}^ku_0)=k\frac{1-q_i^{2lm}}{1-q_i^{2l}}\mathtt b_{il}^{k-1}u_0\equiv k\mathtt b_{il}^{k-1}u_0 \equiv \widetilde{e}_{il}(\mathtt b_{il}^ku_0)\!\!\!  \mod{q L}.
	\end{align*}
	
\vskip 2mm

    \item[(iv)] We may write
    \begin{align*}
    u=\mathtt b_{i,\mathbf c}u_0=\mathtt b_{ic_1}^{a_1}\mathtt b_{ic_2}^{a_2}\cdots\mathtt b_{il}^k\cdots\mathtt b_{ic_r}^{a_r}u_0,
    \end{align*}
where $c_1>c_2>\cdots>l>\cdots>c_r$.
Then we have
\begin{align*}
E_{il}u=\mathtt b_{ic_1}^{a_1}\cdots E_{il}(\mathtt b_{il}^k)\cdots\mathtt b_{ic_r}^{a_r}u_0.
\end{align*}

\vskip 2mm

Let $u'=\mathtt b_{ic_t}^{a_t}\cdots\mathtt b_{ic_r}^{a_r}u_0$.
By the same argument as that in (iii), we can show that
\begin{align*}
E_{il}(\mathtt b_{il}^ku')\equiv k\mathtt b_{il}^{k-1}u'\!\!\! \mod{q L}.
\end{align*}
Hence we have
\begin{align*}
E_{il}(u)= \mathbf c_l\mathtt b_{i,\mathbf c\setminus\{l\}}u_0\equiv \widetilde{e}_{il}(u)\!\!\! \mod{q L}.
\end{align*}
\end{enumerate}
\end{enumerate}	
\end{proof}

\vskip 3mm

Let $V(\lambda)=U_q(\mathfrak g)v_\lambda$ be the irreducible highest weight $U_q(\g)$-module
with highest weight $\lambda \in P^+$.
Let $L(\lambda)$ be the free $\A_0$-submodule of $V(\lambda)$
spanned by $\widetilde{f}_{i_1l_1}\cdots\widetilde{f}_{i_rl_r} v_\lambda$ $(r\geq 0, (i_k,l_k)\in I^{\infty})$
and let
\begin{align*}
B(\lambda):=\{\widetilde{f}_{i_1l_1}\cdots\widetilde{f}_{i_rl_r}v_\lambda+qL(\lambda)\}\setminus\{0\}.
\end{align*}

\vskip 3mm

\begin{theorem}\label{thm:crystal basis of V}

{\rm
The pair $(L(\lambda),B(\lambda))$ is a crystal basis of $V(\lambda)$.
}
\end{theorem}

\vskip 2mm

We will prove this theorem in Section \ref{sec:grand-loop}.

\vskip 3mm

\begin{example}\label{ex: basis for Vi}
{\rm
Let $I=I^{\text{\rm im}}=\{i\}$ and
\begin{align*}
	U=\Q(q)\langle \mathtt a_{il},~\mathtt b_{il},~K_i^{\pm l}\mid l>0    \rangle
	=\Q(q)\langle E_{il},~\mathtt b_{il},~K_i^{\pm l}\mid l>0 \rangle.
\end{align*}

\vskip 2mm

Let $V=\bigoplus_{\mathbf c\in\mathcal C_i}{\Q(q)\mathtt b_{i,\mathbf c}u_0}$ such that
\begin{align*}
	V=U u_0,\quad
	\langle  h_i,\mathrm{wt}(u_0)\rangle =m,\quad
	K_i^{\pm l}u_0=q_{i}^{\pm lm}u_0,\quad E_{ik}u_0=0~ \text{for~any~}k>0,
\end{align*}
and $L=\bigoplus_{\mathbf c\in\mathcal C_i}{\A_0(\mathtt b_{i,\mathbf c}u_0)}$.

\vskip 3mm

If $i \in I^{\text{im}} \setminus I^{\text{\rm iso}}$, for $\mathbf c\in\mathcal C_i$,
let $B_{i,\mathbf c}=\{\mathtt b_{i,\mathbf c}u_0\}$ and $B=\coprod_{\mathbf c\in\mathcal C_i}{B_{i,\mathbf c}}$.
Define
\begin{align*}
	&\widetilde{e}_{il}(\mathtt b_{i,\mathbf c}u_0)=	\begin{cases}	
		\mathtt b_{i,\mathbf c\backslash c_1}u_0,&\text{if}~c_1=l,\\
		0,&\text{otherwise},\\
	\end{cases}\\
	&\widetilde{f}_{il}(\mathtt b_{i,\mathbf c}u_0)=\mathtt b_{i,(l,\mathbf c)}u_0.
\end{align*}

If $i\in I^{\text{\rm iso}}$, for $\mathbf c\in \mathcal C_i$, let  $B_{i,\mathbf c}
=\{ \frac{1} {{\mathbf c}_{l} !} \, \mathtt b_{i,\mathbf c}u_0\}$
and set $B=\coprod_{\mathbf c\in\mathcal C_i}{B_{i,\mathbf c}}$.
Define
\begin{align*}
	&\widetilde{e}_{il}\, (\frac{1} {{\mathbf c}_{l} !} \, \mathtt b_{i,\mathbf c} u_0) \
	=\frac{1}{({\mathbf c}_{l}-1)!} \, \mathtt b_{i,\mathbf c\backslash\{l\}}u_0,\\
	&\widetilde{f}_{il} \, (\frac{1}{{\mathbf c}_{l} !} \, \mathtt b_{i,\mathbf c}u_0)
	\ =\frac{1} {({\mathbf c}_{l} + 1)!} \, \mathtt b_{i,\mathbf c\cup\{l\}}u_0.
\end{align*}

We can verify that the pair $(L,B)$ is a crystal basis of $V$.
}
	
\end{example}


\vskip 5mm


\subsection{Crystal bases for $U_{q}^{-}(\g)$} \label{sub:Uminus} \hfill

\vskip 3mm

Now we will discuss the crystal basis for $U_{q}^-(\g)$.

\vskip 3mm

Let $(i, l) \in I^{\infty}$ and $S\in U_{q}^-(\g)$.
Then there exist unique elements $T, W \in U_{q}^-(\g)$ such that
\begin{equation*}
\mathtt a_{il}S-S\mathtt a_{il}=\frac{K_i^{l}T-K_i^{-l}W}{1-q_i^{2l}}.
\end{equation*}

\vskip 2mm

\noindent
Equivalently, there are uniquely determined elements $T, W \in U_{q}^-(\g)$ such that
\begin{equation}\label{eq:AP-PA}
	A_{il}S-SA_{il}=\frac{K_i^lT-K_i^{-l}W}{q_i^l-q_i^{-l}}.
\end{equation}

\vskip 3mm

We  define the operators $e'_{il}, e''_{il}:U_{q}^-(\g) \longrightarrow U_{q}^-(\g)$ by
\begin{equation}\label{eq:eP=R}
e'_{il}(S)=W,\quad e''_{il}(S)=T.
\end{equation}

\vskip 3mm

\noindent
By \eqref{eq:AP-PA} and \eqref{eq:eP=R}, we have
\begin{equation}\label{eq:AP-PA=}
	A_{il}S-SA_{il}=\frac{K_i^l(e''_{il}(S))-K_i^{-l}(e'_{il}(S))}{q_i^l-q_i^{-l}}.
\end{equation}

\vskip 3mm

Therefore we obtain
\begin{equation}\label{eq:commute}
\begin{aligned}
&e'_{il}\mathtt b_{jk}=\delta_{ij}\delta_{kl}+q_i^{-kla_{ij}}\mathtt b_{jk}e'_{il},\\
&e''_{il}\mathtt b_{jk}=\delta_{ij}\delta_{kl}+q_i^{kla_{ij}}\mathtt b_{jk}e''_{il},\\
&e'_{il}e''_{jk}=q_i^{kla_{ij}}e''_{jk}e'_{il}.
\end{aligned}
\end{equation}

\vskip 3mm

Let $*:U_q(\g)\rightarrow U_q(\g)$ be the $\Q(q)$-linear  anti-involution given by

\begin{align}\label{*}
(q^h)^*=q^{-h},\quad \mathtt a^*_{il}=\mathtt a_{il},
\quad \mathtt b_{il}^*=\mathtt b_{il}.
\end{align}

By \eqref{eq:bar} and Proposition \ref{prop:primitive}, we have $**={\rm id}$, $--={\rm id}$ and $*-=-*$.

\vskip 3mm

By \eqref{eq:AP-PA=}, we have
\begin{equation}
	S^*A_{il}-A_{il}S^*=\frac{{(e''_{il}(S))}^*K_i^{-l}-{(e'_{il}(S))}^*K_i^{l}}{q_i^l-q_i^{-l}}.
\end{equation}

\vskip 2mm

Therefore we obtain
\begin{equation}\label{eq:ePstar}
e'_{il}(S^*)=K_i^l{(e''_{il}S)}^*K_i^{-l},\quad e''_{il}(S^*)=K_i^{-l}{(e'_{il}S)}^*K_i^{l}.
\end{equation}


\vskip 3mm

Let $u \in U_{q}^-(\g)_{-\alpha}$ with $\alpha \in R_{+}$.
For $i \in I^{\text{re}}$, by \cite{Kashi91}, the vector $u$ can be written uniquely as
\begin{equation} \label{eq:real string U}
u = \sum_{k \ge 0} {\mathtt b}_{i}^{(k)} u_k
\end{equation}
such that

\vskip  2mm

\begin{enumerate}

\item[(i)] $e_{i}'u_k = 0$ for all $k \ge 0$,

\vskip 2mm

\item[(ii)] $u_k \in U_{q}^-(\g)_{-\alpha + k \alpha_{i}}$,

\vskip 2mm

\item[(iii)] $u_{k} = 0$ if $\langle h_{i}, -\alpha + k \alpha_{i} \rangle = 0 $.

\end{enumerate}

\vskip 3mm

For $i \in I^{\text{im}}$, by \cite{Bozec2014b, Bozec2014c}, the vector
$u$ can be written uniquely as
\begin{equation} \label{eq:imaginary string U}
u = \sum_{\mathbf c\in \mathcal C_i} {\mathtt b}_{i, \mathbf{c}} u_{\mathbf{c}}
\end{equation}
such that

\vskip 2mm

\begin{enumerate}

\item[(i)] $e_{ik}' u_{\mathbf{c}} = 0$ for all $k >0$,

\vskip 2mm

\item[(ii)] $u_{\mathbf{c}} \in U_{q}^-(\g)_{- \alpha+ |\mathbf{c}| \alpha_{i}}$,

\vskip 2mm

\item[(iii)] $u_{\mathbf{c}} = 0$ if $\langle h_{i}, - \alpha+ |\mathbf{c}| \alpha_{i} \rangle = 0$.

\end{enumerate}

\vskip 2mm

The expressions \eqref{eq:real string U} and \eqref{eq:imaginary string U} are called the
{\it $i$-string decomposition} of $u$.

\vskip 3mm

Given the $i$-string decompositions \eqref{eq:real string U} and \eqref{eq:imaginary string U},
we define the {\it Kashiwara operators} on $U_{q}^-(\g)$ as follows.

\vskip 3mm

\begin{definition} \label{def:Kashiwara operators} \hfill

\vskip 2mm

(a) For $i \in I^{\text{re}}$, we define
\begin{equation} \label{eq:real Kas U}
\begin{aligned}
& \widetilde{e}_{i} u = \sum_{k \ge 1} {\mathtt b}_{i}^{(k-1)} u_{k}, \\
& \widetilde{f}_{i}  u = \sum_{k \ge 0} {\mathtt b}_{i}^{(k+1)} u_{k}.
\end{aligned}
\end{equation}

\vskip 2mm

(b) For $i \in I^{\text{im}}\setminus I^{\text{iso}}$ and $l>0$, we define
\begin{equation} \label{eq: nonisoKas U}
\begin{aligned}
& \widetilde{e}_{il} u = \sum_{\mathbf{c} \in {\mathcal C}_{i}:  c_1 = l} \,
{\mathtt b}_{i, \mathbf{c} \setminus c_1} u_{\mathbf{c}},\\
& \widetilde{f}_{il}  u = \sum_{\mathbf{c} \in {\mathcal C}_{i}}
{\mathtt b}_{i,(l, \mathbf{c})} u_{\mathbf{c}}.
\end{aligned}
\end{equation}

\vskip 2mm

(c) For $i \in I^{\text{iso}}$ and $l>0$, we define
\begin{equation} \label{eq:isoKas U}
\begin{aligned}
& \widetilde{e}_{il} u = \sum_{\mathbf{c} \in {\mathcal C}_{i}} \, \mathbf{c}_{l} \,
{\mathtt b}_{i, \mathbf{c} \setminus \{l\}} u_{\mathbf{c}},\\
& \widetilde{f}_{il}  u = \sum_{\mathbf{c} \in {\mathcal C}_{i}} \, \frac{1}{ \mathbf{c}_{l} + 1} \,
{\mathtt b}_{i,\{l\} \cup \mathbf{c}} u_{\mathbf{c}},
\end{aligned}
\end{equation}
where ${\mathbf c}_{l}$ denotes the number of $l$ in $\mathbf{c}$.

\end{definition}

\vskip 2mm

It is easy to see that $\widetilde{e}_{il}\circ\widetilde{f}_{il}=\text{id}_{U_{q}^-(\g)_{-\alpha}}$ for
$(i,l) \in I^{\infty}$ and $\langle h_i, - \alpha \rangle>0$.

\vskip 3mm

\begin{definition} \label{def:crystal lattice U}

\vskip 2mm

 A free $\A_0$-submodule $ L$ of $U_{q}^-(\g)$ is called a {\it crystal lattice} if the following conditions hold.

 \vskip 2mm

\begin{enumerate}

\vskip 2mm

\item[{\rm (a)}] $\Q(q)\otimes_{\mathbf A_0} L\cong U_q^-(\g)$,

\vskip 2mm

\item[{\rm (b)}] $ L=\oplus_{\alpha\in R_+} L_{-\alpha}$, where $ L_{-\alpha}= L\cap {U_{q}^-(\g)}_{-\alpha}$,

\vskip 2mm

\item[{\rm (c)}] $\widetilde e_{il} L\subset  L$, \ $\widetilde f_{il} L\subset  L$ for all $(i,l)\in I^{\infty}$.

\end{enumerate}

\end{definition}

\vskip 3mm

The condition (c) yields  the $\Q$-linear maps
$${\widetilde e}_{il}, \, {\widetilde f}_{il}: L/qL \longrightarrow L/qL.$$

\vskip 3mm

\begin{definition}

A {\it crystal basis} of $U_{q}^-(\g)$ is a pair $( L, B)$ such that
	
\vskip 2mm
	
\begin{enumerate}
	
\vskip 2mm
	
\item[{\rm (a)}] $L$ is a crystal lattice of $U_{q}^-(\g)$,
		
\vskip 2mm

\item[{\rm (b)}] $B$ is a $\Q$-basis of $L/q L$,

\vskip 2mm

\item[{\rm (c)}] $ B=\sqcup_{\alpha \in R_+} B_{-\alpha}$, where
$ B_{-\alpha}= B\cap{( L/q L)}_{-\alpha}$,

\vskip 2mm

\item[{\rm (d)}] $\widetilde{e}_{il} B\subset B\cup\{0\}$, \ $\widetilde{f}_{il} B \subset B\cup\{0\}$ 
for $(i,l) \in I^{\infty}$,

\vskip 2mm

\item[{\rm (e)}] for any $b,b'\in B$ and $(i,l) \in I^{\infty}$,
we have $\widetilde{f}_{il}b=b'$ if and only if $b=\widetilde{e}_{il}b'$.

\end{enumerate}

\end{definition}

\vskip 3mm

Let $L(\infty)$ be the $\A_0$-submodule of $U_{q}^-(\g)$ spanned by $\widetilde{f}_{i_1l_1}\cdots\widetilde{f}_{i_rl_r}\mathbf 1$ $(r\geq 0, (i_j,l_j)\in I_{\infty})$,
and $B(\infty)=\{\widetilde{f}_{i_1l_1}\cdots\widetilde{f}_{i_rl_r}\mathbf 1+qL(\infty)\}$.

\vskip 3mm

\begin{theorem}\label{thm:crystal basis of U}
{\rm
The pair  $(L(\infty),B(\infty))$ is a crystal basis of  $U^-_q(\mathfrak g)$.
}
\end{theorem}

\vskip 3mm

We will prove this theorem in Section \ref{sec:grand-loop}.

\vskip 3mm

\begin{example}\label{ex: basis for Ui}
{\rm
Let $I=I^{\text{\rm im}}=\{i\}$ and let
	\begin{align*}
		U^-=\mathbf Q(q)\langle \mathtt b_{il}\mid l>0    \rangle,\quad L:=\bigoplus_{\mathbf c\in\mathcal C_i}{\mathbf A_0(\mathtt b_{i,\mathbf c}\mathbf 1)}.
	\end{align*}

	If $i\notin I_{\text{\rm iso}}$, for $\mathbf c\in\mathcal C_i$, define $B_{i,\mathbf c}:=\{\mathtt b_{i,\mathbf c}\mathbf 1\}$ and set $B=\coprod_{\mathbf c\in\mathcal C_i}{B_{i,\mathbf c}}$. Define
	\begin{align*}
		&\widetilde{e}_{il}(\mathtt b_{i,\mathbf c}\mathbf 1)=	\begin{cases}	
			\mathtt b_{i,\mathbf c\backslash c_1}\mathbf 1,&\text{if}~c_1=l,\\
			0,&\text{otherwise},\\
		\end{cases}\\
		&\widetilde{f}_{il}(\mathtt b_{i,\mathbf c}\mathbf 1)=\mathtt b_{i,(l,\mathbf c)}\mathbf 1.
	\end{align*}

	If $i\in I_{\text{\rm iso}}$, for $\mathbf c\in \mathcal C_i$, define
	$B_{i,\mathbf c}:=\{ \frac{1} {{\mathbf c}_{l} ! } \,\mathtt b_{i,\mathbf c}\mathbf 1\}$ and set $B=\coprod_{\mathbf c\in\mathcal C_i}{B_{i,\mathbf c}}$. Define
	\begin{align*}
		&\widetilde{e}_{il} \, ( \frac{1} {{\mathbf c}_{l} !} \,\mathtt b_{i,\mathbf c}\mathbf 1)
		= \frac{1} {(\mathbf c_l-1)!} \, \mathtt b_{i,\mathbf c\backslash\{l\}}\mathbf 1,\\
		&\widetilde{f}_{il} \, (\frac{1}{{\mathbf c}_{l} !}  \, \mathtt b_{i,\mathbf c}\mathbf 1)
		= \frac{1}{(\mathbf c_l+1)!} \, \mathtt b_{i,\mathbf c\cup\{l\}}\mathbf 1.
	\end{align*}

\noindent	
	We can verify that the pair $(L,B)$ is a crystal basis of $U^-$.
	
}
\end{example}


\vskip 5mm

\subsection{Abstract crystals} \label{sub:abstract crystals} \hfill

\vskip 3mm

By extracting the fundamental properties of the crystal bases of $V(\lambda)$
and $U_{q}^-(\g)$, we define the notion of abstract crystals as follows.

\vskip 3mm

\begin{definition}\cite[Definition 2.1]{FKKT21}\label{def: abstract crystal} \hfill

\vskip 2mm

	An {\it abstract crystal}  is a set $B$ together with the maps ${\rm wt}\colon  B \rightarrow P$,
	$\varphi_i,\varepsilon_i\colon  B\rightarrow \Z\cup \{-\infty\}$ $(i\in I)$ and $\widetilde{e}_{il},\widetilde{f}_{il}\colon  B\rightarrow   B\cup \{0\}$ $((i,l)\in I^\infty)$ satisfying the following conditions:
	\vskip 2mm
	\begin{enumerate}
		\item[{\rm (a)}] $\text{wt}(\widetilde{f}_{il}b)=\text{wt}(b)-l\alpha_i$ if $\widetilde{f}_{il}b\neq 0$,\quad $\text{wt}(\widetilde{e}_{il}b)=\text{wt}(b)+l\alpha_i$ if $\widetilde{e}_{il}b\neq 0$.
		
		\vskip 2mm
		
		\item[{\rm (b)}] $\varphi_i(b)= \langle h_i, \text{wt}(b) \rangle+\varepsilon_i(b)$ for $i\in I$ and $b\in  B$.
		
		\vskip 2mm
		
		\item[{\rm (c)}] $\widetilde{f}_{il}b=b'$ if and only if $b=\widetilde{e}_{il}b'$ for $(i,l)\in I^\infty$ and $b,b'\in  B$.
		
		\vskip 2mm
		
		\item[{\rm (d)}]  For any $i\in I^{\text{re}}$ and $b\in  B$, we have
		\vskip 2mm
		\begin{itemize}
			\item [(1)] $\varepsilon_i(\widetilde{f}_ib)=\varepsilon_i(b)+1$, $\varphi_i(\widetilde{f}_ib)=\varphi_i(b)-1$ if $\widetilde{f}_ib\neq 0$,
			\vskip 2mm
			\item [(2)] $\varepsilon_i(\widetilde{e}_ib)=\varepsilon_i(b)-1$, $\varphi_i(\widetilde{e}_ib)=\varphi_i(b)+1$ if $\widetilde{e}_ib\neq 0$.
		\end{itemize}
		
		\vskip 2mm
		
		\item[{\rm (e)}] For any $i\in I^{\text{im}}$, $l>0$ and $b\in B$, we have
		\vskip 2mm
		\begin{itemize}
			\item [($1'$)] $\varepsilon_i(\widetilde{f}_{il}b)=\varepsilon_i(b)$, $\varphi_i(\widetilde{f}_{il}b)=\varphi_i(b)-la_{ii}$ if $\widetilde{f}_{il}b\neq 0$,
			\vskip 2mm
			\item [($2'$)] $\varepsilon_i(\widetilde{e}_{il}b)=\varepsilon_i(b)$, $\varphi_i(\widetilde{e}_{il}b)=\varphi_i(b)+la_{ii}$ if $\widetilde{e}_{il}b\neq 0$.
		\end{itemize}
		
		\vskip 2mm
		
		\item[{\rm (f)}]  For any $(i,l)\in I^\infty$ and $b\in B$ such that $\varphi_i(b)=-\infty$, we have $\widetilde{e}_{il}b=\widetilde{f}_{il}b=0$.
	\end{enumerate}
\end{definition}

\vskip 2mm

\begin{remark} \hfill

\vskip 2mm

(a) In Example \ref{ex: basis for Vi},  define
\begin{equation*}
\begin{aligned}
& \mathrm{wt}(\mathtt b_{i,\mathbf c}u_0)=\mathrm{wt}(u_0)-|\mathbf c|\alpha_i,\quad
	\varepsilon_i(\mathtt b_{i,\mathbf c}u_0)=0, \\
&\varphi_i(\mathtt b_{i,\mathbf c}u_0)=
	\langle h_i,\mathrm{wt}(\mathtt b_{i,\mathbf c}u_0)   \rangle
	=
	\langle  h_i,\mathrm{wt}(u_0)-|\mathbf c|\alpha_i \rangle=m-|\mathbf c|a_{ii}.
\end{aligned}
\end{equation*}

Then the set $B$ together with the maps $\widetilde{e}_{il}$, $\widetilde{f}_{il}$, $\mathrm{wt}$, $\varepsilon_i$, $\varphi_i$  is an abstract crystal.

\vskip 3mm

(b) In Example \ref{ex: basis for Ui},  define
\begin{equation*}
 \mathrm{wt}(\mathtt b_{i,\mathbf c}\mathbf 1)=-|\mathbf c|\alpha_i,\quad
\varepsilon_i(\mathtt b_{i,\mathbf c}\mathbf 1)=0,\quad
\varphi_i(\mathtt b_{i,\mathbf c}\mathbf 1)=-|\mathbf c|a_{ii}.
\end{equation*}

Then the set $B$ together with the maps $\widetilde{e}_{il}$, $\widetilde{f}_{il}$, $\mathrm{wt}$, $\varepsilon_i$, $\varphi_i$  is an abstract crystal.

\end{remark}

\vskip 3mm

\begin{definition} \hfill

\vskip 2mm

	\begin{enumerate}
		\item[{\rm (a)}] A {\it crystal morphism} $\psi$ between two abstract crystals $ B_1$ and $ B_2$ is a map from $ B_1$ to
		$ B_2\sqcup\{0\}$ satisfying the following conditions:
		
		\vskip 2mm
		
		\begin{enumerate}
			\item[{\rm (i)}] for $b\in  B_1$ and $i\in I$, we have $\text{wt}(\psi(b))=\text{wt}(b)$,  $\varepsilon_i(\psi(b))=\varepsilon_i(b)$, $\varphi_i(\psi(b))=\varphi_i(b)$,
			\item[{\rm (ii)}] for $b\in B_1$ and $(i,l)\in I^\infty$ satisfying $\widetilde{f}_{il}b\in B_1$, we have $\psi(\widetilde{f}_{il}b)=\widetilde{f}_{il}\psi(b)$.
		\end{enumerate}	
		
		\vskip 2mm
		
		\item[{\rm (b)}] A crystal morphism $\psi: B_1\to  B_2$ is called {\it strict} if
		\begin{equation*}
			\psi(\widetilde{e}_{il} b) = \widetilde{e}_{il}(\psi(b)), \  \psi(\widetilde{f}_{il} b) = \widetilde{f}_{il}(\psi(b))	
		\end{equation*}	
		for all $(i,l)\in I^\infty$ and $b \in  B_1$.	
	\end{enumerate}	
\end{definition}

\vskip 3mm

We recall  the {\it tensor product rule} from \cite[Section 3]{FKKT21}.
Let $ B_1$ and $ B_2$ be abstract crystals and let
$ B_1 \otimes B_2 = \{ b_1 \otimes b_2 \mid b_1 \in B_1, b_2 \in B_2 \}$.  Define the maps $\text{wt}$, $\varepsilon_i$, $\varphi_i$ $(i \in I)$, $\widetilde{e}_{il}$, $\widetilde{f}_{il}$ $((i,l) \in I^{\infty})$ as follows.
\begin{equation} \label{eq:wt}
	\begin{aligned}
		& \text{wt}(b_1\otimes b_2)=\text{wt}(b_1)+\text{wt}(b_2),\\
		& \varepsilon_i(b_1\otimes b_2)=\text{max}(\varepsilon_i(b_1),\varepsilon_i(b_2)- \langle h_i, \text{wt} (b_1) \rangle),\\
		& \varphi_i(b_1\otimes b_2)=\text{max}(\varphi_i(b_1)+\langle h_i, \text{wt}(b_2) \rangle,\varphi_i(b_2)).
	\end{aligned}
\end{equation}

If $i \in I^{\text{re}}$,
\begin{equation} \label{eq:real}
	\begin{aligned}
		& \widetilde{e}_{i}(b_1 \otimes b_2) = \begin{cases} \widetilde{e}_{i} b_1 \otimes b_2  & \text{if} \  \varphi_{i}(b_1) \ge \varepsilon_{i}(b_2), \\
			b_1 \otimes \widetilde{e}_{i} b_2  & \text{if} \ \varphi_{i}(b_1) < \varepsilon_{i}(b_2),
		\end{cases} \\
		& \widetilde{f}_{i}(b_1 \otimes b_2) = \begin{cases} \widetilde{f}_{i} b_1 \otimes b_2  & \text{if} \ \varphi_{i}(b_1) > \varepsilon_{i}(b_2), \\
			b_1 \otimes \widetilde{f}_{i} b_2  & \text{if} \ \varphi_{i}(b_1) \le \varepsilon_{i}(b_2).
		\end{cases}
	\end{aligned}
\end{equation}

If $i \in I^{\text{im}}$,
\begin{equation} \label{eq:im}
	\begin{aligned}
		& \widetilde{e}_{il}(b_1 \otimes b_2) =  \begin{cases} \widetilde{e}_{il} b_1 \otimes b_2  & \text{if} \ \varphi_{i}(b_1) > \varepsilon_{i}(b_2) - l a_{ii}, \\
			0  & \text{if} \ \varepsilon_{i}(b_{2}) < \varphi_{i}(b_{1}) \le \varepsilon_{i}(b_{2}) - l a_{ii}, \\
			b_{1} \otimes \widetilde{e}_{il} b_2 \ \ & \text{if} \ \varphi_{i}(b_1) \le \varepsilon_{i}(b_{2}),
		\end{cases}\\
		& \widetilde{f}_{il}(b_1 \otimes  b_2) =\begin{cases} \widetilde{f}_{il} b_1 \otimes b_2  & \text{if} \ \varphi_{i}(b_1) > \varepsilon_{i}(b_2), \\
			b_{1} \otimes \widetilde{f}_{il} b_2  & \text{if} \ \varphi_{i}(b_1) \le \varepsilon_{i}(b_2).
		\end{cases}
	\end{aligned}
\end{equation}

\vskip 3mm

\begin{proposition}\cite[Proposition 3.1]{FKKT21}\label{tensor product} \hfill
\vskip 2mm
{\rm
If $B_1$ and $B_2$ are abstract crystals, then $B_1\otimes B_2$ defined in \eqref{eq:wt}--\eqref{eq:im}	is also an abstract crystal.}
	\end{proposition}

\vskip 3mm

From now on, we shall only consider the case with $i\in I^{\text{im}}$,
because the case with $i \in I^{\text{re}}$ has already been
studied in \cite{Kashi91}.

\vskip 3mm

Let $M$ be an object in ${\mathcal O}_{\text{\rm int}}$ and let $( L, B)$ be a crystal basis of $M$.
We already have the maps
\begin{align}\label{def: wt and ef}
\text{wt}: B\rightarrow P,\quad \widetilde{e}_{il}, \widetilde{f}_{il}: B\rightarrow B\cup\{0\}.
\end{align}

 Define
\begin{align}\label{def:epsilon and phi}
	\varepsilon_i(b)=0,\quad \varphi_i(b)=\langle h_i,\text{wt}(b)\rangle\ \text{for any}\ b\in B.
\end{align}

\begin{lemma} {\rm
The set $B$ together with the maps defined in \eqref{def: wt and ef}--\eqref{def:epsilon and phi} is an abstract crystal.}
\end{lemma}
\begin{proof}
By Definition \ref{def:Kashiwara operator} and \eqref{def:epsilon and phi}, we have
\begin{align*}
\varepsilon_i(\widetilde{e}_{il}b)=\varepsilon_i(b)=0\ \text{and} \ \varepsilon_i(\widetilde{f}_{il}b)=\varepsilon_i(b)=0,
\end{align*}
and
\begin{align*}
&\varphi_i(\widetilde{f}_{il}b)=\langle h_i,\text{wt}(\widetilde{f}_{il}b)\rangle=\langle h_i,\text{wt}(b)-l\alpha_i\rangle=\varphi_i(b)-la_{ii},\\
&\varphi_i(\widetilde{e}_{il}b)=\langle h_i,\text{wt}(\widetilde{e}_{il}b)\rangle=\langle h_i,\text{wt}(b)+l\alpha_i\rangle=\varphi_i(b)+la_{ii}.
\end{align*}
Thus our assertion  follows.
\end{proof}

\vskip 3mm

Let $M_1,M_2\in{\mathcal O}_{\text{\rm int}}$ and $( L_1, B_1)$, $( L_2, B_2)$  be their crystal bases, respectively. Set
\begin{equation*}
M=M_1\otimes_{\mathbf Q(q)}M_2,\quad  L= L_1\otimes_{\mathbf A_0} L_2,\quad  B= B_1\otimes B_2.
\end{equation*}

\vskip 2mm

By Proposition \ref{tensor product}, $ B_1\otimes B_2$ is an abstract crystal. The tensor product rule on $ B_1\otimes B_2$ can be simplified as follows.

\vskip 3mm

Set $m_1:=\langle h_i,\text{wt}(b_1)\rangle$ and $m_2:=\langle h_i,\text{wt}(b_2)\rangle$. Then we have
\begin{equation}\label{eq:simplified}
\begin{aligned}
&\text{wt}(b_1\otimes b_2)=\text{wt}(b_1)+\text{wt}(b_2),\\
&\varepsilon_i(b_1\otimes b_2)=0,\quad \varphi_i(b_1\otimes b_2)=m_1+m_2,\\
&\widetilde{f}_{il}(b_1\otimes b_2)=
        \begin{cases}
	\widetilde{f}_{il}b_1\otimes b_2, &\text{if}\ m_1>0,\\
	b_1\otimes\widetilde{f}_{il} b_2, &\text{if}\ m_1=0,
	\end{cases}\\
&\widetilde{e}_{il}(b_1\otimes b_2)=
        \begin{cases}
	\widetilde{e}_{il}b_1\otimes b_2, &\text{if}\ m_1>-la_{ii},\\
	0\ &\text{if}\ 0<m_1\leq-la_{ii},\\
	b_1\otimes\widetilde{e}_{il} b_2,&\text{if}\ m_1=0.
\end{cases}
\end{aligned}
\end{equation}

Note that $m_1\geq 0$ because $\mathrm{wt}(b_1)\in {P}^+$.

\vskip 3mm

Let $V$, $V'$ be $U$-modules as in Example \ref{ex: basis for Vi} and
 let $(L, B)$, $(L', B')$ be their crystal bases, respectively.
Then $B \otimes B'$ is an abstract crystal under the simplified tensor product rule
given in \eqref{eq:simplified}.

\vskip 3mm

\section{Grand-loop argument}\label{sec:grand-loop}

In this section, we will give the proofs of Theorem \ref{thm:crystal basis of V}
and Theorem \ref{thm:crystal basis of U}
following the frame work of Kashiwara's grand-loop argument \cite{JKK05, Kashi91}.
For this purpose, we need to prove the statements given below.
\begin{equation} \label{eq:cond1}
\begin{aligned}
	&\widetilde{e}_{il}L(\lambda)\subset L(\lambda),\quad
	\widetilde{e}_{il}B(\lambda)\subset B(\lambda)\cup\{0\},\\
	&\widetilde{f}_{il}b=b'\ \text{if and only if}\ \widetilde{e}_{il}b'=b\ \text{for any}\ b, b'\in B(\lambda),\\
	&B(\lambda)\ \text{is a $\Q$-basis of}\ L(\lambda)/qL(\lambda),
\end{aligned}
\end{equation}
and
\begin{equation} \label{eq:cond2}
\begin{aligned}
	&\widetilde{e}_{il}L(\infty)\subset L(\infty),\quad
	\widetilde{e}_{il}B(\infty)\subset B(\infty)\cup\{0\},\\
	&\widetilde{f}_{il}b=b'\ \text{if and only if}\ \widetilde{e}_{il}b'=b\ \text{for any}\ b, b'\in B(\infty),\\
	&B(\infty)\ \text{is a $\Q$-basis of}\ L(\infty)/qL(\infty).
\end{aligned}
\end{equation}

\vskip 3mm

To apply the grand-loop argument, we need Kashiwara's
bilinear forms $(\ , \ )_{K}$ defined as follows.

\vskip 3mm

Let $V(\lambda)=U_{q}(\g) v_{\lambda}$ be an irreducible highest weight module
with $\lambda \in P^{+}$.
By a standard argument, one can show that there exists a unique non-degenerate symmetric
bilinear form $(\ , \ )_{K}$ on $V(\lambda)$ given by
\begin{equation} \label{eq:bilinearV}
\begin{aligned}
& (v_{\lambda}, v_{\lambda})_{K} = 1, \ \
(q^{h} u, v)_{K} = (u , q^{h} v)_{K},\\
& (\mathtt{b}_{il} u, v)_{K} = - (u, K_{i}^{l} \mathtt{a}_{il}v)_{K}, \\
& (\mathtt{a}_{il} u, v)_{K} = - (u, K_{i}^{-l} \mathtt{b}_{il} v)_{K},
\end{aligned}
\end{equation}
where $u,v\in V(\lambda)$ and $h \in P^{\vee}$.

\vskip 3mm

Similarly, there exists a unique
non-degenerate symmetric bilinear form $( \ , \ )_K$ on $U_{q}^-(\g)$ satisfying

\begin{equation} \label{eq:bilinearU}
(\mathbf{1}, \mathbf{1})_{K} = 1, \ \
(\mathtt{b}_{il}S, T)_{K} = (S, e_{il}'T)_{K} \ \ \text{for} \ S, T \in U_{q}^{-}(\g).
\end{equation}

\vskip 3mm

Now we begin to follow the grand-loop argument.

\vskip 3mm

For $\lambda \in P^+$, we define a $U^-_q(\g)$-module homomorphism given by
\begin{align}\label{eq:pilambda}
\pi_\lambda\colon U^-_q(\g)\rightarrow V(\lambda),\quad \mathbf 1\mapsto v_\lambda.
\end{align}
Then we obtain $\pi_\lambda(L(\infty))=L(\lambda)$.
 The map $\pi_\lambda$ induces a homomorphism
\begin{align}\label{eq:barpilambda}
\overline{\pi}_\lambda\colon L(\infty)/qL(\infty)\rightarrow L(\lambda)/qL(\lambda),\quad \mathbf 1+qL(\infty)\mapsto v_\lambda+qL(\lambda).
\end{align}

For $\lambda,\mu\in P^+$, there exist  unique $U_q(\mathfrak g)$-module homomorphisms
\begin{align*}
&\Phi_{\lambda,\mu}\colon V(\lambda+\mu)\rightarrow V(\lambda)\otimes V(\mu),\quad v_{\lambda+\mu}\mapsto v_\lambda\otimes v_\mu,\\
&\Psi_{\lambda,\mu}\colon V(\lambda)\otimes V(\mu)\rightarrow V(\lambda+\mu),\quad v_\lambda\otimes v_\mu\mapsto v_{\lambda+\mu}.
\end{align*}
It is easy to verify that $\Psi_{\lambda,\mu}\circ\Phi_{\lambda,\mu}={\rm id}_{V(\lambda+\mu)}$.

\vskip 2mm

On $V(\lambda)\otimes V(\mu)$, we define
\begin{align*}
(u_1\otimes u_2,v_1\otimes v_2)_K=(u_1,v_1)_K(u_2,v_2)_K,
\end{align*}
where $(\ , \ )_{K}$ is the non-degenerate symmetric bilinear form defined in \eqref{eq:bilinearV}.
It is straightforward to verify that
\begin{equation*}
(\Psi_{\lambda,\mu}(u),v)_K=(u,\Phi_{\lambda,\mu}(v))_K \ \ \text{for} \ u\in V(\lambda)\otimes V(\mu),
\ v \in V(\lambda + \mu).
\end{equation*}

\vskip 2mm

We now prove Theorem \ref{thm:crystal basis of V} and Theorem \ref{thm:crystal basis of U} using
Kashiwara's grand-loop argument as follows.

\vskip 2mm

Let $(i,l) \in I^{\infty}$, $\lambda,\mu\in P^+$ and  $\alpha\in  R_+(r)$, where
$R_+(r)=\{\alpha\in R_+\mid |\alpha|\leq r \}$.
\vskip 2mm
\begin{itemize}

\item[$\mathbf{A}(r)$:]  $\widetilde{e}_{il}L(\lambda)_{\lambda-\alpha} \subset L(\lambda)$, \
 $\widetilde{e}_{il}B(\lambda)_{\lambda-\alpha} \subset B(\lambda) \cup \{0\}$.

 \vskip 2mm

\item[$\mathbf{B}(r)$:] For $b \in B(\lambda)_{\lambda - \alpha + l \alpha_{i}}$,
$b' \in B(\lambda)_{\lambda - \alpha}$,\,  $\widetilde{f}_{il} b = b'$ if and only if
$\widetilde{e}_{il} b' = b$.

\vskip 2mm

\item[$\mathbf{C}(r)$:] $\Phi_{\lambda, \mu} (L(\lambda + \mu)_{\lambda + \mu - \alpha})
\subset L(\lambda) \otimes L(\mu)$.

\vskip 2mm

\item[$\mathbf{D}(r)$:] $\Psi_{\lambda, \mu}((L(\lambda) \otimes L(\mu))_{\lambda + \mu - \alpha})
\subset L(\lambda + \mu)$, \
 $\Psi_{\lambda, \mu}((B(\lambda) \otimes B(\mu))_{\lambda + \mu - \alpha})\subset B(\lambda + \mu) \cup \{0\}.$

\vskip 2mm

\item[$\mathbf{E}(r)$:] $\widetilde{e}_{il}L(\infty)_{-\alpha} \subset L(\infty)$, \
 $\widetilde{e}_{il}B(\infty)_{-\alpha} \subset B(\infty) \cup \{0\}$.

\vskip 2mm

\item[$\mathbf{F}(r)$:] For $b \in B(\infty)_{ - \alpha + l \alpha_{i}}$,
$b' \in B(\infty)_{- \alpha}$, \, $\widetilde{f}_{il} b = b'$ if and only if
$\widetilde{e}_{il} b' = b$.

\vskip 2mm

\item[$\mathbf{G}(r)$:] $B(\lambda)_{\lambda - \alpha}$ is a $\Q$-basis of
$(L(\lambda) / q L(\lambda))_{\lambda - \alpha}$, \  $B(\infty)_{- \alpha}$ is a $\Q$-basis of
$(L(\infty) / q L(\infty))_{ - \alpha}$.

\vskip 2mm

\item[$\mathbf{H}(r)$:] $\pi_{\lambda}(L(\infty)_{-\alpha}) = L(\lambda)_{\lambda - \alpha}$.

\vskip 2mm

\item[$\mathbf{I}(r)$:] For $S \in L(\infty)_{-\alpha + l \alpha_{i}}$,\,
$\widetilde{f}_{il}(S\, v_{\lambda}) \equiv (\widetilde{f}_{il} S)\, v_{\lambda}\ \text{mod} \ q L(\lambda)$.

\vskip 2mm

\item[$\mathbf{J}(r)$:] If $B^{\lambda}_{-\alpha}:= \{ b\in B(\infty)_{-\alpha}
\mid \overline{\pi}_{\lambda}(b) \neq 0 \}$,
then $B^{\lambda}_{-\alpha}  \cong B(\lambda)_{\lambda - \alpha}$.

\vskip 2mm

\item[$\mathbf{K}(r)$:] If $b \in B^{\lambda}_{-\alpha}$, then
$\widetilde{e}_{il} \, \overline{\pi}_{\lambda} (b) = \overline{\pi}_{\lambda}\, \widetilde{e}_{il} (b)$.

\end{itemize}

\vskip 3mm
We shall prove the statements ${\bf A}(r),  \ldots, {\bf K}(r)$ by induction.

\vskip 3mm

When $r=0$, $r=1$, our assertions are true. We now assume that ${\bf A}(r-1), \ldots, {\bf K}(r-1)$ are true.

\vskip 3mm

\begin{lemma}\label{alpha and b}
{\rm
Let $\alpha\in R_+(r-1)$ and $b\in {B(\lambda)}_{\lambda-\alpha}$.
If  $\widetilde{e}_{il}b=0$ for any $(i,l)\in I^{\infty}$, then we have $\alpha=0$ and $b=v_\lambda$.
}
\end{lemma}

\begin{proof}
The same argument in \cite[Lemma 7.2]{JKK05}, gives our claim.
	\end{proof}

\vskip 3mm

\begin{lemma}\label{uc in L}

{\rm 	Let $\alpha \in R_{+}(r-1)$, $i\in I^{\text{\rm im}}$, and
$u = \sum_{\mathbf{c} \in \mathcal{C}_{i}} {\mathtt b}_{i, \mathbf{c}}
u_{\mathbf{c}} \in V(\lambda)_{\lambda-\alpha}$
be the $i$-string decomposition of $u$.
If $u\in L(\lambda)$, then $u_{\mathbf c}\in L(\lambda)$ for any $\mathbf c\in\mathcal C_i$.	
}
\end{lemma}

\begin{proof}
	Suppose $u=\sum_{\mathbf c\in\mathcal C_i}{\mathtt b_{i,\mathbf c}u_{\mathbf c}\in L(\lambda)}$. We shall use the induction on $|\mathbf c|$.
	If $|\mathbf c|=0$, the assertion follows naturally.
	If $|\mathbf c|>0$, by $\mathbf{A}(r-1)$, we have  $\widetilde{e}_{il}u\in L(\lambda)$ for any $l>0$. By Definition \ref{def:Kashiwara operator}, we have
	
	\begin{equation*}
		\widetilde{e}_{il}u=\begin{cases}
			\sum_{\mathbf c:c_1=l}{\mathtt b_{i,\mathbf c\backslash c_1}u_{\mathbf c}\in L(\lambda)},& \text{if}\ i\in I^{\text{im}}\setminus I^{\text{iso}},\\
			\sum_{\mathbf c\in\mathcal C_i}  {\mathbf c}_{l}  {\mathtt b_{i,\mathbf c\backslash \{l\}}u_{\mathbf c}\in L(\lambda)},& \text{if}\ i\in I^{\text{iso}}.
		\end{cases}
	\end{equation*}
	Hence $u_{\mathbf c}\in L(\lambda)$ for any $\mathbf c\ne\mathbf 0$.
	
	Set $u_{1}:=\sum_{\mathbf c\ne\mathbf 0}{\mathtt b_{i,\mathbf c}u_{\mathbf c}}$. It follows that
	$u_1\in L(\lambda)$. Hence $u_{ 0}:=u-u_1\in L(\lambda)$, which proves our conclusion.
\end{proof}

\vskip 3mm

\begin{lemma}\label{uc=0}
{\rm
	Let $\alpha\in R_{+}(r-1)$, $i\in I^{\text{\rm im}}$ and let $u=\sum_{\mathbf c\in\mathcal C_i}{\mathtt b_{i,\mathbf c}u_{\mathbf c}\in {V(\lambda)}_{\lambda-\alpha}}$ be the $i$-string decomposition of $u$.
	If $ u+qL(\lambda)\in B(\lambda)$, then there exists $\mathbf c\in\mathcal C_i$ such that
	\begin{enumerate}
		\item[{\rm (a)}] $u\equiv	\widetilde{f}_{i,\mathbf c}u_{\mathbf c}\!\!\! \mod{qL(\lambda)}$,
		\item[{\rm (b)}] $u_{\mathbf c'}\equiv 0\!\!\! \mod{qL(\lambda)}$ for any $\mathbf c'\ne\mathbf c$.
	\end{enumerate}
	}
\end{lemma}

\begin{proof}
The case for  $|\mathbf c|=0$ is trivial. For $|\mathbf c|>0$, by $\mathbf{A}(r-1)$, we have $\widetilde{e}_{il}b\in B(\lambda)\cup \{0\}$ for any $l>0$.

If $\widetilde{e}_{il}b=0$ for any $l>0$, by Lemma \ref{uc in L}, we have
	$u_{\mathbf c}\in qL(\lambda)$ for any $\mathbf c\ne\mathbf 0$.
	Then $u\equiv u_{0}\!\!\! \mod{qL(\lambda)}$. Setting $\mathbf c=\mathbf 0$, our assertion follows trivially.
	
	Suppose $\widetilde{e}_{il}b\ne 0$ for some $l>0$. By induction, there exists $\mathbf c_0\in \mathcal C_i$ such that
	\begin{equation*}
		\widetilde{e}_{il} u=\begin{cases}	
			\widetilde{f}_{i,\mathbf c_0}u_{\mathbf c_0}\!\!\!\!\! \mod{qL(\lambda)},\\
			0\quad\text{for~ any}~\mathbf c_0'\ne \mathbf c_0.\\
		\end{cases}
	\end{equation*}
	
	Set $\mathbf c=(l,\mathbf c_0)$ or $\mathbf c=\mathbf c_0\cup \{l\}$. By $\mathbf{B}(r-1)$, we obtain
	\begin{equation*}
		u\equiv 	\widetilde{f}_{il}\widetilde{e}_{il} u\equiv \widetilde{f}_{il}\widetilde{f}_{i,\mathbf c_0}u_{\mathbf c_0}\equiv \widetilde{f}_{i,\mathbf c}u_{\mathbf c}\!\!\! \mod{qL(\lambda)}.
	\end{equation*}

	If $\mathbf c'\ne \mathbf c$, then $c_1\neq l$ or $c_1=l$, $\mathbf c_0'\ne\mathbf c_0$.
	It follows that $\widetilde{e}_{il}(\widetilde{f}_{i,\mathbf c'}u_{\mathbf c_0})=0$.
\end{proof}

\vskip 2mm
 By the same approach as that for Lemma \ref{uc in L} and Lemma \ref{uc=0}, we have the following lemma.

\vskip 3mm

\begin{lemma}
{\rm
	Let $\alpha\in  R_{+}(r-1)$, $i\in I^{\text{\rm im}}$ and let $u=\sum_{\mathbf c\in\mathcal C_i}{\mathtt b_{i,\mathbf c}u_{\mathbf c}}\in  {U^{-}_q(\g)}_{-\alpha}$ be the $i$-string decomposition of $u$.
	\begin{enumerate}
		\item[{\rm (a)}] If $u\in L(\infty)$, then $u_{\mathbf c}\in L(\infty)$ for any $\mathbf c$.
		\item[{\rm (b)}] If $u+qL(\infty)\in B(\infty)$, then there exists $\mathbf c\in\mathcal C_i$ such that
		\begin{enumerate}
			\item[{\rm (1)}] $u\equiv\widetilde{f}_{i,\mathbf c}u_{\mathbf c}\!\! \mod{qL(\infty)}$,
			\item[{\rm (2)}] $u_{\mathbf c'}\equiv 0\!\! \mod{qL(\infty)}$ for any $\mathbf c'\neq\mathbf c$.
		\end{enumerate}
	\end{enumerate}	

}
\end{lemma}

\vskip 2mm

The following lemma plays an important role in our proofs.

\vskip 2mm

\begin{lemma}\label{a-g}

{\rm
	Let $\alpha, \beta\in R_{+}(r-1)$ and $i\in I^{\text{\rm im}}$.
	\begin{enumerate}
		\item[{\rm (a)}] For all $l>0$, we have
		\begin{align*}
			&	\widetilde{e}_{il}(L(\lambda)_{\lambda-\alpha}\otimes L(\mu)_{\mu-\beta})\subset L(\lambda)\otimes L(\mu),\\	
			&\widetilde{f}_{il}(L(\lambda)_{\lambda-\alpha}\otimes L(\mu)_{\mu-\beta})\subset L(\lambda)\otimes L(\mu).			
		\end{align*}
		\item[{\rm (b)}] For all $l>0$, we have
		\begin{align*}
			&\widetilde{e}_{il}(B(\lambda)_{\lambda-\alpha}\otimes B(\mu)_{\mu-\beta})\subset (B(\lambda)\otimes B(\mu))\cup\{0\},\\
			&\widetilde{f}_{il}(B(\lambda)_{\lambda-\alpha}\otimes B(\mu)_{\mu-\beta})\subset (B(\lambda)\otimes B(\mu))\cup\{0\}.
		\end{align*}
		\item[{\rm (c)}]
		If $\widetilde{e}_{il}(b\otimes b')\ne 0$, then $b\otimes b'=\widetilde{f}_{il}\widetilde{e}_{il}(b\otimes b')$.
		
		\vskip 2mm
		
		\item[{\rm (d)}]
		If $\widetilde{e}_{il}(b\otimes b')= 0$ for all $l>0$, then $b=v_{\lambda}$.
		
		\vskip 2mm
		
		\item[{\rm (e)}]
        For any $(i,l)\in I^{\infty}$, we have $\widetilde{f}_{il}(b\otimes v_{\mu})=
			\widetilde{f}_{il}b\otimes v_{\mu}$ or $0$.	
			
			\vskip 2mm
			
		\item[{\rm (f)}]
		For any $(i_1,l_1),\cdots,(i_r,l_r)\in I^{\infty}$, we have
		\begin{equation*}
			\widetilde{f}_{i_1l_1}\cdots	\widetilde{f}_{i_rl_r}(v_{\lambda}\otimes v_{\mu})\equiv \widetilde{f}_{i_1l_1}\cdots \widetilde{f}_{i_rl_r}v_{\lambda}\otimes v_{\mu}\!\!\!\! \mod{q(L(\lambda)\otimes L(\mu))}
		\end{equation*}
		or $\widetilde{f}_{i_1l_1}\cdots\widetilde{f}_{i_rl_r}v_{\lambda}\equiv 0\!\! \mod{qL(\lambda)}$.
	\end{enumerate}
}
\end{lemma}
\begin{proof}
	The proofs for (a), (b), (c), (e) and (f) are similar to the ones given in \cite[Lemma 7.5]{JKK05}.
	So we shall only show the proof for (d).
	
	\vskip 2mm

    Suppose $\widetilde{e}_{il}(b\otimes b')=0$ for any $l>0$.
If $m=\langle h_i,\mathrm{wt}(b)\rangle>0$, then there exists $l>0$ such that
	\begin{equation*}
		0\leq -a_{ii}\leq\cdots\leq -la_{ii}\leq m\leq -(l+1)a_{ii}\leq\cdots.
	\end{equation*}
	For $0<k\leq l$, we have $\widetilde{e}_{ik}(b\otimes b')=\widetilde{e}_{ik}b\otimes b'=0$, then $\widetilde{e}_{ik}b=0$.
	By \cite[Proposition 4.4]{KK20}, we have $\widetilde{e}_{ik}b=0$ for any $k\geq l+1$. It follows that $\widetilde{e}_{il}b=0$ for any $l>0$.
	
	If $m=0$, then $m\leq -la_{ii}$ for all $l>0$.
	Hence by \cite[Proposition 4.4]{KK20}, we have $\widetilde{e}_{il}b=0$ for all $l>0$.
	Therefore, by Lemma \ref{alpha and b}, we have $b=v_{\lambda}$.
\end{proof}

\vskip 2mm

\begin{proposition}  \label{prop:Cr} $({\bf C}(r))$ \
{\rm
	For any $\alpha\in R_{+}(r)$, we have
	\begin{equation*}
		\Phi_{\lambda,\mu}(L(\lambda+\mu)_{\lambda+\mu-\alpha})\subset L(\lambda)\otimes L(\mu).
	\end{equation*}
	
}	
\end{proposition}

\begin{proof}
	Note that
	\begin{equation*}
		{L(\lambda+\mu)}_{\lambda+\mu-\alpha}=\sum_{(i,l)\in I_{\infty}}{\widetilde{f}_{il}({L(\lambda+\mu)}_{\lambda+\mu-\alpha+l\alpha_i})}.
	\end{equation*}
	Then our assertion follows from ${\bf C}(r-1)$ and Lemma \ref{a-g} (a).
\end{proof}

\vskip 2mm

\begin{lemma}\label{ffvvbb}
{\rm
	Let $(i_1,l_1),\cdots,(i_r,l_r)\in I^{\infty}$.
	Suppose that there exists $t$ with $t<r$ satisfying $i_t \ne i_{t+1}=\cdots=i_r$. Then for any $\mu\in P^{+}$ and $\lambda=\Lambda_{i_t}$, we have
	\begin{equation*}
		\widetilde{f}_{i_1l_1}\cdots\widetilde{f}_{i_rl_r}(v_{\lambda}\otimes v_{\mu})\equiv b\otimes b' \!\!\!\!\mod{q(L(\lambda)\otimes L(\mu))}
	\end{equation*}
	for some $b\in {B(\lambda)}_{\lambda-\alpha}\cup\{0\}$, $b'\in {B(\mu)}_{\mu-\beta}\cup\{0\}$ and $\alpha,\beta\in  R_{+}(r-1)$.
}
\end{lemma}

\begin{proof}
The condition $\Lambda_{i_t}(h_r)=0$ implies $\mathtt b_{i_rl_r}(v_{\lambda})=0$.
Thus for any $v\in V(\mu)$, we have
	\begin{align*}
		\mathtt b_{i_rl_r}(v_{\lambda}\otimes v)=\mathtt b_{i_rl_r}v_{\lambda}\otimes v+K_{i_r}^{l_r}v_{\lambda}\otimes \mathtt b_{i_rl_r}v=v_{\lambda}\otimes \mathtt b_{i_rl_r}v.
	\end{align*}
Set $v=\mathtt b_{i_{t+1}l_{t+1}}\cdots \mathtt b_{i_{r}l_{r}}v_{\mu}$. We have
\begin{align*}
	&\mathtt b_{i_tl_t}(v_{\lambda}\otimes\mathtt b_{i_{t+1}l_{t+1}}\cdots \mathtt b_{i_{r}l_{r}}v_{\mu})\\
	&=\mathtt b_{i_tl_t}v_{\lambda}\otimes\mathtt b_{i_{t+1}l_{t+1}}\cdots \mathtt b_{i_{r}l_{r}}v_{\mu}+K_{i_t}^{l_t}v_{\lambda}\otimes\mathtt b_{i_tl_t}\mathtt b_{i_{t+1}l_{t+1}}\cdots \mathtt b_{i_{r}l_{r}}v_{\mu}\\
	&=\widetilde{f}_{i_tl_t}v_{\lambda}\otimes	\widetilde{f}_{i_{t+1}l_{t+1}}\cdots	\widetilde{f}_{i_rl_r}v_{\mu}+q^{s_{i_t}l_t\langle h_{i_t},\lambda\rangle}v_{\lambda}\otimes \widetilde{f}_{i_tl_t}	 \widetilde{f}_{i_{t+1}l_{t+1}}\cdots\widetilde{f}_{i_rl_r}v_{\mu}\\
&\equiv \widetilde{f}_{i_tl_t}v_{\lambda}\otimes \widetilde{f}_{i_{t+1}l_{t+1}}\cdots	\widetilde{f}_{i_rl_r}v_{\mu}\!\!\!\! \mod{q(L(\lambda)\otimes L(\mu))},
\end{align*}
where $\widetilde{f}_{i_tl_t}v_{\lambda}\in {B(\lambda)}_{\lambda-\alpha}\cup\{0\}$ and $\widetilde{f}_{i_{t+1}l_{t+1}}\cdots	 \widetilde{f}_{i_rl_r}v_{\mu}\in {B(\mu)}_{\mu-\beta}$. Then the lemma follows from the tensor product rule \eqref{eq:im}.
\end{proof}

\vskip 2mm

By a similar argument as that for \cite[Lemma 7.8]{JKK05}, we have the following lemma.

\vskip 2mm

\begin{lemma}\label{LotimesL}
{\rm
	For any $\alpha\in R_{+}(r)$, we have
	\begin{equation*}
		{(L(\lambda)\otimes L(\mu))}_{\lambda+\mu-\alpha}
		=\sum_{(i,l)\in I_\infty}\mathtt b_{il} {(L(\lambda)\otimes L(\mu))}_{\lambda+\mu-\alpha+l\alpha_i}+v_{\lambda}\otimes {L(\mu)}_{\mu-\alpha}.
	\end{equation*}
}
\end{lemma}

For $\lambda,\mu\in P^+$, define a $U_{q}^-(\g)$-module homomorphism
\begin{align*}
S_{\lambda,\mu}: &V(\lambda)\otimes V(\mu)\rightarrow V(\lambda),\quad u\otimes v_\mu\mapsto u,\\
&V(\lambda)\otimes\sum_{(i,l)\in I_{\infty}}\widetilde{f}_{il}V(\mu) \longmapsto 0.
\end{align*}
Hence $u\otimes v\mapsto 0$ unless $v=\alpha v_\mu$ for some $\alpha\in \Q(q)$.

\vskip 2mm

\begin{lemma}\label{Slm}
{\rm
Let $\lambda,\mu\in P^+$.
\begin{enumerate}
	\item[{\rm (a)}] $S_{\lambda,\mu}(L(\lambda)\otimes L(\mu))=L(\lambda)$.
	\item[{\rm (b)}] For any $\alpha\in R_+(r-1)$ and $w\in{(L(\lambda)\otimes L(\mu))}_{\lambda+\mu-\alpha}$, we have
	\begin{align*}
	S_{\lambda,\mu}\circ\widetilde{f}_{il}(w)\equiv\widetilde{f}_{il}\circ S_{\lambda,\mu}(w)\!\!\!\! \mod{qL(\lambda)}.
	\end{align*}
\end{enumerate}	
}
\end{lemma}

\begin{proof}
(a) is obvious. For (b), we may assume that
\begin{align*}
w=u\otimes u'=\mathtt b_{i,\mathbf c}u_{\mathbf c}\otimes \mathtt b_{i,\mathbf c'}u_{\mathbf c'},
\end{align*}	
where $u_{\mathbf c}\in L(\lambda)$, $u_{\mathbf c'}\in L(\mu)$ and $\mathtt a_{ik}u_{\mathbf c}=\mathtt a_{ik}u_{\mathbf c'}=0$ for any $k>0$.

Let $L$ be the $\A_0$-submodule of $V(\lambda)\otimes V(\mu)$ generated by $\mathtt b_{i,\mathbf c}u_{\mathbf c}\otimes \mathtt b_{i,\mathbf c'}u_{\mathbf c'}$ for all $\mathbf c$ and  $\mathbf c'$. Thus $L\subset L(\lambda)\otimes L(\mu)$. By the tensor product rule, we have
\begin{align*}
\widetilde{f}_{il}(w)=\widetilde{f}_{il}(u\otimes u')=\begin{cases}
\widetilde{f}_{il}u\otimes u', & \text{if}\ \varphi_i(u)>0,\\
u\otimes \widetilde{f}_{il}u', & \text{if}\ \varphi_i(u)=0.
\end{cases}
\end{align*}

 If $\varphi_i(u)>0$, then we have $\widetilde{f}_{il}(w)=\widetilde{f}_{il}u\otimes u'$ and
	\begin{align*}
	&	S_{\lambda,\mu}\circ\widetilde{f}_{il}(w)=\begin{cases}
		\widetilde{f}_{il}u, &\text{if}\ \mathbf c'=\mathbf 0,\\
		0, &\text{otherwise},
		\end{cases}\\
	&   S_{\lambda,\mu}(w)=\begin{cases}
		u,& \text{if}\ \mathbf c'=\mathbf 0,\\
		0,& \text{otherwise}.
	\end{cases}
	\end{align*}

Hence we have
\begin{align*}
\widetilde{f}_{il}\circ S_{\lambda,\mu}(w)=\begin{cases}
	\widetilde{f}_{il}u, &\text{if}\ \mathbf c'=\mathbf 0,\\
	0, &\text{otherwise}.
\end{cases}\\
\end{align*}

If $\varphi_i(u)=0$, then we have
\begin{align*}
&\varphi_i(b)=0\Rightarrow u=u_{\mathbf 0},\\
&\mathbf c'=\mathbf 0\Rightarrow u'=u'_{\mathbf 0}.
\end{align*}

By \cite[Proposition 4.4]{KK20}, we have $\widetilde{f}_{il}(w)=u_{\mathbf 0}\otimes \widetilde{f}_{il}u'_{\mathbf 0}=0$. Hence $S_{\lambda,\mu}\circ \widetilde{f}_{il}(w)=0$.
On the other hand, by \cite[Proposition 4.4]{KK20} again, we have
$\widetilde{f}_{il} \circ S_{\lambda, \mu} (u \otimes u')
=\widetilde{f}_{il} (u) = 0$.
\end{proof}

\vskip 3mm

\begin{lemma}\label{fP and eP}
{\rm
Let $\alpha\in R_+$ and $S \in {U^-_q(\g)}_{-\alpha}$. For any $\lambda\gg 0$, we have
\begin{align*}
&(\widetilde{f}_{il}S)v_\lambda\equiv \widetilde{f}_{il}(Sv_\lambda)\!\!\!\! \mod{qL(\lambda)},\\
&(\widetilde{e}_{il}S)v_\lambda\equiv \widetilde{e}_{il}(Sv_\lambda)\!\!\!\! \mod{qL(\lambda)}.
\end{align*}
}
\end{lemma}

\begin{proof}
We may assume that $S=\mathtt b_{i,\mathbf c}T$ and $e'_{ik}T=0$ for any $k>0$.
Then we have $E_{ik}T=0$ for any $k>0$.
Note that
\begin{equation*}
\begin{aligned}
E_{ik}(T v_{\lambda}) & = q_{i}^{-k \langle h_{i}, \text{wt}(T) \rangle} T(E_{ik} v_{\lambda})
+\frac{e_{ik}'(T) - K_{i}^{2k} e_{ik}'' (T)} {1-q_{i}^{2k}} v_{\lambda} \\
& = -  \frac{q_{i}^{2k(\langle h_{i}, \lambda \rangle + k a_{ii} + \langle h_{i}, \text{wt}(T) \rangle)} }
{1 - q_{i}^{2k}} v_{\lambda}.
\end{aligned}
\end{equation*}
 Since $\lambda\gg 0$, we have $E_{ik}(T v_\lambda)\equiv 0 \!\!\mod{qL(\lambda)}$ for any $k>0$.
 \begin{enumerate}
 	\item If $i\notin I^{\text{iso}}$, we have
 	\begin{align*}
 	(\widetilde{f}_{il}S)v_\lambda=(\widetilde{f}_{il}(\mathtt b_{i,\mathbf c}T))v_\lambda	=(\mathtt b_{il}(\mathtt b_{i,\mathbf c}T))v_\lambda=\mathtt b_{il}(\mathtt b_{i,\mathbf c}(Tv_\lambda)).
 	\end{align*}
 Since $E_{ik}(Tv_\lambda)\equiv 0 \!\!\mod{qL(\lambda)}$ for any $k>0$, we have
 \begin{align*}
 &\mathtt b_{il}(\mathtt b_{i,\mathbf c}(Tv_\lambda))=\widetilde{f}_{il}(\mathtt b_{i,\mathbf c}Tv_\lambda)=\widetilde{f}_{il}(Sv_\lambda)\!\!\!\!\mod{qL(\lambda)},
\end{align*}
and
\begin{align*}
(\widetilde{e}_{il}S)v_\lambda=&(\widetilde{e}_{il}(\mathtt b_{i,\mathbf c}T))v_\lambda=(\mathtt b_{i,\mathbf c\setminus c_1}T)v_\lambda\\
=&\mathtt b_{i,\mathbf c\setminus c_1}(Tv_\lambda)=\widetilde{e}_{il}(\mathtt b_{i,\mathbf c}Tv_\lambda)
\equiv \widetilde{e}_{il}(Sv_\lambda)\!\!\!\! \mod{q L(\lambda)}.
\end{align*}
\item If $i\in I^{\text{iso}}$, we have
\begin{align*}
(\widetilde{f}_{il}S)v_\lambda&=(\widetilde{f}_{il}(\mathtt b_{i,\mathbf c}T))v_\lambda
=\frac{1} {\mathbf c_l+1} \, (\mathtt b_{il}\mathtt b_{i,\mathbf c}T)v_\lambda\\
&=\frac{1}{\mathbf c_l+1}\,\mathtt b_{il}(\mathtt b_{i,\mathbf c}Tv_\lambda)=\widetilde{f}_{il}(\mathtt b_{i,\mathbf c}Tv_\lambda)=\widetilde{f}_{il}(Sv_\lambda)\!\!\!\! \mod{qL(\lambda)},\\
 (\widetilde{e}_{il}S)v_\lambda&=(\widetilde{e}_{il}\mathtt b_{i,\mathbf c}T)v_\lambda
 ={\mathbf c}_l \,(\mathtt b_{i,\mathbf c\setminus \{l\}}T)v_\lambda
 ={\mathbf c}_l (\mathtt b_{i,\mathbf c\setminus \{l\}}(Tv_\lambda))\\
&=\widetilde{e}_{il}(\mathtt b_{i,\mathbf c}Tv_\lambda)=\widetilde{e}_{il}(Sv_\lambda)\!\!\!\! \mod{qL(\lambda)}.
\end{align*}
 	\end{enumerate}
	\end{proof}

\vskip 3mm

\begin{proposition}  \label{proof of Ir} $({\bf I}(r))$
{\rm
For $\lambda\in P^+$, $\alpha\in R_+(r-1)$ and $S\in {L(\infty)}_{-\alpha}$,
we have
\begin{align*}
(\widetilde{f}_{il}S)v_\lambda\equiv \widetilde{f}_{il}(Sv_\lambda) \!\!\!\!\mod{qL(\lambda)}.
\end{align*}

In particular, we have
\begin{align*}
(\widetilde{f}_{i_1l_1}\cdots\widetilde{f}_{i_rl_r}\mathbf 1)	v_\lambda\equiv \widetilde{f}_{i_1l_1}\cdots\widetilde{f}_{i_rl_r} v_\lambda\!\!\!\! \mod{qL(\lambda)}.
	\end{align*}
}
\end{proposition}

\begin{proof}
Take $\mu \gg 0$ such that $\lambda+\mu\gg 0$. By Lemma \ref{fP and eP}, we have
\begin{align*}
(\widetilde{f}_{il}S)v_{\lambda+\mu}\equiv\widetilde{f}_{il}(Sv_{\lambda+\mu})\!\!\!\! \mod{qL(\lambda+\mu)}.
\end{align*}

\vskip 2mm

By Proposition \ref{prop:Cr},
$\Phi_{\lambda,\mu}$ gives
\begin{align}\label{fP}
(\widetilde{f}_{il}S)(v_\lambda\otimes v_\mu)\equiv\widetilde{f}_{il}(S(v_\lambda\otimes v_\mu))\!\!\!\! \mod{q(L(\lambda)\otimes L(\mu))}.
\end{align}

\vskip 2mm

On the other hand, by ${\bf H}(r-1)$ and ${\bf C}(r-1)$, we have
\begin{align*}
S(v_\lambda\otimes v_\mu)=\Phi_{\lambda,\mu}(Sv_{\lambda+\mu})\in L(\lambda)\otimes L(\mu).
\end{align*}

Applying $S_{\lambda,\mu}$ to \eqref{fP}, then Lemma \ref{Slm} yields
\begin{align*}
(\widetilde{f}_{il}S)v_\lambda\equiv \widetilde{f}_{il}(Sv_\lambda) \!\!\!\!\mod{qL(\lambda)}.
\end{align*}
\end{proof}

\vskip 2mm

By a similar argument as that for \cite[Proposition 7.13]{JKK05}, we have the following proposition.

\vskip 2mm

\begin{proposition} \label{proof of Hr} $({\bf H}(r))$
{\rm
For any $\lambda\in P^+$ and $\alpha\in R_+(r)$, we have
\begin{align*}
\pi_\lambda({L(\infty)}_{-\alpha})={L(\lambda)}_{\lambda-\alpha}.
\end{align*}	
}
\end{proposition}

\vskip 2mm

\begin{corollary}\label{pi}
{\rm
Consider the $\Q$-linear map
\begin{align*}
\overline{\pi}_\lambda: {L(\infty)}_{-\alpha}/q{L(\infty)}_{-\alpha}\longrightarrow {L(\lambda)}_{\lambda-\alpha}/q{L(\lambda)}_{\lambda-\alpha}.
\end{align*}

\begin{enumerate}
	\item[{\rm (a)}] For any $\beta\in R_+(r-1)$ and $b\in{B(\infty)}_{-\beta}$, we have
	\begin{align*}
	\overline{\pi}_\lambda(\widetilde{f}_{il}b)=\widetilde{f}_{il}(\overline{\pi}_\lambda(b)).
	\end{align*}
\item[{\rm (b)}] For any $\alpha\in  R_+(r)$ and $\lambda\in P^+$, we have
\begin{align*}
\overline{\pi}_\lambda({B(\infty)}_{-\alpha})={B(\lambda)}_{\lambda-\alpha} \cup \{0\}.
\end{align*}
\item[{\rm (c)}] For any $\alpha\in R_+(r)$ and $\lambda\gg 0$, the map $\pi_\lambda$ induces the isomorphisms
\begin{align*}
{L(\infty)}_{-\alpha}\stackrel{\sim}{\rightarrow}{L(\lambda)}_{\lambda-\alpha},\quad  {B(\infty)}_{-\alpha}
\stackrel{\sim}{\rightarrow}{B(\lambda)}_{\lambda-\alpha}.
\end{align*}
\end{enumerate}
}
\end{corollary}

\vskip 3mm

Fix $\lambda\in  P^+$, $i\in I^{\text{im}}$, $l_1,\cdots,l_r>0$ and $\alpha=\sum_{j=1}^{r}l_j\alpha_{i_j}$. Take a finite set $T$ containing $\Lambda_{i_1},\cdots,\Lambda_{i_r}$.

\begin{enumerate}
	\item[i)] Since $T$ is a finite set, we can take a sufficient large $N_1\geq 0$ such that
	\begin{align*}
	\widetilde{e}_{il}{L(\tau)}_{\tau-\alpha}\subset q^{-N_1}L(\tau)\ \text{for all}\ \tau\in T.
	\end{align*}
    \item[ii)] Choose $N_2\geq 0$ such that
 $
    \widetilde{e}_{il}{L(\infty)}_{-\alpha}\subset q^{-N_2}L(\infty)$.

  \vskip 3mm

Then for any $\mu\gg 0$, Lemma \ref{fP and eP} and Proposition \ref{proof of Hr}
yield
\begin{align*}
\widetilde{e}_{il}{L(\mu)}_{\mu-\alpha}&=\widetilde{e}_{il}({L(\infty)}_{-\alpha}v_\mu)\subset (\widetilde{e}_{il}{L(\infty)}_{-\alpha})v_\mu+q{L(\mu)}_{\mu-\alpha}\\
&\subset q^{-N_2}{L(\infty)}_{-\alpha}v_\mu+q{L(\mu)}_{\mu-\alpha}\subset q^{-N_2}L(\mu).
\end{align*}
\end{enumerate}

\vskip 2mm

Therefore, for any $\alpha\in R_+(r)$, there exists $N\geq 0$ such that
\begin{equation}\label{eL}
\begin{aligned}
&\widetilde{e}_{il}{L(\mu)}_{\mu-\alpha}\subset q^{-N}L(\mu)\ \text{for all}\ \mu\gg 0,\\
&\widetilde{e}_{il}{L(\tau)}_{\tau-\alpha}\subset q^{-N}L(\tau)\ \text{for all}\ \tau\in T,\\
&\widetilde{e}_{il}{L(\infty)}_{-\alpha}\subset q^{-N}L(\infty).	
	\end{aligned}
\end{equation}

\vskip 3mm

\begin{lemma}\label{eLL}
{\rm
For any $\alpha\in R_+$, let $N\geq 0$ be a non-negative integer satisfying \eqref{eL}. For any $\mu\gg 0$ and $\tau\in T$, we have
\begin{align*}
\widetilde{e}_{il}{(L(\tau)\otimes L(\mu))}_{\tau+\mu-\alpha}\subset q^{-N}(L(\tau)\otimes L(\mu)).
\end{align*}
}
\end{lemma}

\begin{proof}
Let $u\in {L(\tau)}_{\tau-\beta}$ and $v\in {L(\mu)}_{\mu-\gamma}$	such that $\alpha=\beta+\gamma$.

\vskip 2mm
\noindent $\mathbf{Claim}$: $\widetilde{e}_{il}(u\otimes v)\in q^{-N}(L(\tau)\otimes L(\mu))$.

If $\beta\neq 0$ and $\gamma\neq 0$, the claim is exactly the one in Lemma \ref{a-g} (a).

If $\beta=0$, then $\gamma=\alpha$, we may assume that $u=v_\tau$.
Let $v=\sum_{\mathbf c\in \mathcal C_i}\mathtt b_{i,\mathbf c}v_{\mathbf c}$ be the $i$-string decomposition of $v$. By \eqref{eL}, we have
\begin{align*}
\widetilde{e}_{il}v=\begin{cases}
\sum_{\mathbf c\neq\mathbf 0}\mathtt b_{i,\mathbf c\setminus c_1}v_{\mathbf c}\in q^{-N}L(\mu), & \text{if}\ i\notin I^{\text{iso}},\\
\sum_{\mathbf c\neq\mathbf 0} \mathbf c_l \,
\mathtt b_{i,\mathbf c\setminus \{l\}}v_{\mathbf c}\in q^{-N}L(\mu), & \text{if}\ i\in I^{\text{iso}}.
\end{cases}
\end{align*}
Hence by Lemma \ref{uc in L}, we obtain
\begin{align*}
v_{\mathbf c}\in q^{-N}L(\mu)\ \text{for any}\ \mathbf c\neq\mathbf 0.
\end{align*}

\vskip 2mm

Let $L$ be the $\A_0$-submodule of $L(\tau)\otimes L(\mu)$ generated by $\mathtt b_{i,\mathbf c_1}v_\tau\otimes\mathtt b_{i,\mathbf c_2}v_{\mathbf c}$ for $\mathbf c_1,\mathbf c_2,\mathbf c \neq \mathbf 0$. Then $\widetilde{e}_{il}L\subset L$. It follows that
\begin{align*}
\widetilde{e}_{il}(v_\tau\otimes v)=\sum_{\mathbf c\neq\mathbf 0}\widetilde{e}_{il}(v_{\tau}\otimes\mathtt b_{i,\mathbf c}v_{\mathbf c})\in L\subset q^{-N}(L(\tau)\otimes L(\mu)).
\end{align*}
Similarly, the claim can be shown for the case
 $\beta=\alpha$, $\gamma=0$.
	\end{proof}

\begin{lemma}\label{lemma for Ar and Hr}
{\rm
Let $\alpha\in R_+(r)$ and let $N>0$ be the positive integer satisfying \eqref{eL}. Then we have
\begin{enumerate}
	\item[{\rm (a)}] $\widetilde{e}_{il}{L(\mu)}_{\mu-\alpha}\subset q^{1-N}L(\mu)$ for all $\mu\gg 0$,
	\item[{\rm (b)}] $\widetilde{e}_{il}{L(\tau)}_{\tau-\alpha}\subset q^{1-N}L(\tau)$ for all $\tau\in T$,
	\item[{\rm (c)}] $\widetilde{e}_{il}{L(\infty)}_{-\alpha}\subset q^{1-N}L(\infty)$.	
\end{enumerate}
}
\end{lemma}
\begin{proof}
	(a)
Let $u=\widetilde{f}_{i_1l_1}\cdots\widetilde{f}_{i_t l_t}v_\mu\in {L(\mu)}_{\mu-\alpha}$. Suppose $i_1=i_2=\cdots=i_t$. If $i=i_1$, then
\begin{align*}
u=\mathtt b_{i,\mathbf c}v_\mu,\ \mathbf c=(l_1,\cdots,l_t).
\end{align*}
Hence
\begin{align*}
\widetilde{e}_{il}u=\widetilde{e}_{il}(\mathtt b_{i,\mathbf c}v_\mu)=\begin{cases}
\mathtt b_{i,\mathbf c\setminus c_1}v_\mu, & \text{if}\ i\notin I^{\text{iso}},\ c_1=l,\\
\mathbf c_l \, \mathtt b_{i,\mathbf c\setminus \{l\}}v_\mu, & \text{if}\ i\in I^{\text{iso}},\ l\in\mathbf c,\\
0, & \text{otherwise}.
\end{cases}
\end{align*}
Therefore, we have $\widetilde{e}_{il}u\in L(\mu)$. 

\vskip 2mm

If $i\neq i_1$, then $\widetilde{e}_{il}u=0$.
Thus we may assume that there exists $s$ with $1\leq s<t$ such that
$i_s\neq i_{s+1}=\cdots=i_t$.
Suppose $\mu \gg 0$ and set $\lambda_0=\Lambda_{i_s}$. Then $\mu':=\mu-\lambda_0\gg 0$.
Set
\begin{align*}
w:=\widetilde{f}_{i_1l_1}\cdots\widetilde{f}_{i_t l_t} (v_{\lambda_0}\otimes v_{\mu'}).
\end{align*}
By Lemma \ref{ffvvbb}, we have
\begin{align*}
w\equiv v\otimes v'\!\!\!\!\! \mod{qL(\lambda_0)\otimes L(\mu')}
\end{align*}
for some $v\in {L(\lambda_0)}_{\lambda_0-\beta}$, $v'\in {L(\mu')}_{\mu'-\gamma}$, $\alpha=\beta+\gamma$ and $\beta,\gamma\in R_+(r-1)$.

\vskip 2mm

Then Lemma \ref{a-g} (a) and Lemma \ref{eLL} imply
\begin{align*}
\widetilde{e}_{il}w&\in L(\lambda_0)\otimes L(\mu')+q \, \widetilde{e}_{il}{(L(\lambda_0)\otimes L(\mu'))}_{\lambda_0+\mu'-\alpha}\\
&\subset L(\lambda_0)\otimes L(\mu')+q^{1-N}L(\lambda_0)\otimes L(\mu')=q^{1-N}L(\lambda_0)\otimes L(\mu').
\end{align*}

Thus we have 
\begin{align*}
\widetilde{e}_{il}w\in q^{1-N}{(L(\lambda_0)\otimes L(\mu'))}_{\lambda_0+\mu'-\alpha+l\alpha_i}=q^{1-N}{(L(\lambda_0)\otimes L(\mu'))}_{\mu-\alpha+l\alpha_i}.
\end{align*}
Applying $\Psi_{\lambda_0,\mu'}$ to ${\bf D}(r-1)$, we have
\begin{align*}
\widetilde{e}_{il}u=\widetilde{e}_{il}\widetilde{f}_{i_1l_1}\cdots\widetilde{f}_{i_t l_t}v_\mu\in q^{1-N}L(\mu).
\end{align*}
(b)
Let $\tau\in T$ and set
$
u=\widetilde{f}_{i_1l_1}\cdots\widetilde{f}_{i_tl_t}v_\tau\in {L(\tau)}_{\tau-\alpha}$.
If $u\in qL(\tau)$, our assertion follows from \eqref{eL}.

If $u\notin qL(\tau)$, for any $\mu\in P^+$, Lemma \ref{a-g} (f) gives
\begin{align}\label{ffvv}
\widetilde{f}_{i_1l_1}\cdots\widetilde{f}_{i_tl_t}(v_\tau\otimes v_\mu)\equiv u\otimes v_\mu\!\!\!\! \mod{qL(\tau)\otimes L(\mu)}.
\end{align}

If $\mu\gg 0$, (a) implies
\begin{align*}
\widetilde{e}_{il}\widetilde{f}_{i_1l_1}\cdots\widetilde{f}_{i_tl_t}v_{\tau+\mu}\in q^{1-N}L(\tau+\mu).
\end{align*}

Applying $\Phi_{\tau,\mu}$ and $\mathbf{B}(r-1)$, we obtain
\begin{align*}
\widetilde{e}_{il}\widetilde{f}_{i_1l_1}\cdots\widetilde{f}_{i_tl_t}(v_\tau\otimes v_\mu)\in q^{1-N}L(\tau)\otimes L(\mu).
\end{align*}
By \eqref{ffvv} and Lemma \ref{a-g}, we have
\begin{align}\label{evv}
\widetilde{e}_{il}(v\otimes v_\mu)\in q^{1-N}L(\tau)\otimes L(\mu)+q\widetilde{e}_{il}(L(\tau)\otimes L(\mu))\subset q^{1-N}L(\tau)\otimes L(\mu).
\end{align}

Let $u=\sum_{\mathbf c\in\mathcal C_i}\mathtt b_{i,\mathbf c}u_{\mathbf c}$ be the $i$-string decomposition of $u$. By \eqref{eL}, we have $\widetilde{e}_{il}u\in q^{-N}L(\tau)$.

Recall
\begin{align*}
\widetilde{e}_{il}u=
\begin{cases}
\sum_{\mathbf c\in\mathcal C_i}	\mathtt b_{i,\mathbf c\setminus c_1}u_{\mathbf c}, & \text{if}\ i\notin I^{\text{iso}},\ c_1=l,\\
\sum_{\mathbf c\in\mathcal C_i}	\mathbf c_l \, \mathtt b_{i,\mathbf c\setminus \{l\}}u_{\mathbf c}, & \text{if}\ i\in I^{\text{iso}},\ l\in\mathbf c,\\
	0 & \text{otherwise}.
\end{cases}
\end{align*}
By Lemma \ref{uc in L}, we have $u_{\mathbf c}\in q^{-N}L(\tau)$. Let $L$ be the $\A_0$-submodule of $V(\tau)\otimes V(\mu)$ generated by $\mathtt b_{i,\mathbf c}u_{\mathbf c}\otimes \mathtt b_{i,\mathbf c'}v_{\mu}$ $(c_1=l\ \text{or}\ l\in\mathbf c)$. Then we have $L\subset q^{-N}L(\tau)\otimes L(\mu)$.

\vskip 2mm

The tensor product rule gives
\begin{align*}
\widetilde{e}_{il}(u\otimes v_\mu)\equiv \widetilde{e}_{il}u\otimes v_\mu\!\!\!\! \mod{qL}.
\end{align*}
By \eqref{evv}, we have
\begin{align*}
\widetilde{e}_{il}u\otimes v_\mu\equiv\widetilde{e}_{il}(u\otimes v_\mu)\in q^{1-N}L(\tau)\otimes L(\mu).
\end{align*}
Hence $\widetilde{e}_{il}u\in q^{1-N}L(\tau)$.
\vskip 2mm
(c)
Let $S\in {L(\infty)}_{-\alpha}$ and take $\mu\gg 0$. By Lemma \ref{fP and eP}, we have
$
(\widetilde{e}_{il}S)v_\mu\equiv \widetilde{e}_{il}(Sv_\mu)\!\!\! \mod{qL(\mu)}$.
Thus Proposition \ref{proof of Hr}
implies
\begin{align*}
(\widetilde{e}_{il}S)v_\mu=\widetilde{e}_{il}(Sv_\mu)\in\widetilde{e}_{il}{L(\mu)}_{\mu-\alpha}\subset q^{1-N}L(\mu).
\end{align*}

\noindent
Hence by Corollary \ref{pi} (c), we have
\begin{align*}
\widetilde{e}_{il}S\in q^{1-N}L(\infty).
\end{align*}
	\end{proof}
	
	\vskip 2mm

\begin{corollary}
{\rm
For $\alpha\in R_+(r)$, we have $0\notin {B(\infty)}_{-\alpha}$.
}
\end{corollary}

\begin{proof}
If $b\in {B(\infty)}_{-\alpha}$, then there exist $(i,l)\in I_{\infty}$ and $b'\in{B(\infty)}_{-\alpha+l\alpha_i}$ such that $b=\widetilde{f}_{il}b'$. By $\mathbf{G}(r-1)$, the set ${B(\infty)}_{-\alpha+l\alpha_i}$ forms a $\Q$-basis of ${L(\infty)}_{-\alpha+l\alpha_i}/q{L(\infty)}_{-\alpha+l\alpha_i}$. Then we have $b'\neq 0$. Hence $b\neq 0$.
 	\end{proof}
	
\vskip 2mm

\begin{lemma}\label{pie=epi}
{\rm
Let $\alpha\in R_+(r)$, $(i,l)\in I_{\infty}$, $\lambda\gg 0$ and $b\in{B(\infty)}_{-\alpha}$.
Then we have
\begin{equation*}
\overline{\pi}_\lambda(\widetilde{e}_{il}b)=\widetilde{e}_{il}\overline{\pi}_\lambda(b).
\end{equation*}
}
\end{lemma}

\begin{proof}
The assertion follows directly from Lemma \ref{fP and eP}.
	\end{proof}
	
\vskip 2mm

\begin{corollary}\label{fLLeLL}
{\rm
Let $\lambda,\mu\in P^+$ and $\alpha,\beta\in R_+(r)$.
\begin{enumerate}
	\item[{\rm (a)}] For the $i$-string decomposition $u=\sum_{\mathbf c\in\mathcal C_i}\mathtt b_{i,\mathbf c}u_{\mathbf c}\in {L(\lambda)}_{\lambda-\alpha}$, we have $u_{\mathbf c}\in L(\lambda)$
	for all $\mathbf{c} \in \mathcal{C}_{i}$.
	\item[{\rm (b)}] For any $(i,l)\in I^{\infty}$, we have
	\begin{align*}
	&\widetilde{f}_{il}({L(\lambda)}_{\lambda-\alpha}\otimes {L(\mu)}_{\mu-\beta})\subset L(\lambda)\otimes L(\mu),\\
	&\widetilde{e}_{il}({L(\lambda)}_{\lambda-\alpha}\otimes {L(\mu)}_{\mu-\beta})\subset L(\lambda)\otimes L(\mu).
	\end{align*}
\end{enumerate}
}
\end{corollary}

\begin{proof}
Since Lemma \ref{a-g} depends only on ${\bf A}(r-1)$, the
corollary follows from the proof of Lemma \ref{uc in L}.
	\end{proof}

\vskip 2mm

\begin{lemma}\label{euv}
{\rm
Let $\lambda,\mu\in P^+$ and $\alpha\in R_+(r)$. For any $u\in{L(\lambda)}_{\lambda-\alpha}$, we have
\begin{align*}
\widetilde{e}_{il}(u\otimes v_\mu)\equiv \widetilde{e}_{il}u\otimes v_\mu\!\!\!\! \mod{q(L(\lambda)\otimes L(\mu))}.	
	\end{align*}
}
\end{lemma}

\begin{proof}
The lemma follows from the fact $\widetilde{e}_{il}v_\mu=0$.	
	\end{proof}
	
\vskip 2mm

\begin{proposition}  \label{proof of Kr} $({\bf K}(r))$
{\rm
	Let $\lambda\in P^+$ and $\alpha\in R_+(r)$. If $b\in{B(\infty)}_{-\alpha}$ and $\overline{\pi}_\lambda(b)\neq 0$, then we have
	\begin{align*}
		\widetilde{e}_{il}\overline{\pi}_\lambda(b)=\overline{\pi}_\lambda(\widetilde{e}_{il}b).
		\end{align*}
		
}
\end{proposition}

\begin{proof}
We set
\begin{align*}
S&=\widetilde{f}_{i_1l_1}\cdots\widetilde{f}_{i_tl_t}\mathbf 1\in {L(\infty)}_{-\alpha},\\
b&=S +q{L(\infty)}_{-\alpha}\in {B(\infty)}_{-\alpha},\\
u&=\widetilde{f}_{i_1l_1}\cdots\widetilde{f}_{i_tl_t}v_\lambda.
\end{align*}
By Proposition  \ref{proof of Ir},
we have
\begin{align*}
u=\widetilde{f}_{i_1l_1}(\widetilde{f}_{i_2l_2}\cdots\widetilde{f}_{i_tl_t}v_{\lambda})\equiv (\widetilde{f}_{i_1l_1}\cdots\widetilde{f}_{i_tl_t})v_\lambda=Sv_\lambda\!\!\!\!\mod{qL(\lambda)}.
\end{align*}

\noindent
Since $\overline{\pi}_\lambda(b)\neq 0$ and $u\notin qL(\lambda)$. By Lemma \ref{a-g} (f), for any $\mu\in P^+$, we have
\begin{align*}
\widetilde{f}_{i_1l_1}\cdots\widetilde{f}_{i_tl_t}(v_\lambda\otimes v_\mu)\equiv \widetilde{f}_{i_1l_1}\cdots\widetilde{f}_{i_tl_t}v_\lambda\otimes v_\mu\equiv u\otimes v_\mu\!\!\!\! \mod{q(L(\lambda)\otimes L(\mu))}.
\end{align*}

\noindent
Hence by Lemma \ref{euv}, we have
\begin{align}\label{effvv1}
\widetilde{e}_{il}(\widetilde{f}_{i_1l_1}\cdots\widetilde{f}_{i_tl_t}(v_\lambda\otimes v_\mu))\equiv\widetilde{e}_{il}(u\otimes v_\mu)\equiv\widetilde{e}_{il}u\otimes v_\mu\!\!\!\! \mod{q(L(\lambda)\otimes L(\mu))}.	
	\end{align}
	\vskip 3mm
	
On the other hand, for $\mu\gg 0$, by Lemma \ref{pie=epi}, we have
\begin{align*}
\widetilde{e}_{il}(\widetilde{f}_{i_1l_1}\cdots\widetilde{f}_{i_tl_t}v_{\lambda+\mu})\equiv \widetilde{e}_{il}(Sv_{\lambda+\mu})\equiv (\widetilde{e}_{il}S)v_{\lambda+\mu}\!\!\!\! \mod{qL(\lambda+\mu)}.
\end{align*}
Applying $\Phi_{\lambda,\mu}$ and  Proposition \ref{prop:Cr}, we obtain
\begin{align}\label{effvv2}
\widetilde{e}_{il}(\widetilde{f}_{i_1l_1}\cdots\widetilde{f}_{i_tl_t}(v_\lambda\otimes v_\mu))\equiv(\widetilde{e}_{il}S)(v_\lambda\otimes v_\mu)\!\!\!\! \mod{q(L(\lambda)\otimes L(\mu))}.
\end{align}
Then \eqref{effvv1} and \eqref{effvv2} yield
\begin{align*}
\widetilde{e}_{il}u\otimes v_\mu\equiv (\widetilde{e}_{il}S)(v_\lambda\otimes v_\mu)\!\!\!\! \mod{q(L(\lambda)\otimes L(\mu))}.	
	\end{align*}
Applying $S_{\lambda,\mu}$, we conclude
\begin{align*}
\widetilde{e}_{il}u\equiv (\widetilde{e}_{il}S)v_\lambda\!\!\!\! \mod{qL(\lambda)}.
\end{align*}
Hence $\widetilde{e}_{il}\overline{\pi}_\lambda(b)= \overline{\pi}_{\lambda}(\widetilde{e}_{il}b).$
	\end{proof}
	
\vskip 2mm

\begin{proposition}  \label{proof of Er} $({\bf E}(r))$ \
{\rm	For every $\alpha\in R_+(r)$, we have
	\begin{align*}
	\widetilde{e}_{il}{L(\infty)}_{-\alpha}\subset L(\infty),\quad	\widetilde{e}_{il}{B(\infty)}_{-\alpha}\subset B(\infty)\cup \{0\}.
	\end{align*}
}
\end{proposition}

\begin{proof}
Applying Lemma \ref{lemma for Ar and Hr} (c) repeatedly, the first assertion holds.
For the second assertion, let
	\begin{align*}
		S=\widetilde{f}_{i_1l_1}\cdots\widetilde{f}_{i_tl_t}\mathbf 1\in{L(\infty)}_{-\alpha},\quad b=S+q{L(\infty)}_{-\alpha}\in {B(\infty)}_{-\alpha}.
	\end{align*}
	
	If $i_1=i_2=\cdots=i_t$, our assertion is true as we have seen in the proof of Lemma \ref{lemma for Ar and Hr} (a). Here,  we may assume that there exists $s$ with  $1\leq s<t$ such that $i_s\neq i_{s+1}=\cdots= i_t$. Take $\mu\gg 0$ and set $\lambda_0=\Lambda_{i_s}$, $\lambda=\lambda_0+\mu\gg 0$. Then Lemma \ref{ffvvbb} yields
	\begin{align*}
		S(v_{\lambda_0}\otimes v_\mu)=\widetilde{f}_{i_1l_1}\cdots\widetilde{f}_{i_tl_t}(v_{\lambda_0}\otimes v_\mu)\equiv v\otimes v'\!\!\!\! \mod{q(L(\lambda_0)\otimes L(\mu))}
	\end{align*}
	for some $v\in{L(\lambda_0)}_{\lambda_0-\beta}$, $v'\in{L(\mu)}_{\mu-\gamma}$, $\beta,\gamma\in R_+(r-1)\setminus\{0\}$ and $\alpha=\beta+\gamma$ such that
	\begin{align*}
		v+qL(\lambda_0)\in B(\lambda_0)\cup \{0\},\quad v'+qL(\mu)\in B(\mu)\cup \{0\}.
	\end{align*}
   Therefore we have
	\begin{align*}
		\widetilde{e}_{il}(\widetilde{f}_{i_1l_1}\cdots\widetilde{f}_{i_tl_t}(v_{\lambda_0}\otimes v_\mu))\equiv \widetilde{e}_{il}(v\otimes v')\equiv\widetilde{e}_{il}v\otimes v'\!\!\!\! \mod{q(L(\lambda_0)\otimes L(\mu))}.
	\end{align*}
	By ${\bf A}(r-1)$, we have
	\begin{align*}
		\widetilde{e}_{il}(\widetilde{f}_{i_1l_1}\cdots\widetilde{f}_{i_tl_t}(v_{\lambda_0}\otimes v_\mu))+q(L(\lambda_0)\otimes L(\mu))\in (B(\lambda_0)\otimes B(\mu))\cup \{0\}.
	\end{align*}
	The map $\Psi_{\lambda_0,\mu}$ and $\mathbf{D}(r-1)$ yield
	\begin{align*}
		\widetilde{e}_{il}\overline{\pi}_\lambda(b)=\widetilde{e}_{il}(\widetilde{f}_{i_1l_1}\cdots\widetilde{f}_{i_tl_t}v_\lambda+qL(\lambda))\in B(\lambda)\cup \{0\}.
	\end{align*}
	
	Since $\lambda\gg 0$, Lemma \ref{pie=epi} and Corollary \ref{pi} (c) yield
	\begin{align*}
		\widetilde{e}_{il}b=\widetilde{e}_{il}(\widetilde{f}_{i_1l_1}\cdots\widetilde{f}_{i_tl_t}\mathbf 1+qL(\infty))\in B(\infty)\cup\{0\}.
	\end{align*}
\end{proof}

\vskip 2mm

\begin{proposition}  \label{proof of Ar} $({\bf A}(r))$
{\rm
For any $\lambda\in P^+$ and $\alpha\in R_+(r)$, we have
\begin{align*}
\widetilde{e}_{il}{L(\lambda)}_{\lambda-\alpha}\subset L(\lambda),\quad
\widetilde{e}_{il}{B(\lambda)}_{\lambda-\alpha}\subset B(\lambda)\cup \{0\}.
\end{align*}
}
\end{proposition}

\begin{proof}
	Proposition follows from Lemma \ref{lemma for Ar and Hr}, Proposition  \ref{proof of Kr}, Proposition \ref{proof of Er},
	Corollary \ref{pi} (b) and Proposition \ref{proof of Hr}.	
	\end{proof}

\vskip 2mm

For $(i,l) \in I^{\infty}$, let $u = {\mathtt b}_{il}^m \, u_{0}$ such that $E_{ik} u_{0} = 0$ for all $k>0$.
Define an operator $Q_{il}: V(\lambda) \rightarrow V(\lambda)$ by

\begin{equation} \label{eq:Qill}
Q_{il}(u) = \begin{cases}
(m+1) u & \text{if} \ i\in I^{\text{iso}}, \\
u & \text{otherwise}
\end{cases}
\end{equation}

\vskip 2mm

\begin{lemma}\label{lem:LLA0}
{\rm
	Let $\lambda\in P^{+}$ and $\alpha\in R_{+}(r)$.
	\begin{enumerate}	
		\item[{\rm (a)}] For any $u\in {L(\lambda)}_{\lambda-\alpha+l\alpha_i}$ and $v\in {L(\lambda)}_{\lambda-\alpha}$, we have
		\begin{equation*}
			(\widetilde{f}_{il} Q_{il}\, u,v)_K\equiv (u,\widetilde{e}_{il}v)_K\!\!\!\! \mod{q \A_0}.
		\end{equation*}
	
\vskip 2mm
		\item[{\rm (b)}]
		$({L(\lambda)}_{\lambda-\alpha},{L(\lambda)}_{\lambda-\alpha})_K\subset \A_0$.
	\end{enumerate}
}
\end{lemma}

\begin{proof}
	(a) By \eqref{eq:bilinearV}, we have
	\begin{equation*}
		(\mathtt b_{il} u,v)_K=(u,E_{il}v)_K\equiv (u,\widetilde{e}_{il}v)_K\!\!\!\! \mod{qL(\lambda)}.
	\end{equation*}
	Therefore, if $i\notin I^{\text{iso}}$, the conclusion holds naturally.
	
	If $i\in I^{\text{iso}}$, we may assume that $u=\mathtt b_{i,\mathbf c}u_0$ and $E_{ik}u_0=0$ for any $k>0$.
Then we have
	\begin{align*}
(\widetilde{f}_{il} Q_{il}(u), v)_{K} & = (\mathbf{c}_{l}+1) (\widetilde{f}_{il}(u), v)_{K}
= ({\mathtt b}_{\mathbf{c} \cup \{l\}} \, u_{0}, v)_{K} \\
& = ({\mathtt b}_{i, \mathbf{c}} u_{0}, E_{il} v)_{K}
\equiv (u, \widetilde{e}_{il} v)_{K} \ \text{mod} \ q \A_{0},
	\end{align*}
 which gives our assertion.

 \vskip 2mm
	
	(b) By induction, we have $(u,\widetilde{f}_{il}v)_K\in \A_0$.
	Hence $(\widetilde{f}_{il}u,v)_K \in \A_0$, which proves the claim.
\end{proof}

\vskip 2mm

\begin{lemma}\label{(P,Q)}
{\rm
	Let $\alpha=l_1\alpha_{i_1}+\cdots+l_t\alpha_{i_t}\in R_{+}$, $S,T\in {U^-_q(\g)}_{-\alpha}$ and $m\in \Z$. For any $\lambda\gg 0$, we have
	\begin{equation*}
		(S,T)_K=\prod_{k=1}^{t}{(1-q_i^{2l_k})}^{-1}(Sv_{\lambda},Tv_{\lambda})_K \!\!\!\!\mod{q^{m}\A_0}.
	\end{equation*}
}
\end{lemma}
\begin{proof}
	If $S=\mathbf 1$, then $\alpha=0$, and $(S,T)_K=(v_{\lambda},v_{\lambda})_K=1$.

\vskip 2mm

	We shall prove the assertion by induction on $\mathrm{ht}(\alpha)$. Assume that $S=\mathtt b_{il}W$
	for some $W\in {U^{-}_q(\g)}_{-\alpha+l\alpha_i}$.
	Then we have
	\begin{align*}
		&(Sv_{\lambda},Tv_{\lambda})_K=(Wv_{\lambda},E_{il}(Tv_{\lambda}))_K\\
		=&(Wv_{\lambda},q_i^{-l\langle h_i,\text{wt}(T) \rangle}T(E_{il}v_{\lambda}))_K+(Wv_{\lambda},\frac{e'_{il}T-K_i^{2l}e''_{il}T}{1-q^2_i}v_{\lambda})_K\\
		=&{(1-q_i^{2l})}^{-1}(Wv_{\lambda},(e'_{il}T)v_{\lambda})_K-{(1-q_i^{2l})}^{-1}q_i^{2l\langle h_i,\lambda-\alpha\rangle+2 l^2 a_{ii}}(Wv_{\lambda},(e_{il}''T)v_{\lambda})_K.
	\end{align*}

	Since $\lambda\gg 0$, we have $q_i^{2l\langle h_i,\lambda-\alpha\rangle+2l^2a_{ii}}\equiv 0\!\! \mod{q^{m}\A_0}$.
Hence by induction, we obtain
	\begin{align*}
		(Sv_{\lambda},Tv_{\lambda})_K&\equiv {(1-q_i^{2l})}^{-1}(Wv_{\lambda},(e'_{il}T)v_{\lambda})_K\\
		&\equiv{(1-q_i^{2l})}^{-1}(W,e'_{il}T)_K\equiv{(1-q_i^{2l})}^{-1}(S,T)_K\!\!\!\! \mod{q^m \A_0}.
	\end{align*}
\end{proof}

\vskip 2mm

Let $L$ be a finitely generated $\A_0$-submodule of ${V(\lambda)}_{\lambda-\alpha}$ and set
\begin{align*}
L^{\vee}:=\{u\in {V(\lambda)}_{\lambda-\alpha}\ |\ (u,L)_K\subset \A_0\}.
\end{align*}

Similarly, let $L$ be a finitely generated $\mathbf A_0$-submodule of ${L(\infty)}_{-\alpha}$ and set
\begin{align*}
	L^{\vee}=\{u\in {U^-_q(\g)}_{-\alpha}\ |\ (u,L)_K\subset \A_0\}.
\end{align*}

Then $(L^{\vee})^{\vee}=L$ and we obtain

\begin{lemma}\label{piL=L}
{\rm
	If $\lambda\gg 0$ and $\alpha\in R_{+}(r)$, we have $\pi_\lambda({L(\infty)}^\vee_{-\alpha})={L(\lambda)}_{\lambda-\alpha}^{\vee}$.
}
\end{lemma}

\begin{proof}
	Let ${\{S_k\}}_{k\in I}$ be an $\A_0$-basis of ${L(\infty)}_{-\alpha}$ and let ${\{T_k\}}_{k\in I}$ be its dual basis with respect to the bilinear form $(\ ,\ )_K$, i.e., $(S_i,T_j)_K=\delta_{ij}$. Then ${L(\infty)}_{-\alpha}^{\vee}=\sum_{j\in I}{\A_0 T_j}$.
	
\vskip 2mm

	By Proposition \ref{proof of Hr},
	we have $L(\lambda)=\sum_{k\in I}{\A_0 (S_kv_{\lambda})}$.	By Lemma \ref{(P,Q)}, for $\lambda\gg 0$, we have
	\begin{align*}
	(S_kv_{\lambda},T_jv_{\lambda})_K\equiv \delta_{kj}\!\!\!\! \mod{q \A_0}.	
		\end{align*}

	Hence we conclude
	\begin{equation*}
		{L(\lambda)}_{\lambda-\alpha}^{\vee}=\sum_{j\in I}{\mathbf A_0T_jv_{\lambda}}=\pi_{\lambda}({L(\infty)}_{-\alpha}^{\vee})\ \text{for}\ \lambda\gg 0.
	\end{equation*}
\end{proof}

\vskip 3mm
\begin{lemma}\label{PsiLL}
{\rm
	Let $\lambda\in P^{+}$, $\mu\gg 0$ and $\alpha\in R_{+}(r)$. Then we have
	\begin{equation*}
		\Psi_{\lambda,\mu}({(L(\lambda)\otimes L(\mu))}_{\lambda+\mu-\alpha})\subset {L(\lambda+\mu)}_{\lambda+\mu-\alpha}.
	\end{equation*}
}
\end{lemma}

\begin{proof}
	By Lemma \ref{LotimesL}, we have
	\begin{equation*}
		{(L(\lambda)\otimes L(\mu))}_{\lambda+\mu-\alpha}=\sum_{(i,l)\in I_\infty}{\widetilde{f}_{il}({(L(\lambda)\otimes L(\mu))}_{\lambda+\mu-\alpha+l\alpha_i})+v_{\lambda}\otimes {L(\mu)}_{\mu-\alpha}}.
	\end{equation*}
	By induction hypothesis ${\bf D}(r-1)$, we get
	\begin{align*}
		&\Psi_{\lambda,\mu}(\sum_{(i,l)\in I_\infty}{\widetilde{f}_{il}({(L(\lambda)\otimes L({\mu}))}_{\lambda+\mu-\alpha+l\alpha_i})}\\
		=&\sum_{(i,l)\in I_\infty}{\widetilde{f}_{il}\Psi_{\lambda,\mu}({(L(\lambda)\otimes L({\mu}))}_{\lambda+\mu-\alpha+l\alpha_i})}\\
		\subset&\sum_{(i,l)\in I_\infty}{\widetilde{f}_{il}{L(\lambda+\mu)}_{\lambda+\mu-\alpha+l\alpha_i}}={L(\lambda+\mu)}_{\lambda+\mu-\alpha}.
	\end{align*}
	
	It remains to show
	\begin{equation*}
		\Psi_{\lambda,\mu}(v_{\lambda}\otimes {L(\mu)}_{\mu-\alpha})\subset {L(\lambda+\mu)}_{\lambda+\mu-\alpha}.
	\end{equation*}
	
	Let $u\in {L(\lambda+\mu)}_{\lambda+\mu-\alpha}^{\vee}$. By Lemma \ref{piL=L}, we have $u=Sv_{\lambda+\mu}$ for some $S\in {L(\infty)}_{-\alpha}^{\vee}$.
	Note that
	\begin{equation*}
		\Delta(S)=S\otimes\mathbf 1+(\text{intermediate terms})+K_{\alpha}\otimes S.
	\end{equation*}
	
	Then we have
	\begin{align*}
		&(\Phi_{\lambda,\mu}(u),v_{\lambda}\otimes {L(\mu)}_{\mu-\alpha})
		=(\Delta(S)(v_{\lambda}\otimes v_{\mu}),v_{\lambda}\otimes {L(\mu)}_{\mu-\alpha})\\
		=&(Sv_{\lambda}\otimes v_{\mu}+(\text{intermediate terms})+K_{\alpha}v_{\lambda}\otimes Sv_{\mu}, v_{\lambda}\otimes {L(\mu)}_{\mu-\alpha})\\
		=&(Sv_{\lambda},v_{\lambda})({v_{\mu}},{L(\mu)}_{\mu-\alpha})+(\text{intermediate terms})+K_{\alpha}(v_{\lambda},v_{\lambda})(Sv_{\mu},{L(\mu)}_{\mu-\alpha})\\
		=&q^{(\alpha,\lambda)}(Sv_{\mu},{L(\mu)}_{\mu-\alpha}).
	\end{align*}
	
	Since $\mu\gg 0$, Lemma \ref{piL=L} implies that $Sv_{\mu}\in {L(\mu)}^{\vee}$. Thus
	\begin{align*}
		(u,\Psi_{\lambda,\mu}(v_{\lambda}\otimes {L(\mu)}_{\mu-\alpha}))=(\Phi_{\lambda,\mu}(u),v_{\lambda}\otimes {L(\mu)}_{\mu-\alpha})=q^{(\alpha,\lambda)}(Sv_{\mu},{L(\mu)}_{\mu-\alpha})\subset \A_0.
	\end{align*}
	
	Hence $\Psi_{\lambda,\mu}(v_{\lambda}\otimes {L(\mu)}_{\mu-\alpha})\subset {({L(\lambda+\mu)}_{\lambda+\mu-\alpha}^{\vee})}^\vee={L(\lambda+\mu)}_{\lambda+\mu-\alpha}$.
\end{proof}

\vskip 2mm

\begin{proposition} \label{prop: proof of Fr} 	{\rm $({\bf F}(r))$ \
	Let $\alpha\in R_{+}(r)$ and $b\in {B(\infty)}_{-\alpha}$.
	If $\widetilde{e}_{il}b\ne 0$, then $b=\widetilde{f}_{il}\widetilde{e}_{il}b$.	
}
\end{proposition}

\begin{proof}
	Let $b=\widetilde{f}_{i_1l_1}\cdots\widetilde{f}_{i_tl_t}\mathbf 1\in {B(\infty)}_{-\alpha}$. We assume $\widetilde{e}_{il}b\ne 0$. If $i_1=\cdots=i_t$ and $i\ne i_{1}$, then
	\begin{equation*}
		\widetilde{e}_{il}b=\widetilde{e}_{il}\widetilde{f}_{i_1l_1}\cdots\widetilde{f}_{i_tl_t}\mathbf 1=\cdots=\widetilde{f}_{i_1l_1}\cdots\widetilde{f}_{i_tl_t}\widetilde{e}_{il}\mathbf 1=0.
	\end{equation*}
	Hence we must have $i=i_1=\cdots=i_t$.
	In this case, our assertion follows easily.
	
	Assume that there exists $s$ with $1\leq s<t$ such that $i_s\ne i_{s+1}=\cdots=i_t$. Take $u\gg 0$ and set $\lambda_0=\Lambda_{i_s},~\lambda=\lambda_0+\mu$.
	
	Then Lemma \ref{ffvvbb} yields
	\begin{equation*}
		\widetilde{f}_{i_1l_1}\cdots\widetilde{f}_{i_tl_t}	(v_{\lambda_0}\otimes v_{\mu})\equiv  v\otimes v'\!\!\!\! \mod{q(L(\lambda_0)\otimes L(\mu))}
	\end{equation*}
	for some $v\in {L(\lambda_0)}_{\lambda_0-\beta}$, $v'\in {L(\mu)}_{\mu-\gamma}$, $\beta,\gamma\in \mathrm Q_{+}(r-1)\backslash\{0\}$ and $\alpha=\beta+\gamma$ such that
	\begin{align*}
	v+qL(\lambda_0)\in B(\lambda_0)\cup\{0\},\quad v'+qL(\mu)\in B(\mu)\cup\{0\}.
	\end{align*}
	
	By Corollary \ref{fLLeLL} (b), we have
	\begin{equation*}
		\widetilde{e}_{il}(\widetilde{f}_{i_1l_1}\cdots\widetilde{f}_{i_tl_t}(v_{\lambda_0}\otimes v_{\mu}))\equiv \widetilde{e}_{il}(v\otimes v')\!\!\!\! \mod{q (L(\lambda_0)\otimes L(\mu))}.
	\end{equation*}
	
	Then $\Psi_{\lambda_0,\mu}$ and $\mathbf{H}(r-1)$ yield
	\begin{align*}
		\pi_{\lambda}(\widetilde{e}_{il}\widetilde{f}_{i_1l_1}\cdots \widetilde{f}_{i_tl_t} \mathtt 1)&=\widetilde{e}_{il}\widetilde{f}_{i_1l_1}\cdots\widetilde{f}_{i_tl_t}v_{\lambda_0+\mu}\equiv \Psi_{\lambda_0,\mu}(\widetilde{e}_{il}(v\otimes v'))\!\!\!\! \mod{qL(\lambda)}.
	\end{align*}
Since $\mu \gg 0$, we have $\widetilde{e}_{il}(v\otimes v')\notin q(L(\lambda_0)\otimes L(\mu))$.
	
	By Lemma \ref{a-g} (c), we have
	\begin{align*}
		\widetilde{f}_{i_1l_1}\cdots\widetilde{f}_{i_tl_t}(v_{\lambda_0}\otimes v_{\mu})&\equiv v\otimes v'\equiv	 \widetilde{f}_{il}\widetilde{e}_{il}(v\otimes v')\\
		&\equiv \widetilde{f}_{il}\widetilde{e}_{il}(\widetilde{f}_{i_1l_1}\cdots\widetilde{f}_{i_tl_t}(v_{\lambda_0}\otimes v_{\mu})) \!\!\!\!\mod{q (L(\lambda_0)\otimes L(\mu))}.
	\end{align*}
	
	Applying $\Psi_{\lambda_0,\mu}$ and Lemma \ref{PsiLL}, we obtain
	\begin{align*}
		 \widetilde{f}_{i_1l_1}\cdots\widetilde{f}_{i_tl_t}v_{\lambda_0+\mu}&=\widetilde{f}_{i_1l_1}\cdots\widetilde{f}_{i_tl_t}v_{\mu}=\widetilde{f}_{il}\widetilde{e}_{il}(\widetilde{f}_{i_1l_1}\cdots\widetilde{f}_{i_tl_t}v_{\lambda})\!\!\!\! \mod{qL(\lambda)}.
	\end{align*}
	Since $\lambda\gg 0$, we get $b=\widetilde{f}_{il}\widetilde{e}_{il}b\!\! \mod{qL(\infty)}$.
\end{proof}

\vskip 2mm

\begin{proposition} \label{proof of Br} $({\bf B}(r))$ \
{\rm
	Let $\lambda\in P^{+}$ and $\alpha\in R_{+}(r)$.
	For $b\in {B(\lambda)}_{\lambda-\alpha+l\alpha_i}$ and $b'\in {B(\lambda)}_{\lambda-\alpha}$, we have $\widetilde{f}_{il}b=b'$ if and only if $b=\widetilde{e}_{il}b'$.
	}
\end{proposition}
\begin{proof}
	Suppose $\widetilde{f}_{il}b=b'$. By Lemma \ref{uc=0}, there exists $\mathbf c\in \mathcal C_i$ with $|\mathbf c|\geq l$, such that
	\begin{align*}
	b\equiv \mathtt b_{i,\mathbf c}u_0,\ E_{ik}u_0=0\ \text{for all}\ k>0.
	\end{align*}
	If $i\notin I^{\text{iso}}$, we have
	\begin{align*}
	&\widetilde{f}_{il}b=\mathtt b_{i,(l,\mathbf c)}u_0=b',\\
	&\widetilde{e}_{il}b'=\widetilde{e}_{il}\mathtt b_{i,(l,\mathbf c)}u_0=\mathtt b_{i,\mathbf c}u_0=b.
	\end{align*}
If $i\in I^{\text{iso}}$, we have
	\begin{equation*}
		\widetilde{f}_{il}b= \frac{1}{\mathbf c_l+1} \, \mathtt b_{i,\mathbf c\cup\{l\}}u_0=b'.
	\end{equation*}
	Hence
	\begin{align*}
	\widetilde{e}_{il}b'=\frac{\mathbf c_l+1}{\mathbf c_l+1}\mathtt b_{i,\mathbf c} u_0=b.
	\end{align*}
	
	Conversely, suppose $b'\in {B(\lambda)}_{\lambda-\alpha}$ and $b=\widetilde{e}_{il}b'\in {B(\lambda)}_{\lambda-\alpha+l\alpha_i}$.
	By Corollary \ref{pi} (b), we have $b'=\overline{\pi}_{\lambda}(b_0')$ for some $b_0'\in {B(\infty)}_{-\alpha}$. Proposition    \ref{proof of Kr}
	implies that
	\begin{equation*}
		\overline{\pi}_{\lambda}(\widetilde{e}_{il}b_0')=\widetilde{e}_{il}(\overline{\pi}_{\lambda}(b_0'))=\widetilde{e}_{il}b'\ne 0.
	\end{equation*}
	Hence $\widetilde{e}_{il}b_0'\ne 0$ in $B(\infty)$. By Proposition \ref{prop: proof of Fr}, we have $b_0'=\widetilde{f}_{il}\widetilde{e}_{il}b_0'$.
	
	Applying $\overline{\pi}_{\lambda}$, we obtain
	\begin{equation*}
		\widetilde{f}_{il}b=\widetilde{f}_{il}(\widetilde{e}_{il}b')=\widetilde{f}_{il}	 \overline{\pi}_{\lambda}(\widetilde{e}_{il}b_0')=\overline{\pi}_{\lambda}(\widetilde{f}_{il}\widetilde{e}_{il}b_0')=\overline{\pi}_{\lambda}(b_0')=b'.
	\end{equation*}
\end{proof}

\vskip 2mm

\begin{proposition}  \label{proof of Gr} $({\bf G}(r))$ \
{\rm	Let $\lambda\in P^{+}$ and $\alpha\in R_{+}(r)$. We have the following facts.
	\begin{enumerate}	
		\item[{\rm (a)}] ${B(\lambda)}_{\lambda-\alpha}$ is a $\Q$-basis of ${L(\lambda)}_{\lambda-\alpha}/q{L(\lambda)}_{\lambda-\alpha}$.
		\item[{\rm (b)}] ${B(\infty)}_{-\alpha}$ is a $\Q$-basis of ${L(\infty)}_{-\alpha}/q{L(\infty)}_{-\alpha}$.
	\end{enumerate}
}
\end{proposition}

\begin{proof}
	Suppose $\sum_{b\in {B(\lambda)}_{\lambda-\alpha}}{a_bb=0}$ for $a_b\in \Q$.
	
	By Proposition  \ref{proof of Ar},
	we have $\widetilde{e}_{il}{B(\lambda)}_{\lambda-\alpha}\subset B(\lambda)\cup \{0\}$ for any $(i,l)\in I^{\infty}$,
	which implies that
	\begin{align*}
	\widetilde{e}_{il}(\sum_{b}{a_bb})=\sum_{ b\in {B(\lambda)}_{\lambda-\alpha},\atop   \widetilde{e}_{il}b\ne 0}
	{a_b(\widetilde{e}_{il}b)}=0.
	\end{align*}
	
	By ${\bf G}(r-1)$ and Proposition \ref{proof of Br},
	we have $a_b=0$ whenever  $\widetilde{e}_{il}b\ne 0$.
But for each $b\in {B(\lambda)}_{\lambda-\alpha}$, there exists $(i,l)\in I^{\infty}$ such that $\widetilde{e}_{il}b\ne 0$.
Thus $a_b=0$ for any $b\in {B(\lambda)}_{\lambda-\alpha}$. Hence, the proposition holds.
\end{proof}

\vskip 2mm

\begin{lemma}\label{uLVLU}
{\rm
	Let $\lambda\in P^{+}$ and $\alpha\in \Q_{+}(r)\backslash \{0\}$.
	\begin{enumerate}	
		\item[{\rm (a)}] If $u\in {L(\lambda)}_{\lambda-\alpha}/ q{L(\lambda)}_{\lambda-\alpha}$ and $\widetilde{e}_{il}u=0$ for any $(i,l)\in I^{\infty}$, then $u=0$.	
		\item[{\rm (b)}]
		If $u\in {V(\lambda)}_{\lambda-\alpha}$ and $\widetilde{e}_{il}u\in L(\lambda)$ for any $(i,l)\in I^{\infty}$, then $u\in {L(\lambda)}_{\lambda-\alpha}$.
		\item[{\rm (c)}]
		If $u\in {L(\infty)}_{-\alpha}/q{L(\infty)}_{-\alpha}$ and $\widetilde{e}_{il}u=0$ for any $(i,l)\in I^{\infty}$, then $u=0$.
		\item[{\rm (d)}]
		If $u\in {U^-_q(\g)}_{-\alpha}$ and $\widetilde{e}_{il}u\in L(\infty)$ for any $(i,l)\in I^{\infty}$, then $u\in {L(\infty)}_{-\alpha}$.
	\end{enumerate}
}
\end{lemma}
\begin{proof}
	(a) Let $u=\sum_{b\in{B(\lambda)}_{\lambda-\alpha}}{a_bb}$ ($a_b\in \Q$).
	For any $(i,l)\in I^{\infty}$, we have
	\begin{equation*}
		\widetilde{e}_{il}u=\sum_{b\in {B(\lambda)}_{\lambda-\alpha} \atop \ \widetilde{e}_{il}b\ne 0}{a_b(\widetilde{e}_{il}b)=0}.
	\end{equation*}
	
It follows from the proof of Proposition  \ref{proof of Gr} that
	all $a_b=0$. Hence $u=0$.
	
\vskip 2mm
	
	(b) Choose the smallest $N\geq 0$ such that $q^{N}u\in L(\lambda)$.
	If $N>0$, we have
	\begin{align*}
	\widetilde{e}_{il}(q^{N}u)=q^{N}(\widetilde{e}_{il}u)\in qL(\lambda)
	\end{align*}
for all $(i,l)\in I^{\infty}$.
		By (a), we have
	$q^{N}u\in qL(\lambda)$, i.e., $q^{N-1}u\in L(\lambda)$ which contradicts to the minimality of $N$.
		Hence $N=0$ and $u\in L(\lambda)$.
	The proofs of (c) and (d) are similar.
\end{proof}

\vskip 2mm

By a similar argument as that for \cite[Proposition 7.34]{JKK05}, we have the following proposition.

\begin{proposition} \label{proof of Jr} $({\bf J}(r))$ \
{\rm
	Let $\lambda\in P^{+}$ and $\alpha\in  R_{+}(r)$, then we have
	\begin{equation*}
		B_{-\alpha}^{\lambda}:=\{b\in {B(\infty)}_{-\alpha}\mid \overline{\pi}_{\lambda}(b)\ne 0\}\xrightarrow{\sim} {B(\lambda)}_{\lambda-\alpha}.
	\end{equation*}
}
\end{proposition}
\vskip 2mm

Using all the statements we have proved so far, we can show that Lemma \ref{a-g} holds for all
 $\alpha\in R_+(r)$.

 \vskip 2mm

 In particular, we have

\begin{lemma}\label{eBB}
{\rm
Let $\lambda,\mu\in  P^+$ and $\alpha\in R_+(r)$.

\vskip 2mm

\begin{enumerate}
\item[{\rm (a)}] For all $(i,l)\in I^{\infty}$, we have
\begin{align*}
\widetilde{e}_{il}{(B(\lambda)\otimes B(\mu))}_{\lambda+\mu-\alpha}\subset (B(\lambda)\otimes B(\mu))\cup \{0\}.
\end{align*}
\item[{\rm (b)}] If $b\otimes b'\in {(B(\lambda)\otimes B(\mu))}_{\lambda+\mu-\alpha}$ and $\widetilde{e}_{il}(b\otimes b')\neq 0$, then we have
\begin{align*}
b\otimes b'=\widetilde{f}_{il}\widetilde{e}_{il}(b\otimes b').
\end{align*}
	\end{enumerate}
}
\end{lemma}

\vskip 2mm

\begin{proposition} \label{proof of Dr} $({\bf D}(r))$
{\rm
For every $\lambda,\mu\in P^+$ and $\alpha\in R_+(r)$, we have
\begin{enumerate}
\item[\rm (a)] $\Psi_{\lambda,\mu}({(L(\lambda)\otimes L(\mu))}_{\lambda+\mu-\alpha})\subset L(\lambda+\mu)$,
\item[\rm (b)] $\Psi_{\lambda,\mu}({(B(\lambda)\otimes B(\mu))}_{\lambda+\mu-\alpha})\subset B(\lambda+\mu)\cup \{0\}$.
\end{enumerate}
}
\end{proposition}

\begin{proof}
Proposition follows by Lemma \ref{PsiLL}, Lemma \ref{uLVLU}, Lemma \ref{eBB} and
\cite[Proposition 7.36]{JKK05}.
	\end{proof}

Thus we have completed the proofs of all the statements in Kashiwara's grand-loop argument,
which proves
Theorem \ref{thm:crystal basis of V} and Theorem \ref{thm:crystal basis of U}.

\vskip 2mm
Let $(\ ,\ )_K^0$ denote the $\Q$-valued inner product on $L(\lambda)/qL(\lambda)$ (resp. $L(\infty)/qL(\infty)$) by taking crystal limit of $(\ ,\ )_K$ on $L(\lambda)$ (resp. $L(\infty)$).

\begin{lemma}\label{lem:orthogonal}
{\rm The crystal $B(\lambda)$ (resp. $B(\infty)$) forms an orthogonal basis of $L(\lambda)/qL(\lambda)$ (resp. $L(\infty)/qL(\infty)$) with respect to $(\ ,\ )_K^0$.}
\end{lemma}

\begin{proof}
We first consider the crystal $B(\lambda)$. For all $b,b'\in {B(\lambda)}_{\lambda-\alpha}$, we shall prove $(b,b')_K^0\in \delta_{b,b'} \Z_{>0}$ by using induction on $\mathrm{ht}(\alpha)$, where $\alpha\in R_+(r)$.

\vskip 2mm

If $\mathrm{ht}(\alpha)=0$, then our conclusion is trivial.

\vskip 2mm

If $\mathrm{ht}(\alpha)>0$, we choose $(i,l)\in I^{\infty}$ such that $\widetilde{e}_{il}b\neq 0$.
By $\mathbf B(r)$ and Lemma \ref{lem:LLA0}, we have
\begin{equation*}
(b,b')_K^0=(\widetilde{f}_{il} Q_{il} \widetilde{e}_{il}b,b')_K^0 = (\widetilde{e}_{il}b,\widetilde{e}_{il}b')_K^0
\in \delta_{\widetilde{e}_{il}b,\widetilde{e}_{il}b'} \Z_{>0}=\delta_{b,b'} \Z_{>0}.
\end{equation*}

\vskip 2mm

By Lemma \ref{(P,Q)} and a similar approach above, it is easy to show that the crystal $B(\infty)$ is an orthogonal basis of  $L(\infty)/qL(\infty)$ with respect to $(\ ,\ )_K^0$.
\end{proof}

\vskip 3mm

\section{Global bases}\label{sec:global bases}

\vskip 2mm

Let $\A=\Z[q,q^{-1}]$, $\A_{\Q} = \Q[q, q^{-1}]$ and $\A_{\infty}$
be the subring of $\Q(q)$ consisting of rational functions which are regular at $q=\infty$.

\vskip3mm

\begin{definition} \label{def:balanced triple}

Let $V$ be a  $\Q(q)$-vector space.
Let
$V_{\Q}$, $L_{0}$ and $L_{\infty}$ be an $\A_{\Q}$-lattice, $\A_{0}$-lattice and
$\A_{\infty}$-lattice, respectively.
We say that $(V_{\Q}, L_{0}, L_{\infty})$ is a
{\it balanced triple} for $V$ if the following conditions hold:

\vskip 2mm

\begin{enumerate}

\item[(a)] The $\Q$-vector space $V_{\Q}\cap L_0\cap L_\infty$ is a free $\Q$-lattice of the
$\A_0$-module $ L_0$.

\vskip 2mm

\item[(b)] The $\Q$-vector space $V_{\Q}\cap L_0\cap L_\infty$ is a free $\Q$-lattice of the
$\A_\infty$-module $ L_\infty$.
\vskip 2mm

\item [(c)]The $\Q$-vector space $V_{\Q}\cap L_0\cap L_\infty$ is a free $\Q$-lattice of the
$\A_{\Q}$-module $V_{\Q}$.
\end{enumerate}
\end{definition}

\vskip 3mm

\begin{theorem}\cite{HK02, Kashi91} \label{thm:global basis}  
{\rm
The following statements are equivalent.

\vskip 2mm

\begin{enumerate}
	\item[{\rm (a)}] $(V_{\Q}, L_0, L_\infty)$ is  a balanced triple.
	
	\vskip 2mm
	
	\item[{\rm (b)}] The canonical map $V_{\Q}\cap L_0\cap L_\infty\to L_0/q L_0$ is an isomorphism.
	
	\vskip 2mm
	
	\item[{\rm (c)}] The canonical map $V_{\Q}\cap L_0\cap L_\infty\to L_\infty/q L_\infty$ is an isomorphism.
	
\end{enumerate}

}

\end{theorem}

\vskip 3mm

Let $(V_{\Q}, L_0, L_\infty)$ be a balanced triple and let
\begin{equation*}
G: L_0/q L_0\longrightarrow V_{\Q}\cap L_0\cap L_\infty
\end{equation*}
be the inverse of the canonical isomorphism
$V_{\Q}\cap L_0\cap L_\infty\stackrel{\sim}{\longrightarrow} L_0/q L_0$.

\vskip 3mm

\begin{proposition} \cite{HK02, Kashi91} \hfill  

\vskip 2mm

{\rm
If $B$ is a $\Q$-basis of $ L_0/q L_0$,
then $\B :=\{G(b)\mid b\in B\}$ is an $\A_{\Q}$-basis of $V_{\Q}$.
}
	\end{proposition}

\vskip 3mm

\begin{definition}
Let $(V_{\Q}, L_0, L_\infty)$ be a balanced triple for a $\Q(q)$-vector space $V$.
\begin{enumerate}
\vskip 2mm
	\item[(a)] A $\Q$-basis $B$ of $ L_0/q L_0$ is called a {\it local basis} of $V$ at $q=0$.
	
\vskip 2mm
	
	\item[(b)] The $\A_{\Q}$-basis $\B = \{G(b) \mid b\in B \}$ is called the {\it lower global basis} of $V$ corresponding to the local basis $B$.
\end{enumerate}
\end{definition}	

\vskip 3mm
We define $U^{-}_{\Z}(\g)$ (resp. $U^{-}_{\Q}(\g)$) to be the $\A$-subalgebra
(resp. $\A_{\Q}$-subalgebra) of $U_{q}^{-}(\g)$ generated by
$\mathtt{b}_{i}^{(n)}$ $(i \in I^{\text{re}}, n \ge 0)$ and
${\mathtt b}_{il}$ $(i \in I^{\text{im}}, l>0)$.

\vskip 3mm

Let $V(\lambda) = U_{q}(\g) v_{\lambda}$ be the irreducible highest weight module
with highest weight $\lambda \in P^{+}$.
We  define
$V(\lambda)_{\Z} = U_{\Z}^{-}(\g) \, v_{\lambda}$
and $V(\lambda)_{\Q} = U_{\Q}^{-}(\g) \, v_{\lambda}$.

\vskip 3mm

\begin{lemma}\label{lem:PbQ}
{\rm
	For any $S,T\in U_{q}^-(\g)$, we have
\begin{align}
&(S\mathtt b_{il},T)_K=(S,K_i^le''_{il}TK_i^{-l})_K,\label{eq:pbQK}\\
&(S,T)_K=(S^*,T^*)_K.\label{eq:PQK}
\end{align}

}
\end{lemma}
\begin{proof}
For \eqref{eq:pbQK}, we shall use induction on $|S|$.
We write $S=\mathtt b_{jk}S_0$. By \eqref{eq:commute}, we have
\begin{align*}
(S\mathtt b_{il},T)_K&=(\mathtt b_{jk}S_0\mathtt b_{il},T)_K=(S_0\mathtt b_{il},e'_{jk}T)_K=(S_0,K_i^le''_{il}e'_{jk}TK_i^{-l})_K\\
&=(S_0,e'_{jk}K_i^le''_{il}TK_i^{-l})_K=(\mathtt b_{jk}S_0,K_i^le''_{il}TK_i^{-l})_K=(S,K_i^le''_{il}TK_i^{-l})_K.
\end{align*}

\vskip 3mm

For \eqref{eq:PQK}, it is enough to prove the following claim.

$$({(S\mathtt b_{il})}^*,T^*)_K=(S\mathtt b_{il},T)_K.$$

\vskip 2mm

By \eqref{eq:pbQK} and \eqref{eq:ePstar}, we have
\begin{align*}
({(S\mathtt b_{il})}^*,T^*)_K&=(\mathtt b_{il}S^*,T^*)_K=(S^*,e'_{il}T^*)_K\\
&=(S^*,K_i^l{(e''_{il}T)}^*K_i^{-l})_K=(S,K_i^l(e''_{il}T)K_i^{-l})_K=(S\mathtt b_{il},T)_K,
\end{align*}
which proves our assertion.
\end{proof}

\vskip 3mm

Combining Lemma \ref{lem:PbQ}, Lemma \ref{fP and eP}, Proposition \ref{proof of Hr}, Corollary \ref{pi}, Lemma \ref{pie=epi} and Proposition \ref{proof of Kr} and using the same arguments in \cite[Section 5]{FKKT22}, we obtain

\vskip 3mm

\begin{theorem}\cite[Theorem 5.9]{FKKT22}\label{thm:balanced triple}

\vskip 2mm

{\rm
There exist $\Q$-linear canonical isomorphisms

\vskip 2mm

\begin{enumerate}
	\item[{\rm (a)}] $U^-_{\Q}(\g)\cap L(\infty)\cap\overline{L(\infty)}\stackrel{\sim}{\longrightarrow} L(\infty)/qL(\infty)$,
	where $^{-} : U_{q}(\g) \rightarrow U_{q}(\g)$ is the $\Q$-linear bar involution defined by \eqref{eq:bar},
	
\vskip 2mm

	\item[{\rm (b)}] ${V(\lambda)_{\Q}}\cap L(\lambda)\cap\overline{L(\lambda)}\stackrel{\sim}{\longrightarrow} L(\lambda)/qL(\lambda)$,
where $^{-}$ is the $\Q$-linear automorphism on $V(\lambda)$ defined by
$$
P\, v_\lambda\mapsto \overline{P} \, v_\lambda \ \ \text{for}\ P\in U_q^-(\g).
$$
\end{enumerate}

}
\end{theorem}

\vskip 3mm

Therefore we obtain:

\vskip 2mm

\begin{proposition}

{\rm
Let $G$ denote the inverse of the above isomorphisms.

\begin{enumerate}

\vskip 2mm

\item[(a)] $\B(\infty) : = \{ G(b) \mid b \in B(\infty) \}$ is a
lower global basis of $U_{\Q}^{-}(\g)$.

\vskip 2mm

\item[(b)] $\B(\lambda) : = \{ G(b) \mid b \in B(\lambda) \}$ is a
lower global basis of $V_{\Q}(\lambda)$.

\end{enumerate}

}

\end{proposition}

\vskip 3mm

\section{Primitive canonical bases} \label{sec:primitive}

\vskip 2mm

For clarity and simplicity, we fix the notations for some of
basic concepts in the theory of perverse sheaves.

\vskip 2mm

\begin{itemize}

\item[(a)]  $X$: algebraic variety over $\C$

\vskip 2mm

\item[(b)] $\mathbbm{1} = \mathbbm{1}_{X}$: constant sheaf on $X$

\vskip 2mm

\item[(c)] ${\mathcal Sh}(X)$: abelian category of sheaves on $X$ of $\C$-vector spaces

\vskip 2mm

\item[(d)] ${\mathcal D}(X)$: derived category of complexes of sheaves on $X$

\vskip 2mm

\item[(e)] ${\mathcal D}^{b}(X)$: full subcategory of ${\mathcal D}(X)$ consisting of bounded complexes on $X$

\vskip 2mm

\item[(f)] ${\mathcal D}^{b}_{c}(X)$: full subcategory of ${\mathcal D}^{b}(X)$ consisting of constructible
complexes on $X$

\vskip 2mm

\item[(g)] ${\mathcal Perv}(X)$: abelian category of perverse sheaves on $X$

\vskip 2mm

\item[(h)] For a complex $K$, let $D(K)$ denotes the Verdier dual of $K$.

\end{itemize}

\vskip 3mm

\subsection{Quiver with loops} \label{sub:quivers} \hfill

\vskip 3mm

Let $Q=(I,\Omega)$ be a quiver, where $I$ is the set of vertices
and $\Omega=\{h\mid s(h)\rightarrow t(h)\}$ is the set of arrows,
where $s(h)$ and $t(h)$ are starting vertex and target vertex of $h$, respectively.
Let $\Omega(i)$ denote the set of loops at $i$ and let $\omega_i=|\Omega(i)|$,
the number of loops at $i$.

\vskip 3mm

Let $h_{ij}$ denote the number of arrows $h:i\rightarrow j$.
We define
\begin{equation*}
a_{ij}=\begin{cases}
2(1-\omega_i), &\text{if}\ i=j,\\
-h_{ij}-h_{ji}, &\text{if}\ i\neq j.
\end{cases}
\end{equation*}
Then $A=A_{Q}:={(a_{ij})}_{i,j\in I}$ is a symmetric Borcherds-Cartan matrix.
We will denote by
$(A, P, P^{\vee}, \Pi, \Pi^{\vee})$ the Borcherds-Cartan datum associated with $A$.
Using the same notations as in Section \ref{sec:qBBalg},
we write $R:= \bigoplus_{i \in I} \Z \, \alpha_i$,
$R_{+}: = \sum_{i \in I} \Z_{\ge 0}\, \alpha_{i}$ and $R_{-} = - R_{+}$.
\vskip 3mm

Let  $\alpha=\sum_{i\in I}d_i\alpha_i\in R_{+}$ and
let $V_\alpha=\oplus_{i\in I}V_i$ be an $I$-graded vector space
with $\mathrm{dim} V_i=d_i$.
Then the graded dimension of $V_\alpha$ is given by
$\underline{\mathrm{dim}}V_\alpha=\sum_{i\in I}(\mathrm{dim} V_i) \, \alpha_i$.

\vskip 3mm

For every $I$-graded vector space $X$, we define
\[
E_X=\bigoplus_{h\in\Omega}\Hom(X_{s(h)},X_{t(h)}),\]
and set $E(\alpha)=E_{V_\alpha}$,
$G_\alpha=\prod_{i\in I}GL(V_{i})$.
Then $G_\alpha$ acts on $E(\alpha)$ by conjugation;
i.e.,
$${(g.x)}_h=g_{t(h)} x_h g_{s(h)}^{-1} \ \ \text{for} \  h\in\Omega.$$

\vskip 3mm

Let $\mathbf i=(i_1,\cdots, i_r)\in I^r$ and $\mathbf a=(a_1,\cdots, a_r)\in \Z_{\geq 0}^r$.
We say that $(\mathbf i, \mathbf a)$ is a {\it composition} of $\alpha$,
denoted by  $(\mathbf i,\mathbf a) \vdash\alpha$, if $a_1\alpha_{i_1}+\cdots+a_r\alpha_{i_r}=\alpha$.

\vskip 3mm

\begin{definition}
A flag $W=(\{0\}=W_0\subset\ldots\subset W_r=V_\alpha)$ is called a {\it flag of type $(\mathbf i,\mathbf a)$} if
$\underline{\mathrm{dim}}(W_{k}/W_{k-1})=a_k\alpha_{i_k}$ for all $1 \le k \le r$.
\end{definition}

\vskip 3mm

Let $\mathcal F_{\mathbf i,\mathbf a}$ be the variety consisting of all flags of type $(\mathbf i,\mathbf a)$. Then we have
\begin{equation}\label{eq:dimFia}
\mathrm{dim}(\mathcal F_{\mathbf i,\mathbf a})=\sum_{i_k=i_l,\ k<l}a_ka_l.
\end{equation}

\vskip 3mm

\begin{definition}
For $x={(x_h)}_{h\in\Omega}\in E(\alpha)$, we say that a flag $W$ is $x$-stable if $x_h(W_k\cap V_{s(h)})\subset W_k\cap V_{t(h)}$ for all $h\in\Omega$ and $k=0,1,\cdots,r$.
\end{definition}

\vskip 3mm

Let
\begin{equation*}
\widetilde{\mathcal F}_{\mathbf i,\mathbf a}=\left\{(x,W)\mid
x \in E(\alpha), \,  W\in\mathcal F_{\mathbf i,\mathbf a},\ W \text{ is $x$-stable}\right\}
\subseteq E(\alpha)\times\mathcal F_{\mathbf i,\mathbf a}.
\end{equation*}

\vskip 3mm

By \eqref{eq:dimFia}, we have
\begin{equation}\label{eq:dimFiatilde}
\mathrm{dim}(\widetilde{\mathcal F}_{\mathbf i,\mathbf a})=\sum_{h\in\Omega}\sum_{i_k=s(h) \atop i_l=t(h),\ k<l}a_ka_l+\sum_{i_k=i_l,\ k<l}a_ka_l.
\end{equation}

\vskip 3mm

Consider the natural projection
$$\pi_{\mathbf i,\mathbf a}:\widetilde{\mathcal F}_{\mathbf i,\mathbf a}\rightarrow E(\alpha),\quad (x,W)\mapsto x.$$

\vskip 2mm

Let $\mathbbm{1}=\mathbbm{1}_{\widetilde{\mathcal F}_{\mathbf i,\mathbf a}}$ be the constant sheaf on
$\widetilde{\mathcal F}_{\mathbf i,\mathbf a}$.
We define
$${\widetilde{L}}_{\mathbf i,\mathbf a}={(\pi_{\mathbf i,\mathbf a})}_{!}(\mathbbm{1}) \ \ \text{and}
\ \ L_{\mathbf i,\mathbf a}={\widetilde{L}}_{\mathbf i,\mathbf a}
[\mathrm{dim}\widetilde{\mathcal F}_{\mathbf i,\mathbf a}].$$

\vskip 2mm
\noindent
By \cite{BBD}, $L_{\mathbf i,\mathbf a}$ is semisimple and
stable under the Verdier duality; i.e.,  $D(L_{\mathbf i,\mathbf a})=L_{\mathbf i,\mathbf a}$.

\vskip 3mm

Suppose $(\mathbf i,\mathbf a)\vdash\alpha$. Let $\mathcal P_{\mathbf i,\mathbf a}$ be the set of simple perverse sheaves possibly with some shifts appearing in the decomposition of $L_{\mathbf i,\mathbf a}$.

\vskip 3mm

We define $\mathcal P_\alpha$ to be the full subcategory of
${\mathcal Perv}(E(\alpha))$ consisting of $P=\sum L$, where
\begin{enumerate}
\item[(i)] $L$ is a simple perverse sheaf,

\vskip 2mm

\item[(ii)] $L[d]$ appears as a direct summand of $L_{\mathbf i,\mathbf a}$
for some $(\mathbf i,\mathbf a) \vdash\alpha$ and $d\in \Z$.
\end{enumerate}


\vskip 3mm

Now we define ${\mathcal Q}_\alpha$ to be the full subcategory of
${\mathcal D}(E(\alpha))$ consisting of complexes $K$
such that $K \cong \oplus_{L,d}L[d]$, where $L\in\mathcal P_\alpha$ and $d\in\Z$.

\vskip 5mm

\begin{example}\label{ex:single point}
{\rm
Let $i\in I^{\mathrm{im}}$, $I=\{i\}$, $l>0$ and $\alpha=l\alpha_i$.
Then $(\mathbf i,\mathbf a)\vdash\alpha$ implies $\mathbf i=(\underbrace{i,\cdots,i}_{r})$, $\mathbf a=(a_1,\cdots,a_r)$ and $a_1+\cdots+a_r=l$.
Thus $\mathbf{a}$ is a composition (or a partition) of $l$. 	
	
\vskip 3mm

Let $V = V_{l \alpha_{i}}$ with $\underline{\dim} V = l \alpha_{i}$.
Then $V \cong \C^{l}$, $G_{\alpha} \cong GL(\C^{l})$ and
$$E(\alpha) \cong \Hom(V, V)^{\oplus \omega} \cong M_{l \times l}(\C)^{\oplus \omega}
\cong \C^{\oplus \omega l^2},$$
where $\omega = \omega_{i}$, the number of loops at $i$.

\vskip 3mm

In this special case, for simplicity, we will write $i$ for $\mathbf{i}$.
By \eqref{eq:dimFia} and \eqref{eq:dimFiatilde}, we have
\begin{equation} \label{eq:exdimFiatilde}
\begin{aligned}
& \mathrm{dim}(\mathcal F_{i,\mathbf a})=\sum_{k<l}a_k a_l, \\
& \mathrm{dim}(\widetilde{\mathcal F}_{i,\mathbf a})
= d_{i, \mathbf{a}}:  =\omega(\sum_{k<l}a_k a_l)
+\sum_{k<l}a_k a_l=(\omega + 1)\sum_{k<l} a_k a_l.
\end{aligned}
\end{equation}

\vskip 2mm

\noindent
Then we have
\begin{equation*}
L_{i,\mathbf a}={(\pi_{i,\mathbf a})}_!(\mathbbm{1} _{\widetilde{\mathcal F}_{i,\mathbf a}})[d_{i,\mathbf a}].
\end{equation*}

\vskip 3mm

From now on, we will write
$\mathbbm{1}_{i, \mathbf{a}} : = L_{i, \mathbf{a}}$ for $\mathbf{a} \vdash l$.
In particular, when $\mathbf{a} = (l)$, the trivial composition,
we will write $\mathbbm{1}_{i,l}$ for $\mathbbm{1}_{i, (l)}$.

}
\end{example}

\vskip 5mm

\subsection{Canonical bases} \label{sub:canonical bases} \hfill

\vskip 3mm

Recall that  $\A = \Z[q, q^{-1}]$.
We define $U^-_{\A}(\g)$ to be the $\A$-subalgebra of $U_{q}(\g)$ generated by
$f_{i}^{(n)} \, (i \in I^{\text{re}},n \ge 0$) and $f_{il}  \,(i \in I^{\text{im}}, l>0)$.

\vskip 3mm

Let ${\mathcal K}(\alpha)$ be the Grothendieck group of  ${\mathcal Q}_{\alpha}$.
Then $\A$ acts on $\mathcal{K}_{\alpha}$ via
\begin{equation*}
q^{\pm 1} [P] = ]P[\pm 1]],
\end{equation*}
where $[P]$ is the isomorphism class of a perverse sheaf $P$. Let $\mathcal B_\alpha$ be the set of isomorphic class of simple perverse sheaves in $\mathcal P_\alpha$.
Then $\mathcal{B}_{\alpha}$ is an $\A$-bass of $\mathcal{K}(\alpha)$.
In particular, for $i \in I_{\text{im}}$ and $l>0$, we have
${\mathcal B}_{l \alpha_i} = \{[\mathbbm{1}_{i, \mathbf{a}} ] \mid \mathbf{a} \vdash l \} $
and it is an $\A$-basis of ${\mathcal K}_{l \alpha_i}$.

\vskip 3mm

Set
\begin{equation*}
\mathcal{K} = \bigoplus_{\alpha \in R_{+}} {\mathcal K}(\alpha)
\ \ \text{and} \ \
\mathcal{B} = \bigsqcup_{\alpha \in R_{+}} \mathcal {B}_{\alpha}.
\end{equation*}
Then $\mathcal{B}$ is an $\A$-basis of $\mathcal{K}$.

\vskip 3mm

Let $\gamma=\alpha+\beta$, $V=V_\gamma$ and $W\subset V$ such that $\underline{\mathrm{dim}}(W)=\alpha$. Then we have $\underline{\mathrm{dim}}(V/W)=\beta$.
Consider the natural isomorphisms
$$
p:W\stackrel{\sim}{\longrightarrow} V_\alpha,\quad q:V/W\stackrel{\sim}{\longrightarrow}V_\beta,
$$
which yields a diagram
$$
E(\alpha)\times E(\beta)\stackrel{\kappa}\leftarrow E_\gamma(W)\stackrel{\iota}\hookrightarrow E(\gamma),
$$
where
\begin{enumerate}
\item[(a)] $E_\gamma(W)=\{x\in E(\gamma)\mid x(W)\subset W\}$,
\item[(b)] $\iota$ is the canonical embedding,
\item[(c)] $\kappa(x)=(p_{*}(x|_W),q_{*}(x|_{V/W}))$.
\end{enumerate}

\vskip 3mm

We define
$$
E(\alpha,\beta)=\{(x,W)\mid x\in E(\gamma),\ W\subset V,\ \underline{\mathrm{dim}}(W)=\alpha,\ x(W)\subset W\},
$$
and
$$
{E(\alpha,\beta)}^+=\{(x,W,\sigma,\tau)\mid (x,W)\in E(\alpha,\beta),\ \sigma:W\stackrel{\sim}\rightarrow V_\alpha,\ \tau:V/W\stackrel{\sim}\rightarrow V_\beta\}.
$$

\vskip 2mm

\noindent
Thus we obtain
$$
E(\alpha)\times E(\beta)\stackrel{p_1}\longleftarrow {E(\alpha,\beta)}^+\stackrel{p_2}\longrightarrow E(\alpha,\beta)\stackrel{p_3}\longrightarrow E(\gamma),
$$
where
\begin{align*}
&p_1(x,W,\sigma,\tau)=(p_{*}(x|_{W}),q_{*}(x|_{V/W})),\\	
&p_2(x,W,\sigma,\tau)=(x,W),\quad p_3(x,W)=x.
	\end{align*}

\vskip 3mm

Define the functors
\begin{align*}
&\widetilde{\mathrm{Res}}_{\alpha,\beta}:=\kappa_!\iota^*:\mathcal Q(\gamma)\rightarrow \mathcal Q(\alpha)\boxtimes\mathcal Q(\beta),\\
&\widetilde{\mathrm{Ind}}_{\alpha,\beta}:={p_3}_!{p_2}_\flat p_1^*:\mathcal Q(\alpha)\boxtimes\mathcal Q(\beta)\rightarrow\mathcal Q(\gamma).
\end{align*}

\vskip 2mm

\begin{remark}
It is highly non-trivial to prove
$$
\mathrm{Im}(\widetilde{\mathrm{Res}}_{\alpha,\beta})\subset \mathcal Q(\alpha)\boxtimes\mathcal Q(\beta),\quad \mathrm{Im}(\widetilde{\mathrm{Ind}}_{\alpha,\beta})\subset\mathcal Q(\gamma).
$$
In \cite[Section 9.2]{Lus10}, Lusztig gave a proof.
\end{remark}

\vskip 3mm

Assume that $\alpha=\sum d_i\alpha_i$ and $\beta=\sum d_i'\alpha_i$. Set $\langle\alpha,\beta\rangle:=\sum d_id_i'$ and denote by $l_1$ (resp. $l_2$) the dimension of fibers of $p_1$ (resp. $p_2$). Define the functors
$$
\mathrm{Res}_{\alpha,\beta}:=\widetilde{\mathrm{Res}}_{\alpha,\beta}[l_1-l_2-2\langle\alpha,\beta\rangle],\quad \mathrm{Ind}_{\alpha,\beta}:=\widetilde{\mathrm{Ind}}_{\alpha,\beta}[l_1-l_2].
$$
These functors commute with Verdier duality.

\vskip 3mm

Hence we obtain
\begin{equation*}
\begin{aligned}
&\mathrm{Ind}_{\alpha,\beta}:\mathcal K(\alpha)\otimes\mathcal K(\beta)\rightarrow K(\gamma),\\
&\mathrm{Res}_{\alpha,\beta}:K(\gamma)\rightarrow\mathcal K(\alpha)\otimes\mathcal K(\beta).
\end{aligned}
\end{equation*}

\vskip 2mm

Since $\mathcal{K} = \oplus_{\alpha \in R_{+}} \mathcal{K}_{\alpha}$,
we obtain the $\A$-algebra homomorphisms
\begin{equation*}
\begin{aligned}
&\mu:\mathcal K\otimes\mathcal K\rightarrow \mathcal K,\\
&\delta:\mathcal K\rightarrow\mathcal K\otimes\mathcal K,
\end{aligned}
\end{equation*}
induced by $\mathrm{Ind}_{\alpha,\beta}$
and $\mathrm{Res}_{\alpha,\beta}$.
In this way, $\mathcal K$ becomes an $\A$-bialgebra.

\vskip 3mm

The following theorem is one of the main results in \cite{Bozec2014b}.

\vskip 3mm

\begin{theorem}\label{thm: the algebra K}\cite[Proposition 5, Theorem 1]{Bozec2014b}
{\rm
\begin{enumerate}
\item[{\rm (a)}] The algebra $\mathcal K$ is generated by $[\mathbbm{1}_{il}]$ $((i,l)\in I^{\infty})$.

\vskip 2mm

\item[{\rm (b)}] There exists an isomorphism of $\A$-bialgebras
\begin{equation}\label{eq:Psi}
\Psi: U^-_{\A}(\mathfrak g)\stackrel{\sim}{\longrightarrow}\mathcal K \ \ \text{given by}
\ \ f_{il}\mapsto [\mathbbm{1}_{il}].
\end{equation}
\end{enumerate}

}
\end{theorem}

\vskip 3mm

\begin{definition} \label{def:canB}
The $\A$-basis $\B:=\Psi^{-1}(\mathcal B)$ of $U^-_{\A}(\g)$ is
called the {\it canonical basis} of  $U_{q}^-(\g)$.
\end{definition}

\vskip 3mm

Let $V(\lambda) = U_{q}(\g) \, v_{\lambda}$ be the irreducible highest weight module with highest weight
$\lambda \in P^{+}$.
We define {$V(\lambda)_{\A}:= U^{-}_{\A}(\g) \, v_{\lambda}$.
Then ${\B}^{\lambda}: = \B \, v_{\lambda}$ is an $\A$-basis of $V(\lambda)_{\A}$
\cite{Lus10}.

\vskip3mm

\begin{definition} \label{def:canBlambda}
{\rm The $\A$-basis $\B^{\lambda}$ of $V(\lambda)_{\A}$ is called the
{\it canonical basis} of $V(\lambda)$.
}
\end{definition}

\vskip 3mm

Unfortunately, the canonical bases $\B$ and $\B^{\lambda}$ do {\it not}
coincide with the lower global bases $\B(\infty)$ and $\B(\lambda)$.
To fix this situation, we introduce the notion of {\it primitive canonical bases}.

\vskip 3mm

Recall that there is a $\Q(q)$-algebra automorphism
\begin{equation*}
\phi: U_{q}^-(\g) \longrightarrow U_{q}^-(\g)
\ \ \text{given by} \ \ f_{il} \mapsto {\mathtt b}_{il} \ \text{for} \ (i,l)\in I^{\infty}
\end{equation*}
defined in Proposition \ref{prop:primitive}.
By the definition of $U^-_{\A}(\g)$ and $U^-_{\Q}(\g)$,  it is straightforward to see that
$\phi$ restricts down to the ${\A}_{\Q}$-algebra isomorphism

\begin{equation} \label{eq:primitive iso}
\phi: \Q \otimes U^-_{\A}(\g) \longrightarrow U^-_{\Q}(\g), \ \ f_{il} \mapsto {\mathtt b}_{il}
\ \text{for} \ (i,l) \in I^{\infty}.
\end{equation}

\vskip 3mm

\begin{definition} \label{def:primitiveB}

{\rm The $\A_{\Q}$-basis $\B_{\Q}:=\phi(\B)$ of $U^-_{\Q}(\g)$ is called the
{\it primitive canonical basis} of $U_{q}(\g)$.
}
\end{definition}

\vskip 3mm

For the irreducible highest weight module $V(\lambda)$ with $\lambda \in P^{+}$,
recall that
$V(\lambda)_{\Q}:=U^-_{\Q}(\g) \, v_{\lambda}$.
Then $\B^{\lambda}_{\Q} := \phi(\B)v_{\lambda}$ is an $\A_{\Q}$-basis of $V(\lambda)_{\Q}$.

\vskip 3mm

\begin{definition} \label{def:primBlambda}
{\rm  The $\A_{\Q}$-basis $\B^{\lambda}_{\Q}$ of $V(\lambda)_{\Q}$ is called the
{\it primitive canonical basis} of $V(\lambda)$.

}

\end{definition}

\vskip 3mm

In later sections, we will prove that the primitive canonical bases $\B_{\Q}$ and $\B^{\lambda}_{\Q}$
coincide with the lower global bases $\B(\infty)$ and $\B(\lambda)$, respectively.
 Actually, $\phi$ restricts down to the $\A$-algebra isomorphism between
 $U_{\A}^-(\g)$ and $U_{\Z}^-(\g)$. But to deal with the lower global bases,
 we need to consider $\Q$-extensions, because the lower global bases are
 $\A_{\Q}$-bases for  $U_{\Q}^-(\g)$ and $V(\lambda)_{\Q}$.

\vskip 5mm


\subsection{Geometric bilinear forms} \label{sub:geometric form} \hfill

\vskip 3mm

In this subsection, we recall some of basic parts of Lusztig's theory on perverse sheaves.

\vskip 3mm

Let $X$ be an algebraic variety over $\C$ and let $G$ be a connected algebraic group. Let $A,B$ be two $G$-equivariant semisimple complexes on $X$ with $G$-action.

\vskip 2mm We choose
\vskip 2mm
\begin{enumerate}
	\item[i)] an integer $m>0$,
\vskip 2mm
	\item[ii)] a smooth irreducible algebraic variety $\Gamma$
\end{enumerate}
such that
\begin{enumerate}
	\item[a)] $G$ acts freely on $\Gamma$,
	\vskip 2mm
	\item[b)] $H^k(\Gamma,\mathbb C)=0$ for $k=1,\cdots,m.$
\end{enumerate}

\vskip 3mm

Let $G$ act diagonally on $\Gamma\times X$ and set $_\Gamma X:=G\setminus(\Gamma\times X)$. Consider the diagram
$$
X\stackrel{a}{\longleftarrow}\Gamma\times X\stackrel{b}{\longrightarrow} {_{\Gamma}X}.
$$
Then $_{\Gamma}A$, $_{\Gamma}B$ are well-defined semisimple complexes on $_{\Gamma}X$ and  $a^*A=b^*{{_\Gamma} A}$, $a^*B=b^*{{_\Gamma} B}$.

\begin{proposition} \cite{GL93, Lus10}  

\vskip 2mm

{\rm
If $m$ is sufficiently large, then we have
\begin{equation}\label{prop:dimH}
\mathrm{dim}\, H_c^{j+2\mathrm{dim}\,\Gamma-2\mathrm{dim}\,G}({{_\Gamma} X},{{_\Gamma} A}\otimes{{_\Gamma} B})=\mathrm{dim}\, H_c^j({{_\Gamma} X},{{_\Gamma} A}[\mathrm{dim}\, G\setminus\Gamma]\otimes{{_\Gamma} B}[\mathrm{dim}\, G\setminus\Gamma]).
\end{equation}
}
\end{proposition}

\vskip 3mm

Let $d_j(X,G;A,B)$ denote the equation  \eqref{prop:dimH}.
Then we obtain a series of properties of $d_{j}(X,G; A, B)$.

\vskip 3mm

\begin{lemma}   \cite{GL93, Lus10}  

\begin{enumerate}
\vskip 2mm
	\item[{\rm(a)}] $d_j(X,G;A,B)=d_j(X,G;B,A)$,
	\vskip 2mm
	\item[{\rm(b)}] $d_j(X,G;A[m],B[n])=d_{j+m+n}(X,G;A,B)$,
	\vskip 2mm
	\item[{\rm(c)}] $d_j(X,G;A\oplus A',B)=d_j(X,G;A,B)+d_j(X,G;A',B)$.
\end{enumerate}
\end{lemma}

\vskip 3mm

\begin{lemma} \cite{GL93, Lus10}     
	
{\rm
\begin{enumerate}
	\item[{\rm(a)}] If $A, B$ are perverse sheaves, then $d_j(X,G;A,B)=0$ for $j>0$.
\vskip 2mm
	\item[{\rm(b)}] If $A, B$ are simple perverse sheaves, then
	\begin{equation*}
	d_0(X,G;A,B)=\begin{cases}
	1,& \text{if } A\cong D(B),\\
	0,& \text{otherwise.}
	\end{cases}
	\end{equation*}
\end{enumerate}		
}
\end{lemma}

\vskip 3mm

Let $\alpha=\sum d_i\alpha_i \in R_+$ and $V=\oplus_{i\in I}V_i$
with $\underline{\mathrm{dim}}\, V=\alpha$.
Let $X=E(\alpha)$, $G=G_\alpha$ and $P,P'$ be simple perverse sheaves in $\mathcal P_{-\alpha}$.
We denote by $B=[P]$, $B'=[P']$.
Then we have $\overline{B}=[D(P)]=[P]=B$ and $\overline{B'}=B'$.

\vskip 3mm

For $A,B\in\mathcal Q_{-\alpha}$, we define
$$
(A,B)_G:=\sum_{j\in \Z}d_j(E(\alpha),G_\alpha;A,B)q^{-j}\in \Z[[q]].
$$

\vskip 2mm

\begin{proposition} \label{prop:orthogonal1}  \cite{GL93, Lus10}   
\vskip 2mm
{\rm
\begin{enumerate}
	\item[{\rm (a)}] If $P$, $P'$ are simple perverse sheaves, then we have
	$$
	(B,B')_G\in \delta_{B.B'}+q \Z_{\geq 0}[[q]].
	$$
\vskip 2mm
	\item[{\rm (b)}] $(\ ,\ )_G$ is a Hopf pairing, i.e.
	$$
	(B,B'B'')_G=(\delta(B),B'\otimes B'')_G,
	$$
	where $\delta:\mathcal K\rightarrow\mathcal K\otimes \mathcal K$ is induced by $\mathrm{Res}$ functor.
\end{enumerate}
}
\end{proposition}

\vskip 3mm

Since the map $\Psi$ in \eqref{eq:Psi} is an isomorphism of bialgebras,
 we can identify $(\ ,\ )_L$ with $(\ ,\ )_G$ by setting $(x ,y)_L=(\Psi(x) ,\Psi(y))_G$.

 \vskip 3mm

For convenience, we will write
$B \in \B$ for $\Psi^{-1}(B)$.
Thus we have

\begin{equation} \label{eq:orthonormal}
(B,B')_L\in\delta_{B,B'}+q \Z_{\geq 0}[[q]]
\ \ \text{for} \  B,B' \in \B.
\end{equation}

\vskip 3mm

In the sequel, we use $(\ ,\ )$ to represent $(\ ,\ )_G$ or $(\ ,\ )_L$ if there is no danger of confusion.

\vskip 5mm

\subsection{Bozec's results on perverse sheaves}\label{sub:Bozec} \hfill

\vskip 3mm

For $x\in E(\alpha)$, we define
$V_i^{\Diamond}=\oplus_{j\neq i}V_j$
and $\mathfrak J_i(x)=\C\langle x\rangle V_i^{\Diamond}$.
There exists a stratification $E(\alpha)=\sqcup_{l\geq 0}E_{\alpha;i,l}$, where
$$
E_{\alpha;i,l}:=\{x\in E(\alpha)\mid \underline{\mathrm{codim}}_{V}\mathfrak J_i(x)=l\alpha_i\}.
$$

\vskip 2mm

Set $E_{\alpha;i,\geq l}=\sqcup_{k\geq l}E_{\alpha;i,k}$.
Let $\mathcal P_{-\alpha;i,\geq l}$ be the set of perverse sheaves in $\mathcal P_{-\alpha}$ supported on $E_{\alpha;i,\geq l}$ and let $\mathcal P_{-\alpha;i,l}=\mathcal P_{-\alpha;i,\geq l}\setminus \mathcal P_{-\alpha;i,\geq l+1}$.

\vskip 3mm

\begin{proposition}\cite[Proposition 4]{Bozec2014b}\label{prop:Bozec geometry}

\vskip 2mm

{\rm

Let $(i, l) \in I^{\infty}$. 	
	
\begin{enumerate}
\item[{\rm (a)}] For any simple perverse sheaf $P\in\mathcal P_{-\alpha;i,l}$, there exist a simple perverse sheaf $P_0\in\mathcal P_{-\alpha+l\alpha_i;i,0}$ and a simple perverse sheaf $P_{i,\mathbf c}\in\mathcal P_{-l\alpha_i}$ $(\mathbf c\vdash l)$ such that
		$$
		[P_{i,\mathbf c}][P_0]-[P]\in\bigoplus_{P'\in\mathcal P_{-\alpha;i,\geq l+1}} \A[P'].
		$$	
		
\item[{\rm (b)}] Conversely, for any simple perverse sheaf $P_0\in\mathcal P_{-\alpha+l\alpha_i;i,0}$ and a simple perverse sheaf $P_{i,\mathbf c}\in\mathcal P_{-l\alpha_i}$ $(\mathbf c\vdash l)$, there exists a simple perverse sheaf $P\in\mathcal P_{-\alpha;i,l}$ such that
		$$
			[P_{i,\mathbf c}][P_0]-[P]\in\bigoplus_{P'\in\mathcal P_{-\alpha;i,\geq l+1}} \A[P'].
		$$
\end{enumerate}
	
	}
\end{proposition}

\vskip 3mm

Define
\begin{align*}
&\B_{-\alpha;i,\geq l}:=\{ \Psi^{-1}([P]) \mid P\in\mathcal P_{-\alpha;i,\geq l}\},\\
&\B_{-\alpha;i,l}:=\B_{-\alpha;i,\geq l}\setminus \B_{-\alpha;i,\geq l+1}
=\{ \Psi^{-1}([P]) \mid P\in \mathcal P_{-\alpha;i,l}\}.
\end{align*}

\vskip 3mm

It is straightforward to see that Proposition \ref{prop:Bozec geometry} can be rephrased as

\vskip 3mm

\begin{corollary}\label{cor:BB-B}

{\rm

Let $(i, l) \in I^{\infty}$.

\vskip 2mm

\begin{enumerate}

	\item[{\rm (a)}] For any $B\in \B_{-\alpha;i,l}$, there exist $B_0\in \B_{-\alpha+l\alpha_i;i,0}$
	and
	$B_{i,\mathbf c}\in \B_{-l\alpha_i}$ $(\mathbf c\vdash l)$ such that
	$$
	B_{i,\mathbf c} \, B_0 - B\in \bigoplus_{B'\in \B_{-\alpha;i,\geq l+1}} \A \, B'.
	$$

	\item[{\rm (b)}] Conversely, for any $B_0 \in \B_{-\alpha+l\alpha_i;i,0}$ and $B_{i,\mathbf c}\in \B_{-l\alpha_i}$
	$(\mathbf c\vdash l)$, there exists $B\in \B_{-\alpha;i,l}$ such that
	$$
	B_{i,\mathbf c} \, B_0-B \in \bigoplus_{B'\in \B_{-\alpha;i,\geq l+1}} \A \, B'.
	$$
\end{enumerate}

}
\end{corollary}

\vskip 3mm

Recall that the primitive canonical basis is defined by $\B_{\Q} = \phi(\B)$.
Set
\begin{equation*}
	{(\B_{\Q}})_{-\alpha;i,\geq l}=\phi(\B_{-\alpha;i,\geq l}),
	\quad  {(\B_{\Q}})_{-\alpha;i,l}=\phi(\B_{-\alpha;i,l}).
 \end{equation*}

\noindent
Actually, the above second equation can be rewritten by
$${(\B_{\Q}})_{-\alpha;i,l}
= (\B_{\Q})_{-\alpha; i, \geq l} \setminus (\B_{\Q})_{-\alpha; i, \geq l+1}.$$

\vskip 3mm

Since the map $\phi$ in \eqref{eq:primitive iso} is an $\A_{\Q}$-algebra isomorphism, we obtain

\vskip 3mm

\begin{corollary}\label{cor:betabeta-beta}
	
{\rm	Let $(i, l) \in I^{\infty}$.}
	
	\vskip 2mm
	
	{\rm
	
	\begin{enumerate}
	
		\item[{\rm (a)}] For any $\beta \in (\B_{\Q}) _{-\alpha;i,l}$, there exist
		$\beta_0 \in (\B_{\Q}) _{-\alpha+l\alpha_i;i,0}$ and
		$\beta_{i,\mathbf c} \in (\B_{\Q}) _{-l\alpha_i}$ $(\mathbf c\vdash l)$ such that
		$$
		\beta_{i,\mathbf c} \, \beta_0-\beta\in\bigoplus_{\beta' \in
		(\B_{\Q})_{-\alpha;i,\geq l+1}} \A  \, \beta'.
		$$
		
		\vskip 2mm
		
		\item[{\rm (b)}] Conversely, for any $\beta_0 \in (\B_{\Q})_{-\alpha+l\alpha_i;i,0}$ and
		$\beta_{i,\mathbf c} \in (\B_{\Q})_{-l\alpha_i}$ $(\mathbf c\vdash l)$,
		there exists $\beta \in (\B_{\Q}) _{-\alpha;i,l}$ such that
		$$
		\beta_{i,\mathbf c} \, \beta_0-\beta\in\bigoplus_{\beta'\in (\B_{\Q})_{-\alpha;i,\geq l+1}}
		\A \,  \beta'.
		$$
		
	\end{enumerate}
	}
\end{corollary}

\vskip 3mm

\subsection{Key lemmas on global bases}\label{sub:key lemma} \hfill

\vskip 3mm

Now we will prove some of key  lemmas on lower global bases which will play important roles
in later discussions.

\vskip 3mm

\begin{proposition}\cite[Proposition 5.3.1]{Kashi93}\label{prop:fGb-Gb}

\vskip 2mm

{\rm Let $i\in I^{\mathrm{re}}$, $l\geq 0$.

\vskip 2mm

\begin{enumerate}
	\item[{\rm (a)}] For any $b\in {B(\infty)}_{-\alpha;i,l}$, there exists $b_0\in {B(\infty)}_{-\alpha+l\alpha_i;i,0}$ such that
	$$
	f_i^{(l)}G(b_0)-G(b)\in\bigoplus_{b'\in f_i^{(l+1)}B(\infty)} \A \, G(b').
	$$
	\item[{\rm (b)}]  For any $b_0\in {B(\infty)}_{-\alpha+l\alpha_i;i,0}$, there exists $b\in {B(\infty)}_{-\alpha;i,l}$ such that
	$$
	f_i^{(l)}G(b_0)-G(b)\in\bigoplus_{b'\in f_i^{(l+1)}B(\infty)} \A \,  G(b').
	$$
\end{enumerate}

}
\end{proposition}

\vskip 3mm

Let $i\in I^{\mathrm{im}}$ and $l>0$.
Define
\begin{align*}
&{B(\infty)}_{-\alpha;i,\geq l}:=\bigcup_{\mathbf c\vdash l}\widetilde{f}_{i,\mathbf c}({B(\infty)}_{-\alpha}),\\
&{B(\infty)}_{-\alpha;i,l}:={B(\infty)}_{-\alpha;i,\geq l}\setminus {B(\infty)}_{-\alpha;i,\geq l+1}.
\end{align*}

\vskip 3mm

\begin{lemma}\label{lem: CbGb-Gb}

{\rm

For any $b\in{B(\infty)}_{-\alpha;i,l}$, there exist $b_0\in{B(\infty)}_{-\alpha+l\alpha_i;i,0}$, $\mathbf c\vdash l$ and $C \in \Z_{>0}$ such that
	$$
	C\, \mathtt b_{i,\mathbf c}G(b_0)-G(b)\in\bigoplus_{b'\in{B(\infty)}_{-\alpha;i,\geq l+1}} \A_{\Q} \, G(b').
	$$
Here, $C=1$ for $i\notin I^{\mathrm{iso}}$.
}
\end{lemma}

\begin{proof}
Let $b\in{B(\infty)}_{-\alpha;i,l}$. There exist $b_0\in {B(\infty)}_{-\alpha+l\alpha_i;i,0}$ and $\mathbf c\vdash l$ such that $\widetilde{f}_{i,\mathbf c}b_0=b$; i.e., $\widetilde{f}_{i,\mathbf c}G(b_0)=G(b)\!\!\mod qL(\infty)$.

\vskip 3mm

If $i\notin I^{\mathrm{iso}}$, we have $\widetilde{f}_{i,\mathbf c}=\mathtt b_{i,\mathbf c}$.
Hence
\begin{align*}
\widetilde{f}_{i,\mathbf c}G(b_0)=\mathtt b_{i,\mathbf c}G(b_0)=a_0G(b)+\sum_{j=1}^{r}a_jG(b_j)\!\!\!\!\mod \bigoplus_{b'\in{B(\infty)}_{-\alpha;i,\geq l+1}} \A_{\Q} \, G(b'),
\end{align*}
where $a_0,a_1,\cdots,a_r\in \A_\Q$, $b_1,b_2\cdots,b_r\in{B(\infty)}_{-\alpha;i,l}$.

\vskip 2mm

Since $\overline{\mathtt b_{i,\mathbf c}G(b_0)}=\mathtt b_{i,\mathbf c}G(b_0)$, we must have
\begin{equation}\label{bara}
\overline{a_0}=a_0,\ \overline{a_1}=a_1,\ \cdots,\ \overline{a_r}=a_r.
\end{equation}

\vskip 3mm

On the other hand, we have
$$\widetilde{f}_{i,\mathbf c}G(b_0)=\mathtt b_{i,\mathbf c}G(b_0)=G(b)\!\!\!\!\mod qL(\infty).$$

\vskip 2mm

By taking $q\to 0,$ we obtain
$$b=a_0b+\sum_{j=1}^ra_jb_j\in L(\infty)/qL(\infty).$$

 \vskip 2mm

 Thus $a_0=1, a_1=\cdots a_r=0\!\!\mod q\A_0$. Hence, by \eqref{bara}, we have $a_1=\cdots=a_r=0$, which proves our claim.

\vskip 3mm

If $i\in I^{\mathrm{iso}}$, then we have $\widetilde{f}_{i,\mathbf c}\neq\mathtt b_{i,\mathbf c}$. But since $b_0\in {B(\infty)}_{-\alpha+l\alpha_i;i,0}$, we have $\widetilde{e}_{ik}b_0=0$ for any $k>0$. We obtain $e'_{ik}G(b_0)=0\!\!\mod qL(\infty)$ for any $k>0$. Thus $\widetilde{f}_{i,\mathbf c}G(b_0)=C\,\mathtt b_{i,\mathbf c}G(b)\!\!\mod qL(\infty)$ for some $C \in \Z_{>0}$.
Hence, we may write
$$
\widetilde{f}_{i,\mathbf c}G(b_0)=C\,\mathtt b_{i,\mathbf c}G(b_0)=a'_0G(b)+\sum_{j=1}^ra'_jG(b'_j)\!\!\!\!
\ \ \mod \bigoplus_{b'\in{B(\infty)}_{-\alpha;i,\geq l+1}} \A_{\Q} \, G(b'),
$$
where $a'_0,a'_1,\cdots,a'_r\in\A_{\Q}$, $b'_1,\cdots,b'_r\in{B(\infty)}_{-\alpha;i,l}$.

\vskip 2mm

Since $\overline{C\, \mathtt b_{i,\mathbf c}G(b_0)}=C \, \mathtt b_{i,\mathbf c}G(b_0)$, we have
\begin{equation}\label{eq:bara'}
\overline{a'_0}=a'_0,\ \overline{a'_1}=a'_1,\cdots,\ \overline{a'_r}=a'_r.
\end{equation}

\vskip 3mm

On the other hand, by taking $q\to 0$, we obtain
$$
\widetilde{f}_{i,\mathbf c}G(b_0)=C\,\mathtt b_{i,\mathbf c}G(b_0)=G(b)\!\!\!\!\mod qL(\infty).
$$

\vskip 2mm

\noindent
Hence, we have $b=a'_0b+a'_1b'_1+\cdots+a'_rb'_r\in L(\infty)/qL(\infty)$.
It follows that $a'_0=1, a'_1=\cdots=a'_r=0\!\!\mod q \A_0$.
By \eqref{eq:bara'}, we get $a'_0=1,a'_1=\cdots=a'_r=0$,
which proves our claim.
\end{proof}

\vskip 3mm

\begin{lemma}\label{lem:CbG-G}

{\rm
For any $b_0\in{B(\infty)}_{-\alpha+l\alpha_i;i,0}$ and $\mathbf c \vdash l$, there exist
$b\in{B(\infty)}_{-\alpha;i,l}$ and a positive integer $C>0$ such that
$$
C\, \mathtt b_{i,\mathbf c}G(b_0)-G(b)\in\bigoplus_{b'\in{B(\infty)}_{-\alpha;i,\geq l+1}} \A_{\Q} \, G(b').
$$
Here, $C=1$ for $i\notin I^{\mathrm{iso}}$.

}
\end{lemma}

\begin{proof}

Clearly, $\widetilde{f}_{i,\mathbf c}b_0=b$ for some $b\in {B(\infty)}_{-\alpha;i,l}$.
Hence  $\widetilde{f}_{i,\mathbf c}G(b_0)=G(b)\!\mod qL(\infty)$.

\vskip 2mm

If $i\notin I^{\mathrm{iso}}$, then we have $\widetilde{f}_{i,\mathbf c}=\mathtt b_{i,\mathbf c}$. In this case, the conclusion naturally holds.

\vskip 3mm
If $i\in I^{\mathrm{iso}}$, then we have
$e'_{ik}G(b_0)=0\!\!\mod qL(\infty)$ for any $k>0$, which yields
$$\widetilde{f}_{i,\mathbf c}G(b_0)=C \, \mathtt b_{i,\mathbf c}G(b_0)\!\!\mod qL(\infty)
\ \ \text{for some} \ C \in \Z_{>0}.$$

\vskip 2mm

Thus our claim follows naturally.
\end{proof}

\vskip 3mm

\section{Primitive canonical bases and global bases} \label{sec:Primlobal}

\vskip 2mm

In this section, we will show that the primitive canonical bases coincide with lower global bases.

\vskip 3mm

\subsection{Lusztig's and Kashiwara's bilinear forms} \label{sub:bilinear} \hfill

\vskip 3mm

We first compare Lusztig's bilinear form and Kashiwara's bilinear form defined in
Proposition \ref{prop:Lusztig}, \eqref{eq:bilinearV} and \eqref{eq:bilinearU}.

\vskip 3mm

\begin{lemma} \label{lem:deltab}

{\rm
Let $b_{k} = {\mathtt b}_{i_kl_k}$ $(1 \le k \le r)$.
Then we have

\begin{equation*}
\begin{aligned}
  \delta (b_{1}  \cdots b_{r})  & = 1 \otimes (b_{1} \cdots b_{r})
+ b_{1} \otimes (b_{2} \cdots b_{r}) \\
& + \sum_{k=2}^r q^{-(|b_{k}|, \, \sum_{p=1}^{k-1} |b_{p}|)}
\, b_{k} \otimes (b_{1} \cdots \widehat{b_k} \cdots b_{r})
+ \sum x_i \otimes y_i ,
\end{aligned}
\end{equation*}
where $\widehat{b_k}$ indicates that $b_{k}$ is removed from
$b_{1} \cdots b_{r}$ and $x_i$ is a monomial in $b_{k}$'s of degree $\ge 2$.

}
\end{lemma}

\vskip 3mm

\begin{proof} \
We will use induction on $r$.
If $r=1$, there is nothing to prove.

\vskip 2mm
Assume that $r \ge 2$.
Then we have

\begin{equation*}
\begin{aligned}
 \delta(b_{1} \cdots b_{r})& = \delta(b_{1} \cdots b_{r-1}) \, \delta(b_{r})
= \delta(b_{1} \cdots b_{r-1}) \, (1 \otimes b_{r} + b_{r} \otimes 1)\\
& = 1 \otimes (b_{1} \cdots b_{r-1} b_{r}) + b_{1} \otimes (b_{2} \cdots b_{r-1} b_{r})  \\
& + q^{-( |b_{r}|, \sum_{p=1}^{r-1} |b_{p}|)} \, b_{r} \otimes (b_{1} \cdots b_{r-1})
 + q^{-(|b_{r}|, \sum_{p=2}^{r-1} |b_{p}|)} \,  b_{1} b_{r} \otimes (b_{2} \cdots b_{r-1}) \\
 & + \sum_{k=2}^{r-1} q^{-(|b_{k}| , \, \sum_{p=1}^{k-1} |b_{p}|)}
 \, b_{k} \otimes( b_{1} \cdots \widehat{b_{k}} \cdots b_{r-1} b_{r}) \\
 & + \sum_{k=2}^{r-1} q^{-(|b_{k}| , \, \sum_{p=1}^{k-1} |b_{p}|)}
 \, q^{-(|b_{r}, \, \sum_{p=1, p \neq k}^{r-1} |b_{p}|)}
 (b_{k} b_{r} \otimes b_{1} \cdots \widehat{b_k} \cdots b_{r-1}) \\
 & + \sum x_{i} \otimes y_i \, b_{r} + q^{-(|y_{i}|, |b_{r}|)} x_i b_{r} \otimes y_i \\
 & = 1 \otimes (b_{1} \cdots b_{r}) + b_{1} \otimes (b_{2} \cdots b_{r}) \\
 & + \sum_{k=2}^{r} q^{-(|b_{k}| , \, \sum_{p=1}^{k-1} |b_{p}|)}
 \, b_{k} \otimes (b_{1} \cdots \widehat{b_{k}} \cdots  b_{r})
 + \sum x_{i}' \otimes y_{i}',
\end{aligned}
\end{equation*}
where $\deg x_{i}' \ge 2$ and our assertion follows.
\end{proof}

\vskip 3mm

\begin{corollary} \label{cor:Lusztig} \hfill

\vskip 2mm

{\rm

Let $a_k = {\mathtt b}_{i_k l_k}$ and $b_{k} = {\mathtt b}_{j_k m_k}$ $(1 \le k \le r)$.
Then we have
\begin{equation*}
\begin{aligned}
(a_{1}   \cdots a_{r}, b_{1} \cdots b_{r})_{L}& = (a_{1}, b_{1})_{L} \,
(a_{2} \cdots a_{r}, \, b_{2} \cdots b_{r})_{L} \\
& + \sum_{k=2}^{r} q^{-(|b_{k}|, \sum_{p=1}^{k-1} |b_{p}|)} \,
(a_{1}, b_{k})_{L} \, (a_{2} \cdots a_{r}, b_{1} \cdots \widehat{b_{k}} \cdots b_{r})_{L}.
\end{aligned}
\end{equation*}
}
\end{corollary}

\begin{proof} \ Our assertion follows immediately from Lemma \ref{lem:deltab}.
\begin{equation*}
\begin{aligned}
& (a_{1} \cdots a_{r}, \, b_{1} \cdots b_{r})_{L}
= (a_{1} \otimes a_{2} \cdots a_{r}), \, \delta(b_{1} \cdots b_{r}))_{L} \\
& = (a_{1} \otimes a_{2} \cdots a_{r}, \, 1 \otimes b_{1} \cdots b_{r})_{L}
+ (a_{1} \otimes a_{2} \cdots a_{r}, \, b_{1} \otimes b_{2} \cdots b_{r})_{L} \\
& + (a_{1} \otimes a_{2} \cdots a_{r}, \
\sum_{k=2}^{r} q^{-(|b_{k}|, \, \sum_{p=1}^{k-1} |b_{p}|)} \, b_{k} \otimes b_{1}
\cdots  \widehat{b_k} \cdots b_{r})_{L} \\
& + (a_{1} \otimes a_{2} \cdots a_{r}, \, \sum x_{i}\otimes y_{i})_{L} \\
& =  (a_{1}, b_{1})_{L} \, (a_{2} \cdots a_{r}, \, b_{2} \cdots b_{r})_{L} \\
& + \sum_{k=2}^{r}  q^{-(|b_{k}|, \, \sum_{p=1}^{k-1} |b_{p}|)}
(a_{1}, b_{k})_{L} \, (a_{2} \cdots a_{r},\, b_{1} \cdots \widehat{b_{k}} \cdots b_{r})_{L}.
\end{aligned}
\end{equation*}
\end{proof}

Next, we will show that Lusztig's bilinear form and Kashiwara's bilinear form are equivalent
up to $q \A_{0}$.

\vskip 2mm

\begin{lemma} \label{lem:eb} \hfill
\vskip 2mm

{\rm
For $k = 1, 2, \ldots, r$, we have
\begin{equation} \label{eq:eb}
\begin{aligned}
e_{i_1 l_1}'   (\mathtt b_{j_1 m_1} \cdots\mathtt b_{j_r m_r})
&= \delta_{i_{1} j_{1}} \delta_{l_1 m_1} \, (\mathtt b_{j_2 m_2} \cdots \mathtt b_{j_r m_r}) \\
& + q_{i_1}^{-\sum_{p=1}^r l_1 m_p a_{i_1j_p}} \, (\mathtt b_{j_1 m_1} \cdots\mathtt b_{j_r m_r}) \, e_{i_1 l_1}' \\
& + \sum_{k=2}^r \delta_{i_1, j_k} \delta_{l_1, m_k} \,
q_{i_1}^{-\sum_{p=1}^{k-1} \, l_1 m_p  a_{i_1 j_p}} \,
(\mathtt b_{j_1 m_1} \cdots \widehat{\mathtt b_{j_k m_k}} \cdots\mathtt b_{j_r m_r}).
\end{aligned}
\end{equation}
}
\end{lemma}

\begin{proof} \ We will use induction on $r$.
When $r=1$, by \eqref{eq:commute}, our assertion  follows immediately.

\vskip 3mm

Assume that $r \ge 2$.
Using the induction hypothesis, we have
\begin{equation*}
\begin{aligned}
& e_{i_1 l_1}' (\mathtt b_{j_1m_1} \cdots\mathtt b_{j_{r-1}m_{r-1}} \mathtt b_{j_rm_r})
= (e_{i_1l_1}'\mathtt  b_{j_1m_1} \cdots\mathtt b_{j_{r-1}m_{r-1}}) \mathtt b_{j_rm_r} \\
& = \delta_{i_1, j_1} \, \delta_{l_1, m_1} \,
(\mathtt b_{j_2m_2} \cdots\mathtt b_{j_{r-1}m_{r-1}}\mathtt b_{j_rm_r})  \\
& + q_{i_1}^{-\sum_{p=1}^{r-1} l_1 m_p a_{i_1j_p}}\, (\mathtt b_{j_1 m_1} \cdots\mathtt b_{j_{r-1} m_{r-1}})
(e_{i_1 l_1}'\mathtt b_{j_r m_r}) \\
& + \sum_{k=2}^{r-1} \delta_{i_1, j_k} \, \delta_{l_{1}, m_{k}}\,
q_{i_1}^{-\sum_{p=1}^{k-1} l_1 m_p a_{i_1 j_p}} \,
(\mathtt b_{j_1 m_1} \cdots \widehat{\mathtt b_{j_k m_k}}\cdots\mathtt b_{j_{r-1}, m_{r-1}})\mathtt b_{j_r m_r } \\
& = \delta_{i_1, j_1} \, \delta_{l_1, m_1} \,
(\mathtt b_{j_2 m_2} \cdots\mathtt b_{j_{r-1} m_{r-1}}\mathtt b_{j_r m_r})  \\
& + q_{i_1}^{-\sum_{p=1}^{r-1} l_1 m_p a_{i_1 j_p}}\,
\delta_{i_1, j_r} \, \delta_{l_1, m_r} \,(\mathtt b_{j_1 m_1} \cdots \mathtt b_{j_{r-1} m_{r-1}}) \\
& +  q_{i_1}^{-(\sum_{p=1}^{r-1} l_1 m_p a_{i_1 j_p} + l_1 m_r a_{i_1j_r})} \,
(\mathtt b_{j_1 m_1} \cdots\mathtt b_{j_r m_r} e_{i_1 l_1}') \\
& + \sum_{k=2}^{r-1} \delta_{i_1,j_k} \, \delta_{l_1, m_k} \,
q_{i_1}^{-\sum_{p=1}^{k-1} l_1 m_p a_{i_1 j_p}} \, (\mathtt b_{j_1 m_1} \cdots \widehat{\mathtt b_{j_k m_k}} \cdots\mathtt b_{j_{r-1} m_{r-1}} \mathtt b_{j_r m_r}) \\
& = \delta_{i_1, j_1} \, \delta_{l_1, m_1} \, (\mathtt b_{j_2 m_2} \cdots\mathtt b_{j_r m_r})
+ q_{i_{1}}^{-\sum_{p=1}^{r} l_1 m_p a_{i_1 j_p}} \,
(\mathtt b_{j_1 m_1} \cdots\mathtt b_{j_r m_r} e_{i_1 l_1}')  \\
& + \sum_{k=2}^{r} \delta_{i_1, j_k}\, \delta_{l_1, m_k} \,
q_{i_1}^{-\sum_{p=1}^{k-1} l_1 m_p a_{i_1 j_p}}
\, (\mathtt b_{j_1 m_1} \cdots \widehat{\mathtt b_{j_k m_k}} \cdots \mathtt b_{j_r m_r}),
\end{aligned}
\end{equation*}
as desired.
\end{proof}

\vskip 3mm

\begin{corollary} \label{cor:Kashiwara} \hfill

\vskip 2mm

{\rm Let ${\mathtt b}_{i_k l_k}, \, {\mathtt b}_{j_k m_k} \in U_{q}^{-}(\g)$  $(k =1,2, \ldots, r)$.
Then Kashiwara's bilinear form is given by

\begin{equation} \label{eq:Kashiwara}
\begin{aligned}
& ({\mathtt b}_{i_1l_1} \cdots {\mathtt b}_{i_r l_r},
\, {\mathtt b}_{j_1 m_1} \cdots {\mathtt b}_{j_r m_r} )_{K} \\
& =  \delta_{i_1, j_1} \, \delta_{l_1, m_1} ({\mathtt b}_{i_2l_2} \cdots {\mathtt b}_{i_r l_r},
\, {\mathtt b}_{j_2 m_2} \cdots {\mathtt b}_{j_r m_r})_{K} \\
& + \sum_{k=2}^{r} \delta_{i_1, j_k}\, \delta_{l_1, m_k} \,
q_{i_1}^{- \sum_{p=1}^{k-1} l_1 m_p a_{i_1 j_p}}
\, ({\mathtt b}_{i_2 l_2} \cdots {\mathtt b}_{i_r l_r} , \,
{\mathtt b}_{j_1 m_1} \cdots \widehat{{\mathtt b}_{j_k m_k}} \cdots {\mathtt b}_{j_r m_r})_{K}.
\end{aligned}
\end{equation}

}
\end{corollary}

\vskip 3mm

As we can see in the following proposition,  Lusztig's bilinear form and Kashiwara's bilinear form
are closely related.

\vskip 3mm

\begin{proposition} \label{prop:compareLK} \hfill
\vskip 2mm

{\rm
Let ${\mathtt b}_{i_k l_k}, \, {\mathtt b}_{j_k m_k} \in U_{q}^{-}(\g)$  $(k =1,2, \ldots, r)$.
Then we have
\begin{equation*}
\begin{aligned}
& ({\mathtt b}_{i_1 l_1}  \cdots {\mathtt b}_{i_r l_r}, \, {\mathtt b}_{j_1 m_1}
\cdots {\mathtt b}_{j_r m_r})_{L}     \\
& = \prod_{s=1}^r (1 -q_{i_s}^{2 l_s})^{-1} \,
({\mathtt b}_{i_1 l_1} \cdots b_{i_r l_r}, \, {\mathtt b}_{j_1 m_1} \cdots
{\mathtt b}_{j_r m_r})_{K}.
\end{aligned}
\end{equation*}
Therefore, we have
\begin{equation*}
(x, y)_{L} = (x, y)_{K} \ \ \text{mod} \ \, q\, \A_{0}  \ \ \text{for all} \ x, y \in U_{q}^{-}(\g).
\end{equation*}
}
\end{proposition}

\begin{proof} \ We will use induction on $r$.
If $r=1$, our assertion follows from the definition of these bilinear forms.

\vskip 3mm

Assume that $r \ge 2$.
By Corollary \ref{cor:Lusztig} and induction hypothesis, we have
 \begin{equation} \label{eq:Lusztig1}
 \begin{aligned}
  & ({\mathtt b}_{i_1 l_1} \cdots {\mathtt b}_{i_r l_r}, \,
  {\mathtt b}_{j_1 m_1} \cdots {\mathtt b}_{j_r  m_r})_{L} \\
  & = ({\mathtt b}_{i_1 l_1}, b_{j_1 m_1})_{L} \,
  ({\mathtt b}_{i_2 l_2} \cdots {\mathtt b}_{i_r l_r},
  {\mathtt b}_{j_2 m_2} \cdots {\mathtt b}_{j_r  m_r})_{L} \\
  & + \sum_{k=2}^{r} q^{-\sum_{p=1}^{k-1}
  (m_k \alpha_{j_k}, m_p \alpha_{j_p})} \,
   ({\mathtt b}_{i_1  l_1}, {\mathtt b}_{j_k  m_k})_{L}\\
  & \hskip 5mm
  \times
  ({\mathtt b}_{i_2 l_2} \cdots {\mathtt b}_{i_r l_r},  \,
  {\mathtt b}_{j_1  m_1} \dots \widehat{{\mathtt b}_{j_k m_k}} \cdots {\mathtt b}_{j_r m_r})_{L}
    \\
 & = \delta_{i_1, j_1} \delta_{l_1, m_1} \, (1 - q_{i_1}^{2 l_1})^{-1}
 \prod_{s=2}^r  (1 - q_{{\color{red}i_{s}}}^{2 l_s})^{-1} \\
 & \hskip 5mm
 \times ({\mathtt b}_{i_2 l_2} \cdots {\mathtt b}_{i_r l_r},
 {\mathtt b}_{j_2 m_2} \cdots {\mathtt b}_{j_r  m_r})_{K} \\
 & + \sum_{k=2}^{r} q^{-\sum_{p=1}^{k-1}  (m_k \alpha_{j_k}, m_p \alpha_{j_p}) }
 \, \delta_{i_1, j_k} \, \delta_{l_1, m_k} \, (1 - q_{i_1}^{2 l_1})^{-1} \\
 & \hskip 5mm
 \times
 ({\mathtt b}_{i_2 l_2} \cdots {\mathtt b}_{i_r l_r},  \,
 {\mathtt b}_{j_1  m_1} \dots \widehat{{\mathtt b}_{j_k m_k}} \cdots {\mathtt b}_{j_r m_r})_{L}.
\end{aligned}
 \end{equation}

 \vskip 2mm
 If $\delta_{i_1, j_k} \, \delta_{l_1, m_k} =0$ for some $k\in\{2,\cdots,r\}$, then
 the corresponding summand of formula \eqref{eq:Lusztig1} will disappear. Therefore, we only need to consider the case of $\delta_{i_1 ,j_k} \, \delta_{l_1, m_k} =1$. Then we must have $ j_k= i_1$, $m_k = l_1$,
 which implies
 \begin{equation*}
 \begin{aligned}
 & \sum_{k=2}^{r}  q^{- \sum_{p=1}^{k-1}  (m_k \alpha_{j_k}, m_p \alpha_{j_p}) }
 =  \sum_{k=2}^{r} q^{-\sum_{p=1}^{k-1} m_k m_p s_{j_k} a_{j_k j_p} }    \\
 & = \sum_{k=2}^{r} q^{-\sum_{p=1}^{k-1} l_1 m_p s_{i_1} a_{i_1 j_p} }
 = \sum_{k=2}^{r}  q_{i_1}^{-\sum_{p=1}^{k-1}   l_1 m_p a_{i_1 j_p}  }.
 \end{aligned}
 \end{equation*}

 \vskip 3mm

 It follows from Corollary \ref{cor:Kashiwara} that
\begin{equation*}
 \begin{aligned}
 &({\mathtt b}_{i_1 l_1}  \cdots {\mathtt b}_{i_r l_r},  \,
  {\mathtt b}_{j_1 m_1} \cdots {\mathtt b}_{j_r  m_r})_{L} \\
&  = \delta_{i_1, j_1} \, \delta_{l_1, m_1}
 \prod_{s=1}^{r} (1 - q_{i_s}^{2 l_s})^{-1}
 \, ({\mathtt b}_{i_2 l_2}  \cdots  {\mathtt b}_{i_r l_r}, \,
  {\mathtt b}_{j_2 m_2} \cdots {\mathtt b}_{j_r m_r})_{K} \\
 & \hskip 5mm + \sum_{k=2}^{r} q_{i_1}^{-\sum_{p=1}^{k-1}  l_1 m_p a_{i_1 j_p}}
  \delta_{i_1, j_k} \, \delta_{l_1, m_k} \, \\
& \hskip 5mm \hskip 5mm  \times  \prod_{s=1}^{r} (1 - q_{i_s}^{2 l_s})^{-1} \,
 ({\mathtt b}_{i_2 l_2} \cdots {\mathtt b}_{i_r l_r}, \,
 {\mathtt b}_{j_1 m_1} \cdots \widehat{{\mathtt b}_{j_k m_k}} \cdots {\mathtt b}_{j_r m_r})_{K}  \\
 & = \prod_{s=1}^{r} (1 - q_{i_s}^{2 l_s})^{-1} \, \delta_{i_1, j_1} \delta_{l_1, m_1} \,
 ({\mathtt b}_{i_2 l_2} \cdots {\mathtt b}_{i_r l_r},\,
  {\mathtt b}_{j_2 m_2} \cdots {\mathtt b}_{j_r m_r})_{K} \\
 & \hskip 5mm + \prod_{s=1}^{r} (1 - q_{i_s}^{2 l_s})^{-1} \,
  \sum_{k=2}^{r} q_{i_1}^{-\sum_{p=1}^{k-1}  l_1 m_p a_{i_1 j_p}}
 \delta_{i_1, j_k} \delta_{l_1, m_k} \, \\
&  \hskip 5mm \hskip 5mm \times
 ({\mathtt b}_{i_2 l_2} \cdots {\mathtt b}_{i_r l_r}, \,
 {\mathtt b}_{j_1 m_1} \cdots \widehat{{\mathtt b}_{j_k m_k}} \cdots {\mathtt b}_{i_r m_r})_{K} \\
 & =  \prod_{s=1}^{r} (1 - q_{i_s}^{2 l_s})^{-1} \,
  ({\mathtt b}_{i_1 l_1} \cdots {\mathtt b}_{i_r l_r}, \,
  {\mathtt b}_{j_1 m_1} \cdots {\mathtt b}_{j_r  m_r})_{K}.
 \end{aligned}
 \end{equation*}
 \end{proof}

\begin{proposition} \label{prop:prim bilinear}

\vskip 2mm

{\rm For all $x, y \in U^-_{q}(\g)$, we have
\begin{equation*}
(\phi(x), \phi(y))_{L} = (x, y)_{L}.
\end{equation*}

}
\end{proposition}

\begin{proof} \
It suffices to prove our assertion for monomials only.
Let
\begin{equation*}
x = f_{i_1 l_1} \cdots f_{i_r l_r} \ \ \text{and} \ \
y = f_{j_1 m_1} \cdots f_{j_1 m_r}.
\end{equation*}

\vskip 2mm

By \eqref{eq:comult}, we have
\begin{equation*}
\begin{aligned}
\delta(y)  = & \sum_{a_1 + b_1 = m_1} \cdots  \sum_{a_r + b_r = m_r}
(\prod_{k=1}^{r} q_{(j_k)}^{-a_k b_k} \prod_{k=2}^{r}
q^{-(a_k \alpha_{j_k}, \sum_{p=1}^{k-1} b_p \alpha_{j_p})}) \\
& \times  (\prod_{s=1}^{r} f_{j_sa_s}) \otimes (\prod_{t=1}^{r} f_{j_t b_t}).
\end{aligned}
\end{equation*}

\vskip 2mm

\noindent
It follows that
\begin{equation*}
\begin{aligned}
 (x, y)_{L} & = (f_{i_1 l_1} \otimes f_{i_2 l_2} \cdots f_{i_r l_r},
\delta(y))_{L}\\
 & =  \sum_{a_1 + b_1 = m_1} \cdots  \sum_{a_r + b_r = m_r}
(\prod_{k=1}^{r} q_{(j_k)}^{-a_k b_k} \prod_{k=2}^{r}
q^{-(a_k \alpha_{j_k}, \sum_{p=1}^{k-1} b_p \alpha_{j_p})}) \\
& \hskip 5mm \times (f_{i_1 l_1}, \prod_{s=1}^{r} f_{j_s a_s})_{L} \,
(f_{i_2 l_2} \cdots f_{i_r l_r}, \,
\prod_{t=1}^{r} f_{j_t b_t})_{L}.
\end{aligned}
\end{equation*}

\vskip 2mm

Let
\begin{equation*}
A = (f_{i_1 l_1}, \prod_{s=1}^{r} f_{j_s a_s})_{L}, \ \
B = (f_{i_2 l_2} \cdots f_{i_r l_r}, \,
\prod_{t=1}^{r} f_{j_t b_t})_{L}
\end{equation*}

\vskip 2mm

If $AB \neq 0$, then we have $A\neq 0$ and $B\neq 0$. Thus, there exists a positive integer $k>0$ such that

\vskip 2mm

\begin{enumerate}

\item[(i)] $i_1 = j_k$, $l_1 = a_k$,

\vskip 2mm

\item[(ii)] $a_{p} = 0$ for all $p \neq k$.

\end{enumerate}

\vskip 3mm

Hence $a_k = m_k$, $b_{k} = 0$, $b_p = m_p$ for all $p \neq k$, which implies

$$ B = (f_{i_2 l_2} \cdots f_{i_r l_r},\,
f_{j_1 m_1} \cdots  \widehat{f_{j_k m_k}}\cdots f_{j_r m_r})_{L}.$$

\noindent
Note that $\prod_{k=1}^{r} q_{(j_k)}^{-a_k b_k} = 1$ because
$a_{p} = 0$ for all $p \neq k$ and $b_{k}=0$.

\vskip 3mm

Thus we have
\begin{equation*}
\begin{aligned}
 (x, y)_{L}  & = (f_{i_1l_1}, f_{j_1 m_1})_{L}
\, (f_{i_2 l_2} \cdots f_{i_r l_r},
f_{j_2 m_2}, \cdots f_{j_r m_r})_{L} \\
& + \sum_{k=2}^r q^{-(m_k \alpha_{j_k}, \, \sum_{p=1}^{k-1} m_p \alpha_{j_p})} \,
 (f_{i_1 l_1}, f_{j_k m_k})_{L}\, \\
& \hskip 5mm \times (f_{i_2 l_2} \cdots f_{i_r l_r}, \, f_{j_1 m_1} \cdots \widehat{f_{j_k m_k}} \cdots
f_{j_r m_r})_{L}.
\end{aligned}
\end{equation*}

\vskip 2mm

By induction hypothesis and Corollary \ref{cor:Lusztig}, we obtain
\begin{equation*}
\begin{aligned}
 (x, y)_{L}  & = ({\mathtt b}_{i_1l_1}, {\mathtt b}_{j_1 m_1})_{L}
\, ({\mathtt b}_{i_2 l_2} \cdots {\mathtt b}_{i_r l_r},
{\mathtt b}_{j_2 m_2}, \cdots {\mathtt b}_{j_r m_r})_{L} \\
& + \sum_{k=2}^r q^{-(m_k \alpha_{j_k}, \, \sum_{p=1}^{k-1} m_p \alpha_{j_p})} \,
 ({\mathtt b}_{i_1 l_1}, {\mathtt b}_{j_k m_k})_{L}\, \\
& \hskip 5mm \times ({\mathtt b}_{i_2 l_2} \cdots {\mathtt b}_{i_r l_r}, \,
{\mathtt b}_{j_1 m_1} \cdots \widehat{{\mathtt b}_{j_k m_k}} \cdots
{\mathtt b}_{j_r m_r})_{L} \\
& = (\phi(x), \phi(y))_{L}
\end{aligned}
\end{equation*}
as desired.
\end{proof}

\vskip 3mm

To summarize, combining Proposition \ref{prop:compareLK},
Proposition \ref{prop:prim bilinear} and
Lemma \ref{lem:orthogonal}, we obtain the following proposition.

\vskip 3mm

\begin{proposition} \label{prop:orthogonality} \hfill

\vskip 2mm

{\rm

Let $\B$, $\B_{\Q}$ and $\B(\infty)$ be the canonical basis,
primitive canonical basis and lower global basis of $U_{q}^{-}(\g)$,
respectively.
Then the following orthogonality statements hold.

\vskip 2mm

\begin{enumerate}

\item[(a)] For all $B, B' \in \B$, \ $(B, B')_{L} = \delta_{B, B'}  \ \text{mod} \ q\, \A_{0}.$

\vskip 2mm

\item[(b)] For all $\beta, \beta' \in \B_{\Q}$, \
$(\beta, \beta')_{L} = (\beta, \beta')_{K} = \delta_{\beta, \beta'}  \ \text{mod} \ q \, \A_{0}.$

\vskip 2mm

\item[(c)] For all $b, b' \in B(\infty)$, \
$(G(b), G(b'))_{K} = C \, \delta_{b, b'}  \ \text{mod} \ q \, \A_{0}$ for some $C \in \Z_{>0}$.

\end{enumerate}

}
\end{proposition}

\vskip 3mm

Similarly, we also have

\vskip 3mm

\begin{proposition} \label{prop:orthogonalityBlambda} \hfill

\vskip 2mm

{\rm

Let $\B^{\lambda}$, $\B_{\Q}^{\lambda}$ and $\B(\lambda)$ be the canonical basis,
primitive canonical basis and lower global basis of $V(\lambda)$,
respectively.
Then the following orthogonality statements hold.

\vskip 2mm

\begin{enumerate}

\item[(a)] For all $B, B' \in \B^{\lambda}$, \ $(B, B')_{L} = \delta_{B, B'}  \ \text{mod} \ q\, \A_{0}.$

\vskip 2mm

\item[(b)] For all $\beta, \beta' \in \B_{\Q}^{\lambda}$, \
$(\beta, \beta')_{L} = (\beta, \beta')_{K} = \delta_{\beta, \beta'}  \ \text{mod} \ q \, \A_{0}.$

\vskip 2mm

\item[(c)] For all $b, b' \in B(\lambda)$, \
$(G(b), G(b'))_{K} = C \, \delta_{b, b'}  \ \text{mod} \ q \, \A_{0}$ for some $C \in \Z_{>0}$.

\end{enumerate}

}
\end{proposition}

\vskip 5mm

\subsection{Grojnowski-Lusztig's argument} \label{sub:GLargument} \hfill

\vskip 3mm

Now we are ready to prove that the primitive canonical bases coincide with
lower global bases.

\vskip 3mm

Let $\B_{\Q}$ be the primitive canonical basis of $U_{q}^{-}(\g)$ and
let $\beta$ be an element of $\B_{\Q}$.
Since the lower global basis  $\B(\infty)$ is an $\A_{\Q}$-basis of $U^{-}_{\Q}(\g)$, we may write
\begin{equation*}
\beta = \sum_{b \in B(\infty) \atop j \in \Z} a_{b,j} \, q^j \, G(b)
\ \ \text{for} \ \ a_{b, j} \in \Q.
\end{equation*}

Since $( \ , \ )_{L} = (\ , \ )_{K} \ \, \text{mod} \ \, q \A_{0}$, we will just
use $(\ , \ )$ for both of them.

\vskip 3mm

Let $j_0$ be the smallest integer such that $a_{b,j} \neq 0$ for some $b \in B(\infty)$.
Since $(G(b), G(b')) = 0$ for $b \neq b'$, we have
\begin{equation*}
(\beta, \beta) \in \sum_{b \in B(\infty)} a_{b, j_0}^2 \, q^{2 j_{0}} \, (G(b), G(b))
+ q^{2 j_0 + 1}\, \Q[[q]],
\end{equation*}
which implies
\begin{equation*}
(\beta, \beta) =  \sum_{b \in B(\infty)} a_{b, j_0}^2 \, q^{2 j_0} (G(b), G(b))
\ \, \text{mod} \ q \A_{0}.
\end{equation*}

\vskip 2mm

By Proposition \ref{prop:orthogonality}, we have $(\beta, \beta) = 1 \ \, \text{mod} \ q\A_{0} $.
Hence we must have
\begin{equation*}
j_0 =  0, \ \ a_{b,j} = 0 \ \ \text{for} \ \,  j<0, \ b\in B(\infty).
\end{equation*}

\vskip 3mm

Moreover, there exists $b \in B(\infty)$ such that

\begin{equation*}
a_{b, 0} = \pm 1, \ \ (G(b), G(b)) = 1, \ \ a_{b',0} = 0 \ \, \text{for} \ b' \neq b.
\end{equation*}

\vskip 2mm

Hence $\beta - a_{b,0} G(b)$ is a linear combination of elements in $\B(\infty)$
with coefficients in $q \A_{0}$.
Since  $\beta - a_{b,0} G(b)$ is invariant under the bar involution,
these coefficients are all $0$, which implies
$\beta = a_{b,0} G(b) = \pm G(b)$. That is, we may write
\begin{equation*}
\beta = \epsilon_{\beta} \, G(b_{\beta}), \ \ \text{where} \ \, \epsilon_{\beta} = \pm 1.
\end{equation*}

\begin{theorem} \label{thm:prim and global} \hfill

\vskip 2mm

{\rm
The primitive canonical basis $\B_{\Q}$ coincides with the lower global basis $\B(\infty)$.
}

\end{theorem}

\begin{proof} \ We would like to show that $\epsilon_{\beta} = 1$
for all $\beta \in \B_{\Q}$.

\vskip 3mm

Let $\beta \in (\B_{\Q})_{-\alpha}$ for $\alpha \in R_{+}$.
If $\alpha = 0$, our assertion is trivial.
Hence we assume that $\alpha \neq 0$.
Then there exist $i \in I$ and $l>0$ such that $\beta \in (\B_{\Q})_{-\alpha; i, l}$.

\vskip 3mm

(a) If $i \in I^{\text{re}}$, our assertion was proved in \cite{GL93}.






\vskip 3mm

(b) If $i \in I^{\text{im}} \setminus I^{\text{iso}}$, by Corollary \ref{cor:betabeta-beta} (a),
there exist $\beta_{0} \in (\B_{\Q})_{-\alpha+l\alpha_i ; i, 0}$ and ${\mathbf c} \vdash l$
such that
\begin{equation} \label{eq:a}
{\mathtt b}_{i, {\mathbf c}} \, \beta_{0} - \beta \, \in
\bigoplus_{\beta' \in (\B_{\Q})_{-\alpha;i,\ge l+1}} \, \A_{\Q} \, \beta'
\subset \sum_{|\mathbf{c}'| \ge l+1} {\mathtt b}_{i, {\mathbf c}'} U^{-}_{\Q}(\g).
\end{equation}

\noindent
By induction hypothesis, we obtain $\epsilon_{\beta_{0}}=1$; i.e., $\beta_{0} = G(b_{0})$,
where $b_{0} = b_{\beta_{0}}$.
Note that $e_{ik}' \beta_{0} = e_{ik}' G(b_{0}) = 0$ for all $k>0$.
Since $\widetilde{f}_{il} = {\mathtt b}_{il}$, there exist $b \in B(\infty)_{-\alpha; i, l}$ and ${\mathbf c} \vdash l$ such that
\begin{equation} \label{eq:b}
{\mathtt b}_{i, {\mathbf c}} \, G(b_{0}) - G(b) \, \in
\bigoplus_{b' \in B(\infty)_{-\alpha;i,\ge l+1}} \, \A_{\Q} \, G(b')
\subset \sum_{|\mathbf{c}'| \ge l+1} {\mathtt b}_{i, {\mathbf c}'} U^{-}_{\Q}(\g).
\end{equation}

\vskip 2mm

Comparing \eqref{eq:a} and \eqref{eq:b}, we conclude $G(b) = \beta$,
which yields $G(b) = \beta = \epsilon_{\beta} \, G(b_{\beta}) \in \B(\infty)$.
Since both $G(b)$ and $\epsilon_{\beta} \, G(b_{\beta})$ belong to the
lower global basis $\B(\infty)$, we must have $\epsilon_{\beta} = 1$ and
$b = b_{\beta}$.

\vskip 3mm

(c) If $i \in I^{\text{iso}}$, by Corollary  \ref{cor:betabeta-beta} (a),
there exist $\beta_{0} \in (\B_{\Q})_{-\alpha+l\alpha_i ; i, 0}$ and ${\mathbf c} \vdash l$
such that
\begin{equation} \label{eq:c}
{\mathtt b}_{i, {\mathbf c}} \, \beta_{0} - \beta \, \in
\bigoplus_{\beta' \in (\B_{\Q})_{-\alpha;i,\ge l+1}} \, \A_{\Q} \, \beta'
\subset \sum_{|\mathbf{c}'| \ge l+1} {\mathtt b}_{i, {\mathbf c}'} U^{-}_{\Q}(\g).
\end{equation}

\noindent
By induction hypothesis, we obtain $\epsilon_{\beta_{0}}=1$; i.e., $\beta_{0} = G(b_{0})$,
where $b_{0} = b_{\beta_{0}}$.

\vskip 3mm

By Lemma \ref{lem:CbG-G}, there exist $b \in B(\infty)_{-\alpha; i, l}$ and
a positive integer $C>0$
such that
\begin{equation} \label{eq:d}
C \, {\mathtt b}_{i, {\mathbf c}} \, G(b_{0}) - G(b) \, \in
\bigoplus_{b' \in B(\infty)_{-\alpha;i,\ge l+1}} \, \A_{\Q} \, G(b')
\subset \sum_{|\mathbf{c}'| \ge l+1} {\mathtt b}_{i, {\mathbf c}'} U^{-}_{\Q}(\g).
\end{equation}

By \eqref{eq:c} and \eqref{eq:d}, we obtain $G(b) = C \,\beta = C \, \epsilon_{\beta} \, G(b_{\beta}) \in \B(\infty)$.
Since both $G(b)$ and $C \, \epsilon_{\beta} \, G(b_{\beta})$ are  elements of $\B(\infty)$,
we must have $C \, \epsilon_{\beta}  = 1$.
Since $C$ is a positive integer and $\epsilon_{\beta} = \pm 1$,
we must have $C = \epsilon_{\beta} = 1$.
\end{proof}

As an immediate consequence, we obtain

\vskip 3mm

\begin{corollary} \label{cor:prim and global Blambda}  \hfill

\vskip 2mm

{\rm
The primitive canonical basis $\B_{\Q}^{\lambda}$ coincides with the
lower global basis $\B(\lambda)$.
}

\end{corollary}

\end{document}